\documentclass[11pt, a4paper,twoside]{article}
\usepackage[showframe=false, left=2.5cm, right=2.5cm,  top=2.5cm, bottom=1in]{geometry}
\usepackage[english]{babel}
\usepackage{amsmath}
\usepackage{amssymb}

\usepackage{amsthm}
\usepackage{dsfont}
\usepackage{multicol}
\usepackage{bbm}
\usepackage{tikz}
\usetikzlibrary{arrows,positioning}
\usetikzlibrary{arrows.meta,chains,shapes.geometric}
\usepackage{mathtools}
\usepackage{esint}
\usepackage{nicefrac}
\usepackage{changepage}
\usepackage{enumitem}
\setlist[enumerate,1]{label=(\roman*),font=\itshape}
\usepackage[T1]{fontenc}
\usepackage[utf8]{inputenc}
\usepackage{csquotes}
\usepackage{aliascnt}
\usepackage[hidelinks]{hyperref}
\hypersetup{colorlinks,linkcolor={black},citecolor={teal},urlcolor={teal}}
\mathtoolsset{showonlyrefs}
\usepackage[style=alphabetic,giveninits=true, maxbibnames=5, maxalphanames=4, backend=biber, url=false, isbn=false]{biblatex}
\DefineBibliographyExtras{english}{}
\usepackage[T1]{fontenc}
\usepackage[final, nopatch=footnote]{microtype}
\usepackage{booktabs,caption}
\usepackage[flushleft]{threeparttable}
\usepackage{tasks}
\settasks{label-width=21.8pt}
\usepackage{graphicx}
\usepackage[font=small,labelfont=bf,labelsep=period]{caption}
\graphicspath{ {./images/} }
\usepackage{setspace}

\usepackage{pdfpages}
\usepackage[nocheck,headings]{fancyhdr}
\usepackage{titlesec}
\titleformat{\section}{\centering\large\scshape}{\thesection.}{0.5em}{}
\titlespacing{\section}{0pc}{4ex plus 0ex minus 0ex}{0.3pc}
\titleformat{\subsection}[runin]{\normalfont\bfseries}{\thesubsection.}{0.3em}{}[.]
\titlespacing{\subsection}{0pc}{1.5ex plus 0ex minus 0ex}{0.3pc}
\titleformat{\subsubsection}[runin]{\normalfont\bfseries}{\thesubsubsection.}{0.3em}{}[.]
\titlespacing{\subsubsection}{0pc}{1.5ex plus 0ex minus 0ex}{0.3pc}

\theoremstyle{plain}
\newtheorem{theorem}{Theorem}[section]

\newaliascnt{proposition}{theorem} 
\newtheorem{proposition}[proposition]{Proposition}

\newaliascnt{lemma}{theorem}
\newtheorem{lemma}[lemma]{Lemma}

\aliascntresetthe{lemma}

\newaliascnt{corollary}{theorem}

\aliascntresetthe{corollary}

\newaliascnt{hypothesis}{theorem}

\aliascntresetthe{hypothesis}

\theoremstyle{remark}

\newaliascnt{remark}{theorem}
\newtheorem{remark}[remark]{remark}
\aliascntresetthe{remark}

\newtheoremstyle{rmk}  
  {\topsep}               
  {\topsep}               
  {\normalfont}           
  {0pt}                 
  {\bfseries}            
  {.}                    
  { }             
  {}                     

\theoremstyle{rmk}
\newtheorem{rmk}[remark]{Remark}

\newtheoremstyle{definition}
  {\topsep}
  {\topsep}
  {\itshape}
  {0pt}
  {\bfseries}
  {.}
  { }
  {}

\theoremstyle{definition}
\newaliascnt{definition}{theorem}
\newtheorem{definition}[definition]{Definition}
\aliascntresetthe{definition}  

\newtheoremstyle{breakthm} 
  {\topsep}  
  {\topsep}   
  {\itshape} 
  {0pt}      
  {\bfseries} 
  {.}         
  {\newline} 
  {}
  {}

\theoremstyle{breakthm}
\newtheorem{breaktheorem}[theorem]{Theorem}

\newtheorem{breakproposition}[proposition]{Proposition}
\aliascntresetthe{proposition}

\newtheorem{breaklemma}[lemma]{Lemma}
\aliascntresetthe{lemma}

\aliascntresetthe{corollary}

\numberwithin{equation}{section}
\numberwithin{theorem}{section}
\numberwithin{lemma}{section}
\numberwithin{proposition}{section}
\numberwithin{corollary}{section}
\numberwithin{definition}{section}
\numberwithin{remark}{section}

\addto\extrasenglish{%
}
\def\AA{\mathbb A}\def\BB{\mathbb B}\def\CC{\mathbb C}

\def\NN{\mathbb N}\def\PP{\mathbb P}
\def\QQ{\mathbb Q}\def\RR{\mathbb R}\def\TT{\mathbb T}

\def\ZZ{\mathbb Z}
\def\cA{\mathcal A}\def\cB{\mathcal B}\def\cD{\mathcal D}
\def\cH{\mathcal H}
\def\cK{\mathcal K}
\def\cM{\mathcal M}

\usepackage{mathrsfs}
\def\SP{\mathscr{S}}
\def\rddsym{\RR^{d\times d}_{sym}}
\def\rdd{\RR^{d\times d}}
\def\ttheta{\tilde{\theta}}
\def\tb{\tilde{b}}

\def\ta{\tilde{a}}
\def\tgamma{\tilde{\gamma}}
\def\tu{\tilde{u}}
\def\Sdm{S^{d-1}}
\def\rrdmz{\RR^d\setminus\{0\}}
\def\tcK{\tilde{\cK}}
\def\tchi{\tilde{\chi}}
\def\hcK{\widehat{\cK}}
\def\hF{\widehat{F}}
\DeclareMathOperator*{\argmax}{arg\,max}
\DeclareMathOperator*{\argmin}{arg\,min}
\DeclareMathOperator{\supp}{supp}

\DeclareMathOperator{\dist}{dist}
\DeclareMathOperator{\cc}{{cc}}

\DeclareMathOperator{\spn}{span}
\DeclareMathOperator{\rank}{rank}
\DeclareMathOperator{\curl}{{curl}}
\DeclareMathOperator{\ccurl}{{curlcurl}}
\DeclareMathOperator{\Ran}{Ran}

\DeclareMathOperator{\Lin}{Lin}

\DeclareMathOperator{\sym}{sym}

\DeclareMathOperator{\Skw}{Skew}
\DeclareMathOperator{\Div}{{div}}

\DeclareMathOperator{\length}{length}
\DeclareMathOperator{\per}{Per}

\def\ppa{\PP_{\hspace*{-0.07cm}\cA}}
\def\raa{r_{\hspace*{-0.07cm}\cA}}
\def\Raa{R_{\cA}}
\def\ppaCC{\ppa^{\hspace{0.01cm}\CC}}
\def\ppc{\PP_{\curl}}
\def\ppd{\PP_{\Div}}
\def\ppcc{\PP_{\cc}}
\newcommand{\abs}[1]{\left\lvert#1\right\rvert}
\newcommand{\norm}[1]{\left\lVert#1\right\rVert}
\newcommand{\rddd}[1]{\RR^{#1\times #1}}
\newcommand{\rddsymd}[1]{\RR^{#1\times #1}_{sym}}

\newcommand{\del}{\partial}
\newcommand{\eps}{\epsilon}
\newcommand{\Step}[1]{\smallbreak \indent \textit{Step #1.}\\}

\newcommand{\Case}[1]{\smallbreak \noindent \textit{Case #1.}}
\newcommand{\hsp}{\hspace{1.4cm}}
\usepackage{abstract}
\usepackage{abstract}
\renewenvironment{abstract}{
  \begin{adjustwidth}{1cm}{1cm} 
  \footnotesize 
  \noindent\textsc{Abstract.}
}{
  \end{adjustwidth}
  \normalsize
}
\usepackage[titles]{tocloft}
\begin{document}
\pagenumbering{gobble}
\hypersetup{linkcolor=black}    
\pagestyle{fancy}
\fancyhead{}
\fancyfoot{}

\begin{adjustwidth}{1.0cm}{1.0cm} 
\begin{center}
\setstretch{1.5}
{\LARGE The energy scaling behavior of a class of incompatible~two-well~problems\footnote{
This work is based on the author's master's thesis, supervised by Angkana Rüland and Antonio Tribuzio.
He is deeply grateful to them for their guidance, valuable discussions and constant support.
The author gratefully acknowledges funding from the Studienstiftung des deutschen Volkes during his studies
and from the Deutsche Forschungsgemeinschaft through the Leibniz Prize, project-ID 545985277, awarded to Angkana Rüland, which supports his current doctoral research.}}
\\
\vspace{0.2cm}
\text{\large Noah Piemontese-Fischer}\footnote{Institute for Applied Mathematics, University of Bonn, Endenicher Allee 60, 53115 Bonn, Germany.}\smallbreak
\text{\large December 12, 2025}\smallbreak
\end{center}
\end{adjustwidth} 
\vspace{1pt}
\begin{abstract}
In this article, we study scaling laws for singularly perturbed two-well energies with prescribed Dirichlet boundary data in settings where the wells and/or the boundary data are incompatible.
Our main focus is the \textit{geometrically linear two-well problem}, for which we characterize the energy scaling in two dimensions for nearly all combinations of linear boundary data and stress-free strains. In particular, we prove that if the boundary data enforces oscillations and the weight $\eps$ of the surface energy is small,
the minimal energy upon subtracting the zeroth-order contribution scales either as $\eps^{{4}/{5}}$ or as $\eps^{{2}/{3}}$, depending on whether the wells differ by a rank-one or a rank-two matrix, respectively.
For the \textit{gradient} and \textit{divergence-free two-well problem}, we obtain analogous results, showing an $\eps^{{2}/{3}}$-scaling behavior in two dimensions whenever oscillations are energetically favored.
These results follow by deriving matching upper and lower scaling bounds.
The lower scaling bounds are established in a general \textit{$\cA$-free framework} for incompatible two-well problems, which allows us 
to compute the excess energy and characterize boundary data which enforce oscillations.
The upper scaling bounds are obtained by branching constructions which are adapted to the incompatible setting.

\end{abstract}
\vspace{-7pt}
\setcounter{tocdepth}{2}
\tableofcontents
\newpage

\hypersetup{linkcolor=blue}   
\pagenumbering{arabic}
\fancyhead[RE]{\thepage}
\fancyhead[LO]{\thepage}
\renewcommand{\sectionmark}[1]{\markboth{\thesection.\ #1}{}}
\fancyhead[RO]{\small\nouppercase{\leftmark}}
\renewcommand{\subsectionmark}[1]{\markright{#1}}
\fancyhead[LE]{\rightmark}

\section{Introduction}\label{sec:intro}
Due to their unique elastic properties, shape-memory alloys have received much attention in the mathematical and engineering literature \cite{muller1, otsuka1999shape, chai}.
A key feature of these materials is a temperature-induced solid-solid phase transformation.
At high temperatures, they form a highly symmetric crystal lattice structure, known as \textit{austenite}.
Upon cooling below a critical temperature, the austenite transforms into variants of a less symmetric phase, called \textit{martensite}.
This phase transformation is accompanied by the formation of microstructures in which the different phases are arranged throughout the material \cite{rindler,ruland4}.
This article is concerned with the analysis of such microstructures in shape-memory alloys at subcritical temperature with two variants of martensite.

\subsection{The geometrically linear two-well problem}
In the geometrically linear theory of elasticity, the deformation of the material is described by a \textit{displacement field} $v:\Omega\to\RR^d$ that (approximately) minimizes
an \textit{elastic energy} of the form
\begin{equation}\label{eq:IntroEnergy}
  E_{el}(v) = \int_\Omega W(\nabla^{sym}v)\,dx,
\end{equation}
where $\Omega\subset\RR^d$ is the \textit{reference configuration} and $W:\rddsym\to[0,\infty)$ is the \textit{energy density}.\medbreak

We assume that the energy density is minimized by two matrices $a_0,a_1\in\rddsym$, which are referred to as \textit{wells} or \textit{stress-free strains},
corresponding to the two variants of martensite.
The energy density depends on the \textit{linear strain} which is the symmetric part of the \textit{displacement gradient}:
\begin{equation}
  \nabla^{sym}v := \sym (\nabla v) = \frac{\nabla v + (\nabla v)^T}{2}.
\end{equation}
As was shown in \cite{kohn, DalmasoNegriPercivale-2002}, the energy \eqref{eq:IntroEnergy} can formally be derived by linearization of the nonlinear model developed by Ball and James \cite{bj,bj2}.
Based on this linearization, we assume that the energy density grows quadratically away from the stress-free strains $\cK:=\{a_0,a_1\}$, taking the specific form
\begin{equation}\label{eq:EnergyDensitySquaredDist}
  W(A):=\dist^2(A,\cK),\quad A\in\rddsym.
\end{equation}
Due to the lack of quasiconvexity of this double-well potential, the minimization of \eqref{eq:IntroEnergy} under suitable boundary conditions enforces fine-scale mixing of the phases, giving rise to rich microstructures.
However, as the elastic energy neglects interfacial energy, the variational model favors infinitesimally fine mixtures of the phases and fails to predict the length scale of the microstructure \cite{km92, km94}.\medbreak

This issue can be resolved by introducing a \textit{singular perturbation} penalizing large surface area of the phase interfaces \cite{km92,km94}.
Following \cite{CO09,CO12}, we measure the surface area using \textit{phase arrangements}.
These are maps $\chi\in L^2(\Omega;\cK)$ which allow us to define the \textit{elastic energy at fixed phase arrangement} $\chi$
for displacement fields $v\in W^{1,2}(\Omega;\RR^d)$ as
\begin{equation}\label{eq:ElasticEnergyIntro}
    E_{el}(v,\chi):=\int_\Omega \abs{\nabla^{sym}v-\chi}^2\,dx.
\end{equation}
Selecting at each point the optimal phase for a fixed displacement field $v$, we observe that $E_{el}(v)=\min_{\chi\in L^2(\Omega;\cK)}E_{el}(v,\chi)$.
Now, this description via phase arrangements yields the following \textit{surface energy}
\begin{equation}\label{eq:surfaceenergy}
    E_{surf}(\chi):= \norm{\nabla \chi}_{TV(\Omega)},\quad \chi\in BV(\Omega;\cK),
\end{equation} 
which we use to introduce the \textit{singularly perturbed elastic energy} for $\eps\geq 0$
\begin{equation}\label{eq:SingPertEnergyIntro}
    E_\eps(v,\chi):= E_{el}(v,\chi) +\eps  E_{surf}(\chi),\quad v\in W^{1,2}(\Omega;\RR^d),\, \chi\in BV(\Omega;\cK).
\end{equation}
The singular perturbation prevents arbitrarily fine mixtures of the phases and thus has a regularizing effect. 
Moreover, the singularly perturbed energy selects a length scale, allowing for increasingly fine structures as $\eps\to0$.
This behavior has been analyzed very successfully in terms of \textit{scaling laws},
which quantify the blow-up rate of the surface energy occurring in the minimization of the elastic energy and characterize the complexity of optimal microstructures.
In this context, we consider the following problem:\\
Given $F\in\rdd$ and $\cK=\{a_0,a_1\}\subset\rddsym$, identify the scaling of $E_\eps(F,\cK)$ as $\eps\to0$, where
\begin{equation}\label{eq:baseEnergyIntro}
     E_{\eps}(F,\cK):=\inf\{E_\eps(v,\chi) : \chi\in BV(\Omega;\cK), \,v\in W^{1,2}(\Omega;\RR^d)\text{ with } v(x)=Fx \text{ on }\del\Omega\}.
\end{equation}
In order to ensure that $E_\eps(F,\cK)\to0$ as $\eps\to0$, the following assumptions are typically made:
\begin{alignat}{2}
  &\hspace{-1.0cm}\textit{(i) The wells are compatible (as linear strains): }  &&a_1-a_0=b\odot\xi \text{ for some }b,\xi\in\RR^d. \label{eq:compatiblewellsIntro}\\
  &\hspace{-1.0cm}\textit{(ii) The boundary data is compatible: }\quad &&F=(1-\lambda) a_0+\lambda a_1 \text{ for some }\lambda\in[0,1]. \label{eq:compatiblebdrydataIntro}
\end{alignat}
If both \eqref{eq:compatiblewellsIntro} and \eqref{eq:compatiblebdrydataIntro} are satisfied, we say that the \textit{data $(F,\cK)$ are compatible} (as linear strains);
otherwise, they are called \textit{incompatible}.  \medbreak

The energy scaling is well-understood for compatible data.
In a simplified scalar model, the seminal works \cite{km92,km94}
used a branching construction to show that the minimal singularly perturbed energy scales as $\eps^{\nicefrac{2}{3}}$.
This result disproved the widely expected $\eps^{\nicefrac{1}{2}}$-scaling which had been predicted based on simple laminates.
An improved $\eps^{\nicefrac{4}{5}}$-energy scaling was later reported in \cite{cc15} for a singular perturbation model derived from nonlinear elasticity that,
depending on the structure of the wells, arises instead of the typical $\eps^{\nicefrac{2}{3}}$-scaling.
In \cite{ruland} and \cite{ruland2}, the scaling behavior of two-well energies was then systematically studied in a general $\cA$-free setting.
Taken together, the above works (see, in particular, \cite[Theorem~1.1, Theorem~1.2]{cc15} for the upper bound constructions, 
\cite[Theorem~1, Proposition~3.10]{ruland} for lower bounds and \cite[Theorem~1.2]{ruland2} for a generalized setting)
yield the following scaling laws.
\begin{theorem}[Scaling laws for the singularly perturbed geometrically linear two-well energy]\label{thm:chanconti}
Let $\Omega=(0,1)^2$. Let $\cK=\{a_0,a_1\}\subset\RR^{2\times2}_{sym}$ and $F=(1-\lambda)a_0+\lambda a_1$ for some $\lambda\in(0,1)$.
For $\eps\geq0$, let $E_\eps(F,\cK)$ be given as in \eqref{eq:baseEnergyIntro}.
Then, for small $\eps\geq0$, it holds that
\begin{equation}\label{eq:ScalingIntro}
        E_{\eps}(F,\cK)\sim
                \begin{cases}
                  \eps^{\nicefrac{4}{5}}  & \text{ if } a_1-a_0 \in \{e_1\odot e_1, e_2\odot e_2\},\\
        \eps^{\nicefrac{2}{3}}  & \text{ if } a_1-a_0 = e_1\odot e_2.
        \end{cases}        
\end{equation}
\end{theorem}
In contrast, if $F\in\cK$ (i.e. $\lambda\in\{0,1\}$ in \eqref{eq:compatiblebdrydataIntro}), we obtain the \textit{trivial energy scaling} 
\begin{equation}\label{eq:trivialScalingIntro}
    E_\eps(F,\cK)=0\quad\forall\eps\geq0
\end{equation}
since \eqref{eq:baseEnergyIntro} is minimized by the linear displacement field $v(x):=Fx$ and the constant phase arrangement $\chi(x):= F$.
The trivial energy scaling indicates that no surface energy is required in elastic energy minimization with a single phase occupying the entire domain.
\medbreak

Much less is known if the data are incompatible.
In this direction, the work of Kohn \cite{kohn} analyzed the geometrically linear two-well energy without surface energy using the theory of relaxations.
This showed that layered microstructures are optimal in the minimization of \eqref{eq:ElasticEnergyIntro} for both compatible and incompatible data,
as discussed in more detail in \autoref{sec:relaxation}. 
In this context, Kohn derives an explicit formula for $E_0(F,\cK)$ for all data $F\in\rdd$ and $\cK=\{a_0,a_1\}\subset\rddsym$ (see \autoref{thm:qwdom})
which in particular shows that
\begin{equation}\label{eq:E0VanishIFFDataComp}
    E_0(F,\cK)=0 \iff \text{the data }(F,\cK)\text{ are compatible or } F\in\cK.
\end{equation}
The goal of this article is to extend Kohn's analysis of the incompatible setting to the singular perturbation model.
Seeking to understand features of optimal microstructures, we study scaling laws for $E_\eps(F,\cK)$.
However, for all incompatible data $(F,\cK)$ with $F\notin\cK$, there holds
\begin{equation}
    E_\eps(F,\cK)\xrightarrow{\eps\to 0} E_0(F,\cK)>0.
\end{equation}
Therefore, in order to characterize the \textit{higher order behavior} of the geometrically linear two-well energy,
we subtract the zeroth order term and seek to determine the scaling of the \textit{zeroth-order-corrected energy} $E_\eps(F,\cK)-E_0(F,\cK)$ as $\eps\to0$ in the spirit of \autoref{thm:chanconti}.
The results on this problem are collected in \autoref{thm:incompsyma} below.\medbreak

As in \cite{kohn}, we also consider the corresponding theory for gradients, which provides valuable insight for the geometrically linear problem.
In this context, we systematically study the scaling of the \textit{singularly perturbed two-well energy for the gradient}
\begin{equation}\label{eq:GradSingPertEnergyIntro}
       E^{\text{grad}}_{\eps}(v, \chi):= \int_\Omega \abs{\nabla v-\chi}^2\,dx +\eps E_{surf}(\chi),\quad v\in W^{1,2}(\Omega;\RR^d),\, \chi\in BV(\Omega;\cK),\,\eps\geq0
\end{equation}
with the aim of removing any assumptions on the compatibility of the data.
Note that the energy \eqref{eq:GradSingPertEnergyIntro} arises from \eqref{eq:SingPertEnergyIntro} by disregarding the gauge invariance
with respect to the action of the group $\Skw(d)$ encoding infinitesimal frame indifference.
The results on the scaling behavior of \eqref{eq:GradSingPertEnergyIntro} are collected in \autoref{thm:incompgrad} below.

\subsection{Formulation of the two-well problems in the \texorpdfstring{$\cA$}{A}-free framework}
We work within the $\cA$-free framework developed by Murat and Tartar in the context of compensated compactness \cite{murat1, tartar1, tartar2}.
Following \cite{fonsecamuller99, DPPR18, kristensenraita, ST21, ruland}, we employ this framework to analyze 
the two-well energies \eqref{eq:SingPertEnergyIntro} and \eqref{eq:GradSingPertEnergyIntro} in an abstract setting,
which allows us to formulate some results more generally; see \autoref{thm:qwdom}, \autoref{thm:incompa} and \autoref{lem:LCP}.\medbreak

Let $X$ and $Y$ be real inner product spaces of dimension $n,m\in\NN$, respectively.
For $d\in\NN$, let $\cA:C^\infty(\RR^d; X )\rightarrow C^\infty(\RR^d;Y)$ be a \textit{linear, homogeneous, constant-coefficient differential operator of order $k\in\NN$}
\begin{equation}\label{eq:diffopera}
  \cA = \sum_{\abs{\alpha}=k} A_\alpha \del^\alpha,
\end{equation}
where $\abs{\alpha}=\sum_{i=1}^d \alpha_i $ denotes the length of the multi-index $\alpha\in\NN^d$ and $A_\alpha\in\Lin(X;Y)$.
If necessary, $\cA$ is understood in the sense of distributions. 
For brevity, when referring to a differential operator $\cA$ of the form \eqref{eq:diffopera},
we understand the spaces $X,Y$ and the dimension $d$ to be part of~$\cA$.
In order to further simplify the notation, we denote by $(\cdot,\cdot)_\cA$ the inner product in $X$, writing $v\perp_\cA w$ if two vectors $v,w\in X $ are perpendicular.
Moreover, we denote by $\lvert \cdot \rvert _\cA$ the norm in $X$ and by $\dist_\cA(\cdot,U)$ the distance associated to a set $U\subset X$.\medbreak

Now, let $\Omega\subset\RR^d$ be a bounded domain.
Let $F\in X$ be boundary data and $\cK=\{a_0, a_1\}\subset X$ a set of distinct wells. We consider
the \textit{$\cA$-free two-well energy}
\begin{equation}\label{eq:afreemulti}
  E_{el}^\cA(u,\chi):=\int_\Omega \abs{u-\chi}_\cA^2 \,dx,\quad u\in\cD^\cA_F(\Omega),\, \chi\in L^2(\Omega;\cK),
\end{equation}
where the \textit{set of admissible maps} is given by
\begin{equation}\label{eq:admapsa}
  \cD^\cA_F(\Omega):=\{u\in L^2_{\text{loc}}(\RR^d; X ):\cA u = 0 \text{ in } \RR^d, \,u=F \text{ on } \RR^d\setminus\overline{\Omega}\}.
\end{equation}
In addition, we introduce the \textit{singularly perturbed $\cA$-free two-well energy}
\begin{equation}\label{eq:SingPertAfreemulti}
  E_{\eps}^\cA(u,\chi):= E_{el}^\cA(u,\chi)+\eps E_{surf}(\chi),\quad u\in\cD^\cA_F(\Omega),\, \chi\in BV(\Omega;\cK),\,\eps\geq0,
\end{equation}
where $E_{surf}(\chi)$ is given by \eqref{eq:surfaceenergy}. 
\medbreak

In the general $\cA$-free setting, we seek to derive \textit{lower scaling bounds} for the zeroth-order-corrected energy $E_{\eps}^\cA(F,\cK)-E_0^\cA(F,\cK)$ as $\eps\to0$, where
\begin{equation}\label{eq:MinSingPertAfreeEnergy}
  E_{\eps}^\cA(F,\cK):=\inf\{E_\eps(u,\chi) : u\in\cD^\cA_F(\Omega),\, \chi\in BV(\Omega;\cK)\},\quad\eps\geq0.
\end{equation}
Under compatibility assumptions on the data (see \eqref{eq:compatiblewells} and \eqref{eq:compatiblebdrydata} below), which
guarantee that $E_0^\cA(F,\cK)$ vanishes, such lower scaling bounds have been studied in \cite{ruland,ruland2}.
Building on these works, we extend the analysis to the incompatible setting where the \textit{excess energy} $E_0^\cA(F,\cK)$ is strictly positive.\medbreak

In order to introduce the notion of compatibility in the $\cA$-free framework, we need some definitions.
For $\xi\in\RR^d\setminus\{0\}$, we set
\begin{equation}\label{eq:xicompatiblestates}
  V_\cA(\xi) := \ker(\AA(\xi))\subset X,
\end{equation}
where $\AA(\xi)$ denotes the symbol of $\cA$, which we take to be
\begin{equation}\label{eq:symbol}
  \AA(\xi)=\textstyle\sum_{\abs{\alpha}=k} \xi^\alpha A_\alpha \in \Lin(X,Y).
\end{equation}
Note that $V_\cA(\xi)$ only depends on the direction ${\xi}/{\lvert{\xi}\rvert}$.
The differential operator $\cA$ is said to be \textit{elliptic} if $\AA(\xi)$ is injective for all $\xi\in\rrdmz$.
As noted in \cite{tartar2, murat2}, the \textit{wave cone} associated to $\cA$ 
\begin{equation}\label{eq:wavecone}
  \Lambda_\cA := \bigcup_{\xi\in S^{d-1}} V_\cA(\xi)
\end{equation}
generalizes the (symmetrized) rank-one connections for the (symmetrized) gradient and indicates whether there exists a \textit{simple laminate} of two vectors in $X$.
More precisely, if $\cK=\{a_0,a_1\}\subset X$ with $a_1-a_0\in V_\cA(\xi)$ for some $\xi\in\RR^d$ and $h:\RR\to\{0,1\}$ is measurable,
then the function $u:\RR^d\to X$ given by $u(x):=(1-h(x\cdot\xi))a_0+h(x\cdot\xi)a_1$
is a solution to the \textit{differential inclusion}
\begin{equation}\label{eq:AfreeDifIncl}
\begin{cases}
    u\in\cK \text{ in }\RR^d,\\
    \cA u = 0 \text{ in }\RR^d.
\end{cases}
\end{equation}
In this context, we refer to a function $u:\RR^d\to\{a_0,a_1\}$
that only depends on $x\cdot\xi$ as a \textit{simple laminate of $a_0$ and $a_1$ with lamination direction $\xi$}.
Moreover, we note that the work \cite{DPPR18} showed that for $\cK=\{a_0,a_1\}\subset X$,
there exist non-constant solutions of \eqref{eq:AfreeDifIncl} if and only if $a_1-a_0\in\Lambda_\cA$,
which generalized the exact rigidity result \cite[Proposition~1]{bj}.
Based on this observation, the compatibility notions \eqref{eq:compatiblewellsIntro} and \eqref{eq:compatiblebdrydataIntro} generalize to the $\cA$-free setting as follows:
\begin{alignat}{2}
  &\hspace{-1.0cm}\textit{(i) The wells $a_0,a_1\in X$ are compatible if }&&a_1-a_0\in \Lambda_\cA, \label{eq:compatiblewells}\\
  &\hspace{-1.0cm}\textit{(ii) The boundary data $F\in X$ is compatible if }\hspace{0.1cm}&&F=(1-\lambda) a_0+\lambda a_1 \text{ for some }\lambda\in[0,1]. \label{eq:compatiblebdrydata}
\end{alignat}
We say that \textit{the data $(F,\cK)$ are compatible} if both assumptions \eqref{eq:compatiblewells} and \eqref{eq:compatiblebdrydata} are satisfied;
otherwise, they are said to be \textit{incompatible}.\medbreak
Next, we formulate the two-well energy for the gradient \eqref{eq:GradSingPertEnergyIntro} in the $\cA$-free framework by choosing $\cA$ to be the differential operator
$\curl:C^\infty(\RR^d;\RR^{d\times d})\rightarrow C^\infty(\RR^d;\RR^{d\times d\times d})$ given by
\begin{equation}\label{eq:curl}
  (\curl u)_{i,j,k}= \del_k u_{i,j}-\del_j u_{i,k}, \quad  1\leq i,j,k\leq d.
\end{equation}
This is motivated by the Poincaré lemma, which states that in a simply connected domain, a matrix field is a gradient if and only if it is $\curl$-free.
Similarly, for the geometrically linear energy \eqref{eq:SingPertEnergyIntro}, we choose $\cA$ as the differential operator
$\ccurl:C^\infty(\RR^d;\RR^{d\times d}_{sym})\rightarrow C^\infty(\RR^d;\RR^{d\times d\times d\times d})$~given~by
\begin{equation}\label{eq:curlcurl}
  (\ccurl u)_{i,j,k,l} = \del^2_{i,j}u_{k,l} + \del^2_{k,l}u_{i,j} - \del^2_{i,l}u_{k,j} - \del^2_{k,j}u_{i,l},\quad 1\leq i,j,k,l\leq d.
\end{equation}
This is based on the observation that in a simply connected domain $\Omega\subset\RR^d$, a symmetric matrix field $u:\Omega\to \rddsym$ is a symmetrized gradient if and only if the Saint-Venant compatibility condition $\ccurl u = 0$ is satisfied; see \cite{CiarletVenant} and \cite{gmeineder2022natural}.
When $\ccurl$ appears as subscript or superscript, we abbreviate it by $\cc$.\medbreak

In addition to the geometrically linear setting \eqref{eq:SingPertEnergyIntro} and the corresponding setting for gradients \eqref{eq:GradSingPertEnergyIntro},
we also study another prototypical example, namely the \textit{singularly perturbed two-well energy for the divergence}.
This energy is obtained from \eqref{eq:SingPertAfreemulti} by choosing $\cA$ to be the divergence operator $\Div:C^\infty(\RR^d;\RR^{d\times d})\to C^\infty(\RR^d;\RR^d)$
\begin{equation}\label{eq:Div}
(\Div u)_{i}=\textstyle\sum_{j=1}^d \del_j u_{i,j}.
\end{equation}
The divergence-free singularly perturbation model was first analyzed in \cite{ruland}, where quantitative rigidity estimates were established.
This analysis was inspired by the results in \cite{Garroni,PP04} on exact and approximate rigidity for the differential inclusion \eqref{eq:AfreeDifIncl}
with $\cA=\Div$ for two and three wells.
In this context, \cite[Theorem 1]{ruland} proved that the two-well energy $E_\eps^{\Div}$ scales as $\eps^{\nicefrac{2}{3}}$ for compatible data.
In the present work, we aim to extend this analysis to the incompatible setting.
We remark that the two-well problem for the divergence is of particular interest as any differential operator $\cA$ of the form \eqref{eq:diffopera}
can be reduced, via a suitable linear transformation, to the setting of a (possibly higher order) divergence operator; see \cite[Appendix A]{ST21} and \cite[Appendix B]{ruland}.\medbreak

The spaces $\rdd$, $\rddsym$, $\RR^{d\times d\times d}$ and $\RR^{d\times d\times d\times d}$ in \eqref{eq:curl}, \eqref{eq:curlcurl} and \eqref{eq:Div}
are endowed with the Frobenius inner product. Moreover, we restrict our attention to square matrix fields in \eqref{eq:curl} and \eqref{eq:Div} for simplicity.

\subsection{Minimization of an \texorpdfstring{$\cA$}{A}-free two-well energy without surface energy}
In order to study the scaling of $E_{\eps}^\cA(F,\cK)-E_0^\cA(F,\cK)$ for incompatible data $(F,\cK)$,
it will prove valuable to first characterize the excess energy $E_0^\cA(F,\cK)$.
The work of Kohn \cite{kohn} provides an explicit formula for the excess energy if $\cA\in\{\curl,\ccurl\}.$
We generalize this result to the $\cA$-free setting using Kohn's arguments in the context of $\cA$-quasiconvex envelope developed in \cite{fonsecamuller99,aquasi}
and applying the results of Rai\c t\u a \cite{raitapot19}.\medbreak

Before stating the main result on this topic, we need to introduce some terminology.
As noted in \cite{kohn}, minimization of the two-well energy \eqref{eq:afreemulti} is closely related to the following projection operators.
\begin{definition}[Compatibility projection]\label{def:compatibilityProjection}
Let $\cA$ be a differential operator as in \eqref{eq:diffopera}. 
For $\xi\in\RR^d\setminus\{0\}$, let $V_\cA(\xi)$ be given as in \eqref{eq:xicompatiblestates} and denote by $\ppa(\xi)$ the orthogonal projection onto $V_\cA(\xi)$ in $X$.
This induces the map
\begin{equation}\label{eq:compatibilityProjection}
  \begin{aligned}
    \ppa:\RR^d\setminus\{0\}&\rightarrow \Lin( X ),\\
    \xi &\mapsto \ppa(\xi),
  \end{aligned}
\end{equation}
to which we refer as the \textit{compatibility projection}.
\end{definition}
\begin{rmk}
As $V_\cA(\xi)$ only depends on ${\xi}/{\lvert\xi\rvert}$, the compatibility projection is zero-homogeneous:
\begin{equation}\label{eq:zerohom}
    \ppa(\xi)=\ppa\big(\tfrac{\xi}{\abs{\xi}}\big)\quad\forall\xi\in\RR^d\setminus\{0\}.
\end{equation}
\end{rmk}
Based on the notions from \cite{kohn} for $\cA\in\{\curl,\ccurl\}$, we define the following functions.
\begin{definition}[Compatibility quantifiers]\label{def:compquant}
Let $\cA$ be a differential operator as in \eqref{eq:diffopera}. Let $\ppa$ be given as in \autoref{def:compatibilityProjection}.
We define the two functions $h_\cA,\,g_\cA: X\to[0,\infty)$ for $a\in X$ by
\begin{equation}\label{eq:compquant}
h_\cA(a) := \inf_{\xi\in S^{d-1}} \abs{a-\ppa(\xi)a}_\cA^2, \qquad\qquad g_\cA(a) :=  \sup_{\xi\in S^{d-1}}\abs{\ppa(\xi)a}_\cA^2.
\end{equation}
We refer to both functions as compatibility quantifiers.
\end{definition}
\begin{rmk}
Let $a\in X$. For all $\xi\in\rrdmz$, there holds $(a-\ppa(\xi)a) \perp_\cA \ppa(\xi)a$, which implies
\begin{equation}\label{eq:compquanteqNotExtremal}
    \abs{a-\ppa(\xi)a}_\cA^2 + \abs{\ppa(\xi)a}_\cA^2 = \abs{a}^2_\cA\quad\forall\xi\in\rrdmz.
\end{equation}
This shows that the optimization problems in \eqref{eq:compquant} are equivalent, and that the compatibility quantifiers are connected through the relation
\begin{equation}\label{eq:compquanteq}
    h_\cA(a) + g_\cA(a) = \abs{a}^2_\cA.
\end{equation}
As the wave cone collects the compatible states, the following calculation shows that $h_\cA(a)$ indeed quantifies how far a state $a\in X$ is from being compatible:
\begin{equation}\label{eq:OptimalCompatibleApprox2}
  \dist^2(a, \Lambda_{\cA}) = \inf_{\xi\in\Sdm}\dist^2(a, V_{\cA}(\xi)) =  \inf_{\xi\in\Sdm} \abs{a-\ppa(\xi)a}^2 = h_\cA(a).
\end{equation}
\end{rmk}
In \autoref{subsec:aquasi}, we will see that the optimal $\xi\in S^{d-1}$ in \eqref{eq:compquant} determine the
lamination directions of optimal microstructures. This motivates the following definition.
\begin{definition}[Optimal lamination directions]\label{def:optLaminationDirections}
Let $\cA$ be a differential operator as in \eqref{eq:diffopera}. Let $\ppa$ be given as in \autoref{def:compatibilityProjection}.
For $a\in X$, we define the set of optimal lamination directions by
\begin{equation}\label{eq:optLaminationDirections}
  S_\cA(a) := \{\xi\in S^{d-1}: \abs{\ppa(\xi)a}_\cA^2 = g_\cA(a)\}\subset S^{d-1}.
\end{equation}
\end{definition}
To apply the results of \cite{raitapot19}, we often assume that $\cA$ satisfies the \textit{constant rank property.}
\begin{definition}[Constant rank property, \cite{wilcox}]\label{def:constantrank}
A differential operator $\cA$ as in \eqref{eq:diffopera} is said to have constant rank if there exists $\raa\in\NN$ such that
its symbol $\AA(\xi)$ (see \eqref{eq:symbol}) satisfies
\begin{equation}\label{eq:constantrank}
\rank \AA(\xi) = \raa \quad \forall\xi\in\RR^d \setminus\{0\}.
\end{equation}
\end{definition}
In \autoref{subsec:ApplicationOfGeneralLowerScaling}, we will see that the differential operators $\cA\in\{\curl,\Div,\ccurl\}$ have constant rank.
Finally, given a set of wells $\cK=\{a_0,a_1\}\subset X$, we denote the \textit{weighted average} by
\begin{equation}\label{eq:atheta}
  a_\theta := (1-\theta)a_0 + \theta a_1\in X,\quad\theta\in[0,1].
\end{equation}
We are now ready to state the main result on the minimization of the $\cA$-free two-well energy without surface energy,
which for $\cA\in\{\curl,\ccurl\}$ had already been established in \cite{kohn}.
\begin{theorem}\label{thm:qwdom}
Let $\cA$ be a differential operator as in \eqref{eq:diffopera} satisfying the constant rank property; see \autoref{def:constantrank}.
Let $\Omega\subset\RR^d$ be bounded domain.
Given $F\in X$ and $\cK=\{a_0,a_1\}\subset X$, let $E_0^\cA(F,\cK)$ be given as in \eqref{eq:MinSingPertAfreeEnergy}.
Let $a_\theta$ be as in \eqref{eq:atheta}. For $a:=a_1-a_0$, let $h_\cA(a)$ be as in \autoref{def:compquant}.
Then, it holds that
\begin{equation}\label{eq:baseenergyAfree}
  E_0^\cA(F,\cK)  = \abs{\Omega}\min_{\theta\in[0,1]} \Big(\abs{F-a_\theta}_\cA^2 +\theta(1-\theta)h_\cA(a)\Big).
\end{equation}
\end{theorem}
This extends \eqref{eq:E0VanishIFFDataComp} to the $\cA$-free setting and relates the excess energy $E_0^\cA(F,\cK)$ to the incompatibility
of the data, measured by $h_\cA(a)$ and the distance of $F$ to the line segment $\overline{a_0a_1}$.
As we will see in the proof of \autoref{thm:qwdom}, the variable $\theta$ in \eqref{eq:baseenergyAfree} arises by minimizing \eqref{eq:MinSingPertAfreeEnergy}
subject to the constraint that the phase $\{\chi=a_1\}$ has the volume fraction $\theta$, which is afterwards removed by minimizing over $\theta\in[0,1]$.
For simplicity, we want to restrict to the setting, where the optimal $\theta$ in \eqref{eq:baseenergyAfree} is unique.
Therefore, we require the following property.
\begin{definition}[Spanning wave cone, \cite{guerraraita}]\label{def:SpanWaveCone}
A differential operator $\cA$ as in \eqref{eq:diffopera} is said to have spanning wave cone if
\begin{equation}\label{eq:spanningwavecone}
\spn \Lambda_\cA =  X.
\end{equation}
\end{definition}
This assumption is standard in the $\cA$-free framework \cite{guerraraita,kristensenraita}, and it is satisfied by the model differential operators
$\cA\in\{\curl,\Div,\ccurl\}$; see \autoref{subsec:ApplicationOfGeneralLowerScaling}.
Restricting our analysis to differential operators with spanning wave cone bears the following advantage.
\begin{proposition}[Optimal volume fraction]\label{prop:optvolfrac}
Assume the hypotheses of \autoref{thm:qwdom}. In addition, suppose that $\cA$ has spanning wave cone.
Then, there exists a unique minimizer $\theta\in[0,1]$ for \eqref{eq:baseenergyAfree}
\begin{equation}\label{eq:optvolfrac}
  \ttheta_\cA(F,\cK):=\argmin_{\theta\in[0,1]} \Big(\abs{F-a_\theta}_\cA^2 +\theta(1-\theta)h_\cA(a)\Big)
\end{equation}
to which we refer as the optimal volume fraction.
\end{proposition}
\begin{rmk}
Given a set $\cK=\{a_0,a_1\}\subset X$, the optimal volume fraction partitions the space $X$ into three regions:
two opposing half spaces and the slab separating them. 
As illustrated in \autoref{fig:optvolfrac}, the half spaces comprise all $F\in X$ with $\ttheta_\cA(F,\cK)=0$ and $\ttheta_\cA(F,\cK)=1$, respectively,
while the slab contains all $F\in X$ with $\ttheta_\cA(F,\cK)\in(0,1)$. The geometry of this partition is discussed in more detail in \autoref{rmk:GeometryOptVolFrac} below.
\end{rmk}
\begin{figure}
\begin{center}

\tikzset{every picture/.style={line width=0.75pt}}      

\begin{tikzpicture}[x=0.75pt,y=0.75pt,yscale=-1,xscale=1]
\draw  [draw opacity=0] (429.89,20.33) -- (549.89,20.33) -- (549.89,200.33) -- (429.89,200.33) -- cycle ; \draw  [color={rgb, 255:red, 204; green, 204; blue, 204 }  ,draw opacity=1 ] (459.89,20.33) -- (459.89,200.33)(489.89,20.33) -- (489.89,200.33)(519.89,20.33) -- (519.89,200.33) ; \draw  [color={rgb, 255:red, 204; green, 204; blue, 204 }  ,draw opacity=1 ] (429.89,50.33) -- (549.89,50.33)(429.89,80.33) -- (549.89,80.33)(429.89,110.33) -- (549.89,110.33)(429.89,140.33) -- (549.89,140.33)(429.89,170.33) -- (549.89,170.33) ; \draw  [color={rgb, 255:red, 204; green, 204; blue, 204 }  ,draw opacity=1 ] (429.89,20.33) -- (549.89,20.33) -- (549.89,200.33) -- (429.89,200.33) -- cycle ;
\draw  [draw opacity=0] (99.89,20.33) -- (219.89,20.33) -- (219.89,200.33) -- (99.89,200.33) -- cycle ; \draw  [color={rgb, 255:red, 204; green, 204; blue, 204 }  ,draw opacity=1 ] (129.89,20.33) -- (129.89,200.33)(159.89,20.33) -- (159.89,200.33)(189.89,20.33) -- (189.89,200.33) ; \draw  [color={rgb, 255:red, 204; green, 204; blue, 204 }  ,draw opacity=1 ] (99.89,50.33) -- (219.89,50.33)(99.89,80.33) -- (219.89,80.33)(99.89,110.33) -- (219.89,110.33)(99.89,140.33) -- (219.89,140.33)(99.89,170.33) -- (219.89,170.33) ; \draw  [color={rgb, 255:red, 204; green, 204; blue, 204 }  ,draw opacity=1 ] (99.89,20.33) -- (219.89,20.33) -- (219.89,200.33) -- (99.89,200.33) -- cycle ;
\draw [color={rgb, 255:red, 0; green, 0; blue, 0 }  ,draw opacity=1 ]   (189.78,110.33) -- (459.78,110.33) ;
\draw  [color={rgb, 255:red, 0; green, 174; blue, 179 }  ,draw opacity=1 ][fill={rgb, 255:red, 0; green, 174; blue, 179 }  ,fill opacity=1 ] (183.93,110.33) .. controls (183.93,107.04) and (186.6,104.37) .. (189.89,104.37) .. controls (193.18,104.37) and (195.85,107.04) .. (195.85,110.33) .. controls (195.85,113.62) and (193.18,116.29) .. (189.89,116.29) .. controls (186.6,116.29) and (183.93,113.62) .. (183.93,110.33) -- cycle ;
\draw  [color={rgb, 255:red, 0; green, 174; blue, 179 }  ,draw opacity=1 ][fill={rgb, 255:red, 0; green, 174; blue, 179 }  ,fill opacity=1 ] (453.93,109.33) .. controls (453.93,106.04) and (456.6,103.37) .. (459.89,103.37) .. controls (463.18,103.37) and (465.85,106.04) .. (465.85,109.33) .. controls (465.85,112.62) and (463.18,115.29) .. (459.89,115.29) .. controls (456.6,115.29) and (453.93,112.62) .. (453.93,109.33) -- cycle ;
\draw   (219.89,140.65) -- (250.64,140.65) -- (250.64,109.89) ;
\draw  [fill={rgb, 255:red, 0; green, 0; blue, 0 }  ,fill opacity=1 ] (238.06,124.92) .. controls (238.06,126.27) and (236.97,127.37) .. (235.62,127.37) .. controls (234.26,127.37) and (233.17,126.27) .. (233.17,124.92) .. controls (233.17,123.57) and (234.26,122.47) .. (235.62,122.47) .. controls (236.97,122.47) and (238.06,123.57) .. (238.06,124.92) -- cycle ;

\draw    (220.2,80.8) -- (190.2,80.8) ;
\draw [shift={(190.2,80.8)}, rotate = 360] [color={rgb, 255:red, 0; green, 0; blue, 0 }  ][line width=0.75]    (0,5.59) -- (0,-5.59)   ;
\draw [shift={(220.2,80.8)}, rotate = 360] [color={rgb, 255:red, 0; green, 0; blue, 0 }  ][line width=0.75]    (0,5.59) -- (0,-5.59)   ;
\draw    (459.49,80.01) -- (430.2,80.01) ;
\draw [shift={(430.2,80.01)}, rotate = 360] [color={rgb, 255:red, 0; green, 0; blue, 0 }  ][line width=0.75]    (0,5.59) -- (0,-5.59)   ;
\draw [shift={(459.49,80.01)}, rotate = 360] [color={rgb, 255:red, 0; green, 0; blue, 0 }  ][line width=0.75]    (0,5.59) -- (0,-5.59)   ;
\draw  [dash pattern={on 4.5pt off 4.5pt}]  (292.78,23.19) -- (273.66,110.13) ;
\draw  [color={rgb, 255:red, 0; green, 174; blue, 179 }  ,draw opacity=1 ][fill={rgb, 255:red, 0; green, 174; blue, 179 }  ,fill opacity=1 ] (287.93,23.85) .. controls (287.93,20.99) and (290.25,18.67) .. (293.12,18.67) .. controls (295.98,18.67) and (298.3,20.99) .. (298.3,23.85) .. controls (298.3,26.72) and (295.98,29.04) .. (293.12,29.04) .. controls (290.25,29.04) and (287.93,26.72) .. (287.93,23.85) -- cycle ;
\draw  [color={rgb, 255:red, 0; green, 174; blue, 179 }  ,draw opacity=1 ][fill={rgb, 255:red, 0; green, 174; blue, 179 }  ,fill opacity=1 ] (267.93,110) .. controls (267.93,106.71) and (270.6,104.04) .. (273.89,104.04) .. controls (277.18,104.04) and (279.85,106.71) .. (279.85,110) .. controls (279.85,113.29) and (277.18,115.96) .. (273.89,115.96) .. controls (270.6,115.96) and (267.93,113.29) .. (267.93,110) -- cycle ;
\draw   (399.77,110.21) -- (399.77,140.65) -- (430.2,140.65) ;
\draw  [fill={rgb, 255:red, 0; green, 0; blue, 0 }  ,fill opacity=1 ] (415.33,128.2) .. controls (413.99,128.2) and (412.91,127.11) .. (412.91,125.78) .. controls (412.91,124.44) and (413.99,123.36) .. (415.33,123.36) .. controls (416.67,123.36) and (417.75,124.44) .. (417.75,125.78) .. controls (417.75,127.11) and (416.67,128.2) .. (415.33,128.2) -- cycle ;

\draw (165,120) node [anchor=north west][inner sep=0.75pt]  [font=\Large] [align=left] {${a_{0}}$};
\draw (466,120) node [anchor=north west][inner sep=0.75pt]  [font=\Large] [align=left] {${a_{1}}$};
\draw (425.5,34) node [anchor=north west][inner sep=0.75pt]  [font=\Large] [align=left] {$\tfrac{h_\cA(a)}{2\abs{a}_\cA}$};
\draw (185.5,34) node [anchor=north west][inner sep=0.75pt]  [font=\Large] [align=left] {$\tfrac{h_\cA(a)}{2\abs{a}_\cA}$};
\draw (110,174) node [anchor=north west][inner sep=0.75pt]  [font=\Large] [align=left] {$\ttheta_{\cA}( \cdot ,\cK)=0$};
\draw (325,174) node [anchor=north][inner sep=0.75pt]  [font=\Large] [align=left] {$\ttheta_{\cA}( \cdot ,\cK)\in ( 0,1)$};
\draw (441,174) node [anchor=north west][inner sep=0.75pt]  [font=\Large] [align=left] {$\ttheta_{\cA}( \cdot ,\cK)=1$};
\draw (302,14) node [anchor=north west][inner sep=0.75pt]  [font=\Large] [align=left] {$F$};
\draw (270,120) node [anchor=north west][inner sep=0.75pt]  [font=\Large] [align=left] {$a_\theta$};

\end{tikzpicture}
\end{center}
\caption{Given $\cK=\{a_0,a_1\}\subset X$, the optimal volume fraction $\ttheta_\cA(\cdot,\cK)$ partitions $X$ into two opposing half spaces and a slab separating them.
As detailed in \autoref{rmk:GeometryOptVolFrac}, the boundaries of the half spaces are orthogonal to $a_1-a_0$,
and the line segment $\overline{a_0a_1}$ extends into each half space by a distance of $\tfrac{h_\cA(a)}{2\lvert{a}\rvert _\cA}$.}
\label{fig:optvolfrac}
\end{figure}
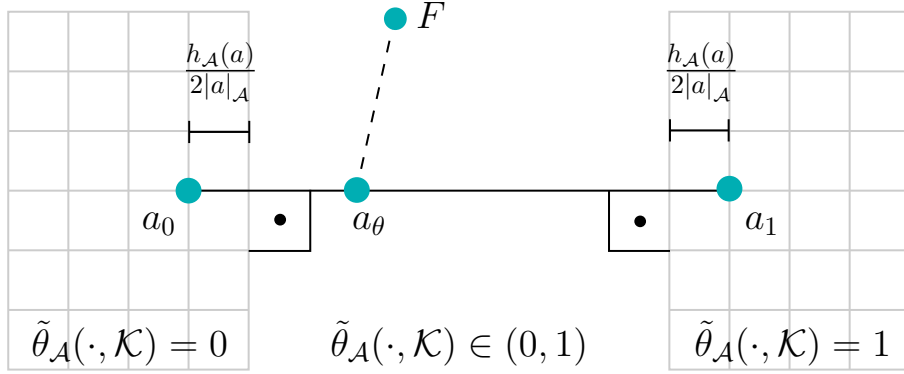  

\subsection{Scaling laws for the model two-well problems}
We now come to the main results of this article.
These characterize the higher-order scaling behavior of the singularly perturbed two-well energy in the geometrically linear setting, in the corresponding gradient setting and in the divergence-free setting, removing the compatibility assumptions \eqref{eq:compatiblewells} and \eqref{eq:compatiblebdrydata}.
These scaling laws provide new information on the complexity of optimal microstructures arising when the data are incompatible,
indicating that branching-type microstructures, which refine near the boundary, are energetically favored.\medbreak

We prove that the typical $\eps^{\nicefrac{2}{3}}$-scaling, originally derived in \cite{km92, km94},
also applies in the incompatible setting upon subtracting the zeroth-order excess energies for both the two-well energies for the gradient and the divergence.
\begin{breaktheorem}[Scaling of the two-well energies for the gradient and the divergence]\label{thm:incompgrad}
Let $\Omega\subset \RR^d$ be a bounded Lipschitz domain with $d\geq2$.
Let $\cA\in\{\curl,\Div\}$; see \eqref{eq:curl} and \eqref{eq:Div}.
Given $F\in\RR^{d\times d}$ and $\cK=\{a_0,a_1\}\subset\rdd$, let $E^{\cA}_\eps(F,\cK)$ be given as in \eqref{eq:MinSingPertAfreeEnergy}.
For $a:=a_1-a_0$, let $S_{\cA}(a)$ be given as in \autoref{def:optLaminationDirections}.
Then, the optimal volume fraction $\ttheta_{\cA}=\ttheta_{\cA}(F,\cK)\in[0,1]$ (see \autoref{prop:optvolfrac}) determines the energy scaling in the following way:
\begin{enumerate}
\item(Pure phase) If $\ttheta_{\cA} \in \{0,1\}$, there holds the trivial energy scaling
\begin{equation}\label{eq:gradientpolarcase}
    E^{\cA}_\eps (F,\cK)-E^{\cA}_0(F,\cK) = 0 \quad \forall \eps>0
\end{equation}
as is observed by considering the minimizer $(u,\chi)$ of \eqref{eq:MinSingPertAfreeEnergy}
given by the constant maps $u = F$ and $\chi =(1-\ttheta_{\cA})a_0 + \ttheta_{\cA} a_1$.

\item(Phase mixing -- lower bound) If $\ttheta_{\cA}\in(0,1)$ and $a\not\in \RR O(d)$ (scaled orthogonal matrices),
there exist $c=c(a,\Omega,d,\ttheta_{\cA})>0$ and $\eps_0= \eps_0(a,\Omega,d, \ttheta_{\cA})>0$ such that
\begin{equation}\label{eq:ThmGradLowerBound}
    E^{\cA}_\eps (F,\cK)-E^{\cA}_0(F,\cK)
    \geq c\eps^{\nicefrac{2}{3}} \quad \forall \eps\in (0,\eps_0).
\end{equation}

\item(Phase mixing -- upper bound) If $\ttheta_{\cA}\in(0,1)$ and $d=2$, let $\xi^*\in S_{\cA}(a)$ and assume that $\Omega\subset \RR^2$ is a rotated unit square with two faces normal to $\xi^*$.
Then, we have the matching upper bound: there exists $C=C(a, \ttheta_{\cA})>0$ such that
\begin{equation}\label{eq:ThmGradUpperBound}
    E^{\cA}_\eps (F,\cK)-E^{\cA}_0(F,\cK)
    \leq C\eps^{\nicefrac{2}{3}} \quad \forall \eps \in (0,1).
\end{equation}
\end{enumerate}
\end{breaktheorem}
This result offers an almost complete picture of the scaling behavior for $d=2$.
It illustrates that either a single phase fills the entire domain without any interfacial energy; or 
the phases mix, producing branching-type microstructures of Kohn-Müller type \cite{km92,km94}.
We prove the upper bound \eqref{eq:ThmGradUpperBound} using a branching construction which agrees with the one in \cite{km92,km94} if the data are compatible
but is somewhat different for incompatible data.
In the latter case, the constructed $\cA$-free map $u\in\cD_F^{\cA}(\Omega)$, which is used to estimate \eqref{eq:MinSingPertAfreeEnergy}, oscillates between two compatible matrices
that approximate the wells.\medbreak

To complete the picture for $d=2$, it remains to determine the energy scaling when $a\in\RR O(2)$.
In this case, all lamination directions are optimal:
\begin{equation}\label{eq:equiCompIntro}
  S_{\cA}(a)=S^{1},
\end{equation}
which is a rather peculiar situation not occurring in the compatible setting; see \autoref{subsec:ApplicationOfGeneralLowerScaling}.
This is also the reason, why the Fourier-based technique (\autoref{thm:incompa}), used in the proof of \eqref{eq:ThmGradLowerBound}, fails to produce a non-trivial lower bound in this case,
as discussed in \autoref{rmk:FourierEquicomp}.
As \eqref{eq:equiCompIntro} may increase flexibility, it is possible that a non-$\eps^{\nicefrac{2}{3}}$-scaling behavior arises when $a\in\RR O(2)$.
It is planned to address this question in a forthcoming article \cite{MPF26}.\medbreak

Based on the rotation-type arguments introduced in \cite{ruland,ruland3}, the upper bound \eqref{eq:ThmGradUpperBound} is expected to remain valid in higher dimensions $d\geq3$.
For simplicity, we only carry out the upper bound constructions for $d=2$ in the present work.\medbreak

We now come to the main result in the geometrically linear setting, which generalizes the well-known scaling behavior summarized in \autoref{thm:chanconti}. 
\begin{breaktheorem}[Scaling of the geometrically linear two-well energy]\label{thm:incompsyma}
Let $\Omega\subset \RR^2$ be a bounded Lipschitz domain.
Let $d=2$ and $\cA=\ccurl$; see \eqref{eq:curlcurl}.
Given $F\in\rddsymd{2}$ and $\cK=\{a_0,a_1\}\subset\rddsymd{2}$, let $E^{\cc}_\eps(F,\cK)$ be given as in \eqref{eq:MinSingPertAfreeEnergy}.
For $a:=a_1-a_0$, let $S_{\cc}(a)$ be given as in \autoref{def:optLaminationDirections}.
Then, the optimal volume fraction $\ttheta_{\cc}=\ttheta_{\cc}(F,\cK)\in[0,1]$ (see \autoref{prop:optvolfrac}) determines the energy scaling in the following way:
\begin{enumerate}
\item(Pure phase) If $\ttheta_{\cc} \in \{0,1\}$, it holds that
\begin{equation}
    E^{\cc}_\eps (F,\cK)-E^{\cc}_0(F,\cK) = 0 \quad \forall \eps>0
\end{equation}
as is observed by considering the minimizer $(u,\chi)$ of \eqref{eq:MinSingPertAfreeEnergy}
given by the constant maps $u = F$ and $\chi = (1-\ttheta_{\cc})a_0 + \ttheta_{\cc} a_1$.

\item(Phase mixing -- lower bound) If $\ttheta_{\cc}\in(0,1)$, assume that $a \not\in \RR I_2$ (scaled identity). Then,
there exist $c=c(a,\Omega, \ttheta_{\cc})>0$ and $\eps_0= \eps_0(a,\Omega, \ttheta_{\cc})>0$ such that
\begin{equation}\label{eq:ThmCCLowerBound}
    E^{\cc}_\eps(F,\cK)-E^{\cc}_0(F,\cK)
    \geq 
    \begin{cases}
    c\eps^{\nicefrac{4}{5}} &\text{if } \rank a =1,\\
    c\eps^{\nicefrac{2}{3}} &\text{if } \rank a =2,
    \end{cases}
    \quad
    \forall \eps\in (0,\eps_0).
\end{equation}
\item(Phase mixing -- upper bound) 
If $\ttheta_{\cc}\in(0,1)$, let $\xi^*\in S_{\cc}(a)$ and assume that $\Omega\subset \RR^2$ is a rotated unit square with two faces normal to $\xi^*$.
Then, we have the matching upper bound: there exists $C=C(a, \ttheta_{\cc})>0$ such that
\begin{equation}\label{eq:ThmCCUpperBound}
    E^{\cc}_\eps(F,\cK)-E^{\cc}_0(F,\cK)
    \leq
    \begin{cases}
    C\eps^{\nicefrac{4}{5}} &\text{if } \rank a =1,\\
    C\eps^{\nicefrac{2}{3}} &\text{if } \rank a =2,
    \end{cases}
    \quad
    \forall \eps\in (0,1).
\end{equation}
\end{enumerate}
\end{breaktheorem}
Aside from the case $a\in\RR I_2$, which corresponds to the peculiar situation where $S_{\cc}(a)=S^1$, this result fully characterizes the scaling behavior in two dimensions.
Assuming that the phases mix, the scaling law indicates that the energy minimization produces branched microstructures of the
Kohn-Müller type \cite{km92,km94} when $a$ has rank two and of the Chan-Conti type \cite[Theorem 1.2]{cc15} when $a$ has rank one.
In the latter case, the minimal singularly perturbed energy can be substantially reduced compared to 
the $\eps^{\nicefrac{2}{3}}$-scaling by an interaction of the components of the displacement field.
A consequence of \autoref{thm:incompsyma} is that such an additional energy reduction is not possible if the wells are incompatible
since then the corrected energy scales as~$\eps^{\nicefrac{2}{3}}$.\medbreak

In the proof of the upper bounds \eqref{eq:ThmCCUpperBound} for incompatible data,
we construct a linear strain that oscillates between two compatible states which approximate the stress-free strains.
The $\eps^{\nicefrac{2}{3}}$-upper bound is derived by linking the gradient theory and the geometrically linear theory,
thereby avoiding an explicit upper bound construction and relying instead on the construction used in the proof of \autoref{thm:incompgrad}.
We prove the $\eps^{\nicefrac{4}{5}}$-upper bound employing the Chan-Conti~branching construction \cite{cc15}, repeating their arguments 
in the setting of incompatible boundary data.\medbreak

Generalizing these results to higher dimensions presents additional challenges, both on the level of the lower bounds and the upper bounds.
We therefore defer the analysis of the scaling laws for $d\geq3$ to future work.
In this context, the recent work \cite{machill2025energyscalingbehavioursingular} extended the Chan-Conti construction \cite{cc15} to a three-dimensional cylindrical domain.\medbreak

To conclude our discussion of the main results, we briefly turn to the energy scaling behavior in the $\cA$-free setting.
Here, we highlight \autoref{thm:incompa} which generalizes the lower bounds \eqref{eq:ThmGradUpperBound} and \eqref{eq:ThmCCUpperBound}.
Using the Fourier-theoretic framework developed in \cite[Theorem 1.4]{ruland2}, this theorem provides a systematic method to derive
lower scaling bounds for 
\begin{equation}
    E_\eps^\cA(F,\cK)-E_0^\cA(F,\cK)\text{ as }\eps\to 0,
\end{equation}
applicable to a broad class of differential operators $\cA$ of the form \eqref{eq:diffopera}
for both compatible and incompatible data $F\in X$ and $\cK=\{a_0,a_1\}\subset X$.

\subsection{Relation to the literature} 
Differential inclusions and their associated multi-well energies arise naturally in variational models of martensitic phase transformations
\cite{bj, bj2, kohn, DolzmannMueller-1995, muller1, MS, chai, dolzmann, CDK, rindler, CDPRZZ20, RS23}.
Starting with the foundational work \cite{km92,km94}, a large body of research has developed quantitative results for the corresponding singular perturbation models,
producing upper bound constructions and matching rigidity/lower bound estimates
\cite{CO09, CO12,Z14,cc15,conti2016low, ruland2016rigidity,RZZ18,RTZ19, RZZ19,  Sim21a,Sim21b, RT22Tartar, ruland3, AKKR24, RT24}.
Important complementary contributions include self-similarity of minimizers of the Kohn-Müller energy \cite{conti1}
and the identification of sharp-interface limits \cite{CSGammaSO2, CSGammaGeoLin,davoli2020two}.
Closely related nucleation phenomena have been investigated in \cite{KK11, kko13, ball2013nucleation, conti2017piecewise, KO19,CDMZ20, RT23a, TZ24, GRTZ24}.\medbreak

Rooted in the work of Murat and Tartar \cite{murat1, tartar1, tartar2}, the $\cA$-free framework provides a unified PDE-constrained viewpoint
and underpins modern results on compensated compactness and relaxation theory \cite{dacoWeakCont,fonsecamuller99, aquasi, raitapot19, SW21, guerraraita, kristensenraita, GRS22}.
Approximative rigidity of the $\cA$-free differential inclusion \eqref{eq:AfreeDifIncl} with two wells was established in \cite{DPPR18}.
Further, quantitative rigidity properties of the corresponding singular perturbation model were derived in \cite{ruland,ruland2}.
Beyond the differential operators $\cA\in\{\curl,\ccurl\}$ arising in nonlinear and linear elasticity,
the divergence operator $\cA=\Div$ has received considerable attention in two- and three-well problems \cite{PP04, Garroni, PS09, ruland, ruland2}.\medbreak

The present article is inspired by Kohn's work \cite{kohn} on the relaxation of incompatible two-well energies (without surface energy)
and aims to investigate the corresponding singular perturbation models and their generalization within the $\cA$-free framework.
The results of \autoref{thm:incompgrad} and \autoref{thm:incompsyma} show that for $\cA\in\{\curl,\Div,\ccurl\}$,
the singularly perturbed $\cA$-free two-well energies produce microstructures analogous to those in the compatible case \cite{km92, km94, cc15, ruland, ruland2}.
Furthermore, \autoref{thm:incompa} is motivated by the aim of extending \cite[Theorem 3]{ruland2} to the incompatible setting
and builds upon the Fourier-based technique developed in \cite{ruland2}.\medbreak

Finally, the problems studied here are closely related to variational models in compliance minimization \cite{KohnWirth, KW16, PW22},
micromagnetics \cite{CKO99, DesimoneKnuepferOtto-2006, DesimoneKohnMuellerOtto-2006}, microstructure in composites \cite{PS09, Palombaro10} 
and dislocation microstructures in plasticity \cite{CO05Dislocation}.

\subsection{Outline of the article}
The remainder of this article is structured as follows:
In \autoref{sec:notation}, we collect the relevant notation.
In \autoref{sec:relaxation}, we prove \autoref{thm:qwdom} using the theory of relaxations in the $\cA$-free setting.
In \autoref{sec:LowerBounds}, we establish the general result on singularly perturbed $\cA$-free two-well energies (\autoref{thm:incompa}),
which is then applied to the model problems to prove the first and second parts of both \autoref{thm:incompgrad} and \autoref{thm:incompsyma}.
In \autoref{sec:UpperBoundGradDiv}, we complete the proof of \autoref{thm:incompgrad} employing an $\eps^{\nicefrac{2}{3}}$-upper bound construction.
In \autoref{sec:UpperBoundCC}, we turn to the upper bounds of \autoref{thm:incompsyma} for which we implement an $\eps^{\nicefrac{4}{5}}$-upper bound construction.
In \autoref{sec:AppendixA}, we provide some details to a relaxation result (\autoref{lem:domaininvarianceInText}) that we use in \autoref{sec:relaxation}.

\section{Notation}\label{sec:notation}
In this section, we summarize the relevant notation and conventions.
\begin{itemize}
\item A set $\Omega\subset\RR^d$ is said to be a domain if it is nonempty, open and connected.
\item When writing $\cK=\{a_0,a_1\}$, it is understood that the elements are distinct. We will typically write $a=a_1-a_0$ for the difference and
$a_\theta$ for the weighted average as given by \eqref{eq:atheta}.
\item The energy $E^\cA_\eps(F,\cK)$ defined in \eqref{eq:MinSingPertAfreeEnergy} depends on the domain $\Omega$ which may be explicitly indicated as $E^\cA_\eps(F,\cK;\Omega)$ when ambiguous.
\item When referring to a differential operator $\cA$ of the form \eqref{eq:diffopera}, we understand $\cA$ to stand for the differential operator as well as the inner product spaces $X$ and $Y$. In particular, the notation discussed below \eqref{eq:diffopera} for the inner product on $X$ and its associated norm/distance/orthogonality are in effect.
\item In our analysis of incompatible $\cA$-free two-well problems, the following quantities will play an important role:
the compatibility projection $\ppa$ given as in \autoref{def:compatibilityProjection},
the compatibility quantifiers $h_{\cA}(a)$ and $g_{\cA}(a)$ given as in \autoref{def:compquant},
the set of optimal lamination directions $S_{\cA}(a)$ given as in \autoref{def:optLaminationDirections} and
the optimal volume fraction $\ttheta_{\cA}(F,\cK)$ given as in \eqref{eq:optvolfrac}.
\item Given a function $f\in L^1(\Omega)$, we denote its average by 
$\overline{f}=\fint_\Omega f\, dx = \tfrac{1}{\abs{\Omega}}\int_\Omega f\, dx.$
\item The space $\RR^d$ is understood to be equipped with the Euclidean inner product. The spaces $\rdd$ and $\rddsym$ are endowed with the Frobenius inner product.
\item We identify the torus $\TT^d$ with the standard unit cube $[0,1]^d$.
\item Given $f\in L^1(\TT^d)$, we denote its Fourier transform by $\hat{f}(\xi)=\int_{\TT^d} f(x) e^{-2\pi i\xi\cdot x} dx$ for $\xi\in\ZZ^d$.
\item Given $f\in L^1(\RR^d)$, we denote its Fourier transform by $\hat{f}(\xi)=(2\pi)^{-\frac{d}{2}}\int_{\RR^d}f(x) e^{-i\xi\cdot x} dx$ for $\xi\in\RR^d$.
\item For a matrix $a\in\rdd$ with eigenvalue $\lambda\in\RR$, we denote by $E(a,\lambda)=\ker (a-\lambda I_d)$ the associated eigenspace.
Moreover, we write $\lambda_-=\lambda_-(a)$ for the smallest eigenvalue of $a$ and $\lambda_+=\lambda_+(a)$ for the largest eigenvalue of $a$, if they exist.
As we will need to compare the eigenvalues of $a$ with the eigenvalues of $a^Ta$, we reserve the
notation $\lambda_{min}=\lambda_{min}(a^Ta)$ and $\lambda_{max}=\lambda_{max}(a^Ta)$ for the smallest and largest eigenvalues of $a^Ta$, respectively.
\item Given a finite-dimensional real vector space $X$ and a domain $\Omega\subset\RR^d$, we denote by $\cM(\Omega;X)$ the space of finite $X$-valued Radon measures.
\end{itemize}

\section{Relaxation under an \texorpdfstring{$\cA$}{A}-free constraint}\label{sec:relaxation}
In this section, we study the minimization of the $\cA$-free two-well energy \eqref{eq:afreemulti} with the goal of proving \autoref{thm:qwdom}.
After optimizing \eqref{eq:afreemulti} in $\chi\in L^2(\Omega;\cK)$, it remains to minimize
\begin{equation}\label{eq:AFreeEnergyIntroRelaxation}
    E^{\cA}_{el}(u;\cK):=\int_\Omega \dist^2_\cA(u,\cK)\,dx\text{ among } u\in\cD^\cA_F(\Omega).
\end{equation}
In \autoref{subsec:aquasi}, we thus compute the $\cA$-quasiconvex envelope of the energy density $W:=\dist^2_\cA(\cdot,\cK)$ 
by adapting Kohn's arguments \cite{kohn} to the $\cA$-free setting.
Our proof of \autoref{thm:qwdom} then builds on the work \cite{raitapot19}, which shows that the $\cA$-quasiconvex envelope fully determines the minimum of \eqref{eq:AFreeEnergyIntroRelaxation}; see \autoref{lem:domaininvarianceInText}.\medbreak

In \autoref{subsec:compApprox}, we establish uniqueness of the optimal volume fraction (\autoref{prop:optvolfrac}) and discuss its geometric interpretation.
Finally, we conclude this section by introducing the notion of compatible approximations (\autoref{def:compatibleApprox}).

\subsection{The \texorpdfstring{$\cA$}{A}-quasiconvex envelope}\label{subsec:aquasi}
We begin by recalling the notion of $\cA$-quasiconvex envelope in \autoref{def:aquasi} which can be considered a
periodic variant of the variational problem \eqref{eq:AFreeEnergyIntroRelaxation}.
The notion of $\cA$-quasiconvex envelope was first introduced in \cite{fonsecamuller99} and
then used to generalize relaxation results to the $\cA$-free setting in \cite{aquasi}.
\begin{definition}[$\cA$-quasiconvex envelope]\label{def:aquasi}
Let $\cA$ be a differential operator as in \eqref{eq:diffopera} and $\cK=\{a_0,a_1\}\subset X$.
Let $W:=\dist_\cA^2(\cdot,\cK): X\to\RR$.
Following \cite{fonsecamuller99}, we introduce the $\cA$-quasiconvex envelope of $W$ as the function $Q^\cA W: X\to\RR$ given by
\begin{equation}\label{eq:AQuasiconvexification}
  Q^\cA W(F) := \inf \left\{\int_{\TT^d} W(u)\,dx : u\in\cD^{\cA,\textup{per}}_F\right\},\quad F\in X,
\end{equation}
where the set of admissible maps is defined as
\begin{equation}\label{eq:admissibleMapsPeriodic}
    \cD_F^{\cA,\textup{per}} :=\{u\in L^2(\TT^d; X ):\cA u = 0 \text{ in } \TT^d,\, \overline{u}=F\}.
\end{equation}
Based on \cite{kohn}, we introduce the $\cA$-quasiconvex envelope at fixed volume fraction $\theta\in[0,1]$ as
the function $Q^\cA_\theta W: X\to\RR$ given by
\begin{equation}\label{eq:AQuasiconvexificationAtVolFrac}
  Q^\cA_\theta W(F) := \inf \left\{\int_{\TT^d} \abs{u-\chi}_\cA^2\,dx : u\in\cD^{\cA,\textup{per}}_F,\,\chi\in L^2(\TT^d;\cK) \text{ with }\overline{\chi}=a_\theta\right\},\quad F\in X,
\end{equation}
where $a_\theta$ is defined as in \eqref{eq:atheta} and encodes the volume constraint $\lvert\{\chi=a_1\}\rvert=\theta$.
\end{definition}
In the classical theory of relaxations, that is, for $\cA=\curl$, it is well-known that the minimization \eqref{eq:AFreeEnergyIntroRelaxation}
is equivalent to computing the quasiconvex envelope of the energy density
since the definition of the quasiconvex envelope does not depend on the domain; see \cite[Section~5.1.1.2]{Dacorogna-Book} and \cite[Section 7.1]{rindler}.
Due to the following lemma, this remains valid in the $\cA$-free setting if $\cA$ has constant rank.
\begin{lemma}\label{lem:domaininvarianceInText}
Let $\cA$ be a differential operator as in \eqref{eq:diffopera} satisfying the constant rank property; see \autoref{def:constantrank}.
Let $F\in X$ and $\cK=\{a_0,a_1\}\subset X$. Furthermore, let $Q^\cA W$ be the $\cA$-quasiconvex envelope of $W:=\dist_\cA^2(\cdot,\cK)$; see \eqref{eq:AQuasiconvexification}.
Given a bounded Lipschitz domain $\Omega\subset\RR^d$, let $E_0^\cA(F,\cK;\Omega)$ be given as in \eqref{eq:MinSingPertAfreeEnergy}.
Then, it holds that
\begin{equation}
  E_0^\cA(F,\cK;\Omega) = \abs{\Omega}Q^\cA W(F).
\end{equation}
\end{lemma}
This result essentially follows from the work of Rai\c t\u a \cite{raitapot19}, but for completeness we give a proof in \autoref{sec:AppendixA}.\medbreak

As noted in \cite{kohn}, to compute $Q^\cA W$, it suffices to consider the $\cA$-quasiconvex envelope at fixed volume fraction, which is summarized in the following lemma.
\begin{lemma}\label{lem:QuasiAtFixed}
In the setting of \autoref{def:aquasi}, it holds that
\begin{equation}
    Q^\cA W(F) = \inf_{\theta\in[0,1]}Q^\cA_\theta W(F)\quad\forall F\in X.
\end{equation}
\end{lemma}
\begin{proof}
Given $F\in X$, we compute
\begin{equation}
    \begin{aligned}
            \inf_{\theta\in[0,1]}Q^\cA_\theta W(F) =  \inf \big\{\textstyle\int_{\TT^d} \abs{u-\chi}_\cA^2 \,dx : u\in\cD^{\cA,\textup{per}}_F,\,\chi\in L^2(\TT^d;\cK) \big\}& \\
            =\inf \big\{\textstyle\int_{\TT^d} \dist^2_\cA(u,\cK)\,dx : u\in\cD^{\cA,\textup{per}}_F \big\}& = Q^\cA W(F).
    \end{aligned}
\end{equation}
\end{proof}
In the subsequent analysis, we employ Fourier methods to derive an explicit formula for $Q^\cA_\theta W$ for all $\theta\in[0,1]$.
Since the Fourier transform of periodic functions is concentrated on the integer lattice $\mathbb{Z}^d$, the following definition will prove useful.
\begin{definition}\label{def:rationaldirection}
A direction $\xi \in S^{d-1}$ is said to be \textit{rational} if $\xi = {k}/{\lvert{k}\rvert}$ for some $k \in \mathbb{Z}^d \setminus \{0\}$;
otherwise, it is called \textit{irrational}. We denote the set of rational directions by $S_{\mathbb{Q}}^{d-1}$.
\end{definition}
\begin{rmk}\label{def:rationaldirectionDense}
It is well-known that $S_{\mathbb{Q}}^{d-1}$ is dense in $S^{d-1}$, which follows from density of $\QQ^d$ in $\RR^d$ together with the continuity
of the map $k\mapsto {k}/{\lvert k\rvert}$ for $k\in\RR^d\setminus\{0\}$.
\end{rmk}
With this we are ready to formulate the key result of \autoref{subsec:aquasi}.
\begin{theorem}\label{thm:qtw}
Let $\cA$ be a differential operator as in \eqref{eq:diffopera} satisfying the constant rank property; see \autoref{def:constantrank}.
Given $\cK=\{a_0,a_1\}\subset X$ and $\theta\in[0,1]$, let $Q^\cA_\theta W$ be the $\cA$-quasiconvex envelope of $W=\dist^2_\cA(\cdot,\cK)$ at fixed volume fraction $\theta$;
see \autoref{def:aquasi}.
Let $a_\theta$ be as in \eqref{eq:atheta}. For $a:=a_1-a_0$, let $h_\cA(a)$ and $S_\cA(a)$ be as in \autoref{def:compquant} and \autoref{def:optLaminationDirections}, respectively.
Then, it holds that
\begin{equation}\label{eq:qtw}
Q^\cA_\theta W(F)= \abs{F-a_\theta}_\cA^2 + \theta(1-\theta)h_\cA(a) \quad \forall F\in X .
\end{equation}
If there exists $\xi^*\in S_\cA(a)\cap S^{d-1}_\QQ $ (see \autoref{def:rationaldirection}) then the variational problem \eqref{eq:AQuasiconvexificationAtVolFrac}
admits minimizers $(u,\chi)$ that are simple laminates with lamination direction $\xi^*$.
\end{theorem}

The advantage of restricting our attention to constant rank differential operators is the following.
\begin{proposition}[\cite{fonsecamuller99}]\label{prop:smoothCompProj}
Let $\cA$ be a differential operator as in \eqref{eq:diffopera} satisfying the constant rank property; see \autoref{def:constantrank}.
Let $\ppa$ be given as in \autoref{def:compatibilityProjection}.
Then, it holds that 
\begin{equation}\label{eq:smoothCompProj}
    \ppa\in C^\infty(\RR^d\setminus\{0\};\,\Lin(X)).
\end{equation}
\end{proposition}
A proof of \autoref{prop:smoothCompProj} can be found in \cite[Theorem 4.6]{prosinski}.
A consequence of \autoref{prop:smoothCompProj} is that for any constant rank differential operator $\cA$ as in \eqref{eq:diffopera} and any state $a\in X$,
the set of optimal lamination directions $S_{\cA}(a)$ is nonempty.
With this observation, we now turn to the proof of \autoref{thm:qtw}, which is based on the arguments in \cite[Theorem 3.1]{kohn}.
\begin{proof}[Proof of \autoref{thm:qtw}]
We proceed in two steps: In the first step, we prove that the right-hand side of \eqref{eq:qtw} is a lower bound for $Q^\cA_\theta W(F)$. 
In the second step, we then minimize \eqref{eq:AQuasiconvexificationAtVolFrac} using $\cA$-free simple laminates to show that this lower bound is sharp.
It will prove useful to identify phase arrangements with the indicator function of their $a_1$-phase through the relation
\begin{equation}
    \chi=(1-\chi_1)a_0+\chi_1a_1 \text{ in }\TT^d.
\end{equation}
Under this relation, phase arrangements $\chi\in L^2(\TT^d;\cK)$ with $\overline{\chi}=a_\theta$
are in one-to-one correspondence with indicator functions $\chi_1\in L^2(\TT^d;\{0,1\})$ with $\overline{\chi_1}=\theta$.\smallbreak
\Step{1: Lower bound}
Now, let $F\in X $ and fix a phase arrangement $\chi=(1-\chi_1)a_0+\chi_1a_1\in L^2(\TT^d;\cK)$ with $\overline{\chi}=a_\theta$.
For $u\in L^2(\TT^d;\RR^d)$, taking the Fourier transform of $\cA u=0$ implies that
\begin{equation}\label{eq:cAu}
  u\in \cD^{\cA,\textup{per}}_F \iff \begin{cases}
    \hat{u}(0)=\overline{u}=F, \\
    \hat{u}(\xi) \in [V_\cA(\xi)]^\CC \quad\forall\xi\in\ZZ^d\setminus\{0\},
  \end{cases}
\end{equation}
where $[V_\cA(\xi)]^\CC$ denotes the complexification of $V_\cA(\xi)$, viewed a linear subspace of the complexification $X^\CC$ of $X$.
Note that $X^\CC$ naturally inherits a complex inner product structure, which, by a slight abuse of notation, we also denote by $(\cdot,\cdot)_\cA$.
For $u\in \cD^{\cA,\textup{per}}_F$, we apply Parseval's theorem to compute
\begin{equation}\label{eq:el}
  E_{el}^\cA(u,\chi) =  \int_{\TT^d} \abs{u-\chi}_\cA^2\, dx = \sum_{\xi\in\ZZ^d} \abs{\hat{u}(\xi)-\hat{\chi}(\xi)}_\cA^2
  = \abs{F-a_\theta}_\cA^2 + \sum_{\xi\in\ZZ^d\setminus\{0\}} \abs{\hat{u}(\xi)-\hat{\chi}(\xi)}_\cA^2.
\end{equation} 
In view of \eqref{eq:cAu}, optimizing $\hat{u}(\xi)$ at each frequency yields a minimizer $u_\chi \in\cD^{\cA,\textup{per}}_F$ for \eqref{eq:el} given by
\begin{equation}\label{eq:uchi}
  u_\chi(x):=F +\sum_{\xi\in\ZZ^d\setminus\{0\}}  e^{2\pi i \xi\cdot x} \,\ppaCC(\xi) [\hat{\chi}(\xi)],\quad x\in\TT^d,
\end{equation}
where $\ppaCC(\xi)$ denotes the orthogonal projection onto $[V_\cA(\xi)]^\CC$ in $X^{\CC}$.
Note that for any $\xi\in\ZZ^d\setminus\{0\}$, we have $\ppaCC(\xi)[a+ia']=\ppa(\xi)[a]+i\ppa(\xi)[a']$ for all $a,a'\in X$.\medbreak

Next, we define $f :=(\chi_1-\theta )\in L^2(\TT^d;\{-\theta,1-\theta\})$ to expand $\chi$ about its mean 
\begin{equation}\label{eq:MeanExpansion}
  \chi (x)= a_\theta + f(x)a\quad \forall x\in\TT^d.
\end{equation}
From this, we obtain
\begin{equation}\label{eq:FourierCoeffChi}
  \hat{\chi}(\xi) = \hat{f}(\xi)a \quad \forall\xi\in\ZZ^d\setminus\{0\},
\end{equation}
which, together with \eqref{eq:uchi}, allows us to compute the Fourier coefficients of $u_\chi$:
\begin{equation}\label{eq:uchi2}
  \hat{u}_\chi(\xi) = \ppaCC(\xi) [\hat{\chi}(\xi)] = \ppaCC(\xi) [\hat{f}(\xi)a] =  \hat{f}(\xi) \ppa(\xi) a \quad \forall\xi\in\ZZ^d\setminus\{0\}.
\end{equation}
Plugging this into \eqref{eq:el} and using \eqref{eq:FourierCoeffChi} yields the formula
\begin{equation}\label{eq:perenchi}
  \min_{u\in\cD^{\cA,\textup{per}}_F}E_{el}^{\cA}(u,\chi) = E_{el}^\cA(u_\chi,\chi)
  = \abs{F-a_\theta}_\cA^2 + \sum_{\xi\in\ZZ^d\setminus\{0\}} \lvert \hat{f}(\xi)\rvert^2\lvert a -\ppa(\xi) a \rvert _\cA^2.
\end{equation}
Since $\ppa$ is zero-homogeneous, see \eqref{eq:zerohom}, we find
\begin{equation}
  \abs{a -\ppa(\xi) a}_\cA \geq h_\cA(a)\quad \forall\xi\in\RR^d\setminus\{0\}.
\end{equation}
Using Parseval's theorem and $\hat{f}(0)=\overline{\chi}_1-\theta = 0$, it follows that
\begin{equation}\label{eq:fhatsquared}
  \sum_{\xi\in\ZZ^d\setminus\{0\}} \big| \hat{f}(\xi)\big|^2 = \int_{\TT^d} \abs{f}^2\,dx = \abs{\{\chi=a_0\}} \theta^2 + \abs{\{\chi=a_1\}}(1-\theta)^2= \theta(1-\theta),
\end{equation}
which yields the estimate 
\begin{equation}\label{eq:upboundqatw}
\min_{u\in\cD^{\cA,\textup{per}}_F}E_{el}^{\cA}(u,\chi)\geq \abs{F-a_\theta}_\cA^2 + \theta(1-\theta) h_\cA(a).
\end{equation}
Minimizing over phase arrangements $\chi\in L^2(\TT^d,\cK)$ with $\overline{\chi}=a_\theta$, we find
\begin{equation}
  Q^\cA_\theta W(F) \geq \abs{F-a_\theta}_\cA^2 + \theta(1-\theta) h_\cA(a).
\end{equation} 
To establish the formula for the $\cA$-quasiconvex envelope \eqref{eq:qtw}, it remains to show that this estimate is sharp.\smallbreak
\Step{2: Upper bound}
For a rational direction $\zeta\in S^{d-1}_\QQ$, there exists a unique $k\in\ZZ^d$ with minimal length $|k|$ such that $\zeta = {k}/{\abs{k}}$.
Utilizing the indicator function $h=\mathbbm{1}_{[0,\theta]+\ZZ}:\RR\rightarrow\{0,1\}$ as profile, we 
associate to $\zeta$ the phase arrangement $\chi_{\zeta}$ given by
\begin{equation}\label{eq:chixi}
  \chi_{\zeta}(x):=(1-\chi_{1,\zeta}(x))a_0+\chi_{1,\zeta}(x)a_1,\quad x\in \TT^d,
\end{equation}
where $\chi_{1,\zeta}(x) := h(k \cdot x)$ for $x\in\TT^d$. Note that $\chi_\zeta$ is a simple laminate with lamination direction~$\zeta$.\medbreak
By homogenization, it is verified that $\overline{\chi}_{1,\zeta}=\theta$. Since $k\in\ZZ^d$, the Fourier transform of $f_{\zeta}=\chi_{1,\zeta}-\theta$ is
supported on a line; that is, $\supp \hat{f}_{\zeta} \subset \spn\{\zeta\}$.
With formula \eqref{eq:perenchi} and zero-homogeneity of $\ppa$, we derive 
\begin{equation}\label{eq:perenchi2}
   \min_{u\in\cD^{\cA,\textup{per}}_F}E_{el}^{\cA}(u, \chi_{\zeta}) =  \abs{F-a_\theta}_\cA^2 + \theta(1-\theta) \lvert a -\ppa(\zeta) a\rvert _\cA^2\quad\forall \zeta\in S^{d-1}_\QQ.
\end{equation}
Now, first suppose there exists a rational direction $\xi^*\in S_\cA(a)\cap S^{d-1}_\QQ $.
By \autoref{def:optLaminationDirections} of $S_\cA(a)$, it holds that $\lvert a -\ppa(\xi^*) a\rvert _\cA^2=h_\cA(a)$.
Hence, the associated phase arrangement $\chi_{\xi^*}$ saturates the inequality \eqref{eq:upboundqatw}, proving that
\begin{equation}\label{eq:enchistar}
  Q^\cA_\theta W(F) =  \min_{u\in\cD^{\cA,\textup{per}}_F}E_{el}^{\cA}(u, \chi_{\xi^*})  = \abs{F-a_\theta}_\cA^2 + \theta(1-\theta) h_\cA(a).
\end{equation}
Due to \eqref{eq:perenchi}, we have $E_{el}^{\cA}(u_{\chi_{\xi^*}},\chi_{\xi^*})=Q^\cA_\theta W(F)$ for $u_{\chi_{\xi^*}}\in\cD^{\cA,\textup{per}}_F$
given as in \eqref{eq:uchi}. This shows that $(u_{\chi_{\xi^*}},\chi_{\xi^*})$ is a minimizer of the variational problem \eqref{eq:AQuasiconvexificationAtVolFrac}.
By \eqref{eq:uchi}, \eqref{eq:uchi2} and \eqref{eq:chixi}, it follows that
\begin{equation}
    u_{\chi_{\xi^*}} = F + (\chi_{1,\xi^*}(x)-\theta)\ppa(\xi^*)a\quad \forall x\in \TT^d.
\end{equation}
In particular, the map $u_{\chi_{\xi^*}}$ is a simple laminate of the two matrices
\begin{equation}\label{eq:TildeWellsInQuasiconvexificationPf}
    \ta_0 := F -\theta \ppa(\xi^*)a  \hspace{2cm}   \ta_1 := F +(1-\theta) \ppa(\xi^*)a
\end{equation}
with lamination direction $\xi^*$.\medbreak
If there are no rational directions in $S_\cA(a)$, we must rely on the continuity of $\ppa$; see \autoref{prop:smoothCompProj}.
This ensures that there exists $\xi^* \in S_\cA(a)\setminus S^{d-1}_\QQ$ with
\begin{equation}\label{eq:optimalIrrational}
  |a -\ppa(\xi^*) a|_\cA^2=h_\cA(a).
\end{equation}
However, as $\xi^*$ is irrational, it is not evident how to extend the definition \eqref{eq:chixi} for $\zeta=\xi^*$ such that $\chi_{\xi^*}$ is $\TT^d$-periodic 
and its Fourier transform is supported on the line $\spn\{\xi^*\}$.
We therefore approximate $\xi^*$ using the density of rational directions in the sphere.
Let $(\xi_l)_l\subset S^{d-1}_\QQ$ be a sequence with $\xi_l \rightarrow \xi^*$ as $l\rightarrow\infty$.
Again using the continuity of $\ppa$, we obtain
\begin{equation}
  \lim_{l\rightarrow\infty}|a -\ppa(\xi_l) a|_\cA^2=|a -\ppa(\xi^*) a|_\cA^2=h_\cA(a),
\end{equation}
which, together with \eqref{eq:perenchi2}, allows us to infer that
\begin{equation}\label{eq:chil}
  Q^\cA_\theta W(F) = \lim_{l\rightarrow\infty} \min_{u\in\cD^{\cA,\textup{per}}_F} E_{el}^{\cA}(u, \chi_{\xi_l}) = \abs{F-a_\theta}_\cA^2 + \theta(1-\theta) h_\cA(a)
\end{equation}
and concludes the proof.
\end{proof}

With \autoref{thm:qtw} in hand, we are ready to prove the main result on the $\cA$-free relaxation.
\begin{proof}[Proof of \autoref{thm:qwdom}]
Applying \autoref{lem:domaininvarianceInText}, \autoref{lem:QuasiAtFixed} and \autoref{thm:qtw}, it follows that
\begin{equation}
    E_0^\cA(F,\cK;\Omega) = \abs{\Omega}Q^\cA W(F) =\abs{\Omega} \min_{\theta\in[0,1]} \Big(\abs{F-a_\theta}_\cA^2 + \theta(1-\theta)h_\cA(a)\Big).
\end{equation}
This completes the proof.
\end{proof}
\subsection{The optimal volume fraction and compatible approximations}\label{subsec:compApprox}
In this section, we prove \autoref{prop:optvolfrac} and discuss the geometric interpretation of the optimal volume fraction.
We then conclude this section by introducing compatible approximations in \autoref{def:compatibleApprox},
which will play an important role in our upper bound constructions in \autoref{sec:UpperBoundGradDiv} and
\autoref{sec:UpperBoundCC}.\medbreak

We will now show that the optimal volume fraction is unique if $\cA$ has constant rank and spanning wave cone.
\begin{proof}[Proof of \autoref{prop:optvolfrac}]
First, let us define $H:[0,1]\to \RR$ by
\begin{equation}\label{eq:quasiconvexificationPfOptVolFrac}
    H(\theta):=Q^\cA_\theta W(F)=\abs{F-a_\theta}_\cA^2 +\theta(1-\theta)h_\cA(a),\quad \theta\in[0,1].
\end{equation}
Using \eqref{eq:compquanteq}, we observe that $H$ is quadratic with leading coefficient
\begin{equation}\label{eq:leadingCoefficient}
  \abs{a}_\cA^2-h_\cA(a)=g_\cA(a) \geq 0.
\end{equation}
To see this, note that 
\begin{equation}\label{eq:PythagorasProjectionF}
  \abs{F-a_\theta}_\cA^2 = \abs{F-F_*}_\cA^2 + \abs{F_*-a_\theta}_\cA^2 = \abs{F-F_*}_\cA^2 + (\theta-\theta^*)^2\abs{a}^2_\cA \quad\forall\theta\in[0,1],
\end{equation}
where $F_*=(1-\theta^*)a_0+\theta^* a_1$ is the orthogonal projection of $F$ onto the line spanned by $a_0$ and $a_1$ in $X$.
Here, the orthogonality $(F-F_*,a)_\cA=0$ uniquely determines $\theta^*\in\RR$ to be
\begin{equation}
  \theta^*=\theta^*_\cA(F,\cK):= \argmin_{\theta\in\RR} \abs{F-a_\theta}_\cA^2=\frac{(F-a_0,a)_\cA}{\lvert a\rvert^2_\cA} .
\end{equation}
By \autoref{def:compquant}, there holds $g_\cA(a)=0$ if and only if $a\perp_\cA \spn\Lambda_\cA$.
However, as the wave cone is spanning and $a=a_1-a_0\ne 0$, it follows that $g_{\cA}(a)>0$. 
In particular, $H$ is strictly convex and has a unique minimizer.
\end{proof}

\begin{rmk}[Geometric interpretation of the optimal volume fraction]\label{rmk:GeometryOptVolFrac}
Since the optimal volume fraction plays a central role in our main results (\autoref{thm:incompgrad} and \autoref{thm:incompsyma}),
we next discuss its geometric interpretation, illustrated in \autoref{fig:optvolfrac}.\smallbreak
Let $\cA$ be a differential operator as in \eqref{eq:diffopera} that has constant rank and spanning wave cone.
Let $F\in X$ and $\cK=\{a_0,a_1\}\subset X$. As usual, we use the notation $a:=a_1-a_0$.
From \eqref{eq:optvolfrac}, we observe that $\ttheta_\cA(F,\cK)$ is invariant under translations of $F$ in directions perpendicular to $a$:
\begin{equation}
  \ttheta_\cA(F+v,\cK)=\ttheta_\cA(F,\cK) \quad\forall v\in X :v\perp_\cA a =0.
\end{equation}
This is immediate from \eqref{eq:quasiconvexificationPfOptVolFrac} and \eqref{eq:PythagorasProjectionF}.
Although these translations do not influence the optimal volume fraction, they do change
the minimal energy $E_0^\cA(F,\cK)$ as follows from \autoref{thm:qwdom}.
Defining $\Raa:=\Raa(a):=\tfrac{1}{2}{h_\cA(a)}/{\abs{a}_\cA^2}\in[0,\tfrac{1}{2})$, a quick computation (see \cite[Theorem 3.5]{kohn}) shows that
\begin{equation}
\begin{aligned}
\ttheta_\cA(F,\cK)=0&\iff \theta^*_\cA(F,\cK) \leq \Raa,\\
\ttheta_\cA(F,\cK)=1&\iff \theta^*_\cA(F,\cK) \geq 1-\Raa.
\end{aligned}
\end{equation}
With this, we see that the optimal volume fraction partitions $X$ into three parts: the two half spaces
\begin{equation}
\begin{aligned}
\{F\in X :\ttheta_\cA(F,\cK)=0\}&= \{F\in X : (F-a_{\Raa}, a)_\cA\leq 0\},\\
\{F\in X :\ttheta_\cA(F,\cK)=1\}& = \{F\in X : (F-a_{1-\Raa},a)_\cA\geq 0\}
\end{aligned}
\end{equation}
containing boundary data resulting in a pure phase, and the slab separating them
\begin{equation}
  \{F\in X :\ttheta_\cA(F,\cK)\in(0,1)\} = \{F\in X : \lvert(F-a_{1/2}, a)_\cA\rvert< (\tfrac{1}{2}-\Raa)\abs{a}_\cA^2 \}
\end{equation}
containing boundary data that causes mixing of the phases.
Note that both $h_\cA(a)$ and $\Raa(a)$ vanish if and only if the wells are compatible.
In this case, the slab is given by $\{a_\theta+v\in X :\theta\in(0,1)\text{ and }v\perp_\cA a\}$ and 
the optimal volume fraction determines the projection of $F$ onto the line segment $\overline{a_0a_1}$:
\begin{equation}
    \ttheta_{\cA}(F,\cK)=\argmin_{\theta\in[0,1]}\abs{F-a_\theta}_\cA^2.
\end{equation}
If, in addition, the boundary data $F$ satisfies the compatibility condition $F=(1-\lambda)a_0+\lambda a_1$ for some $\lambda\in[0,1]$ then 
the optimal volume fraction can be read off from $F$ with $\ttheta_{\cA}(F,\cK)=\lambda$.
For this reason, the notion of the optimal volume fraction introduced in this work is primarily of interest in the incompatible setting.
\end{rmk}

We conclude this section by briefly reviewing the arguments used in the proof of
\autoref{thm:qwdom}, which then leads us to introduce the notion of compatible approximations (\autoref{def:compatibleApprox}).

\begin{rmk}[Compatible approximation]\label{rmk:compatibleApprox}
Let $\cA$ be a differential operator as in \eqref{eq:diffopera} that has constant rank and spanning wave cone.
Let $\Omega\subset\RR^d$ be a bounded Lipschitz domain.
Given $F\in X$ and $\cK=\{a_0,a_1\}\subset X$ with $a:=a_1-a_0$, we again consider the minimization of $E^{\cA}_{el}(u;\cK)$ among $u\in\cD^\cA_F(\Omega)$; see \eqref{eq:AFreeEnergyIntroRelaxation}.
The definition of the optimal volume fraction $\ttheta_\cA=\ttheta_\cA(F,\cK)$ in \eqref{eq:optvolfrac} together
with \autoref{lem:domaininvarianceInText}, \autoref{lem:QuasiAtFixed} and \autoref{thm:qtw} imply that
\begin{equation}
  E_{el}(F,\cK)=\abs{\Omega}Q^{\cA}_{\ttheta_{\cA}}W(F)= \abs{\Omega} \Big(\big\lvert F-a_{\ttheta_{\cA}}\big\rvert_\cA^2 +\ttheta_{\cA}(1-\ttheta_{\cA})h_\cA(a)\Big).
\end{equation}
In particular, the proofs of \autoref{lem:domaininvarianceInText} and \autoref{thm:qtw} suggest that we can minimize $E^{\cA}_{el}(u;\cK)$
by picking any optimal lamination direction $\xi^*\in S_{\cA}(a)$ and considering simple laminates $u\in\cD^\cA_F(\Omega)$ of the two compatible states 
\begin{equation}\label{eq:TildeWellsInQuasiconvexificationPfcA}
\ta_0 := F -\ttheta_{\cA}(F,\cK) \ppa(\xi^*)a  \hspace{2cm}   \ta_1 := F +(1-\ttheta_{\cA}(F,\cK)) \ppa(\xi^*)a
\end{equation}
with lamination direction $\xi^*$ and volume proportion $\ttheta_\cA(F,\cK)$ for the $\ta_1$-phase.
When we investigate singularly perturbed two-well energies with incompatible data in \autoref{sec:UpperBoundGradDiv} and \autoref{sec:UpperBoundCC},
we will see that branching constructions based on these simple laminates allow us
to prove the upper scaling bounds of the main results \autoref{thm:incompgrad} and \autoref{thm:incompsyma}.
\medbreak

It is worth noting that if the data $(F,\cK)$ are compatible with $F=(1-\lambda)a_0+\lambda a_1$ for some $\lambda\in[0,1]$,
it follows that $\ta_0=a_0$ and $\ta_1=a_1$ because $\ttheta_\cA(F,\cK)=\lambda$ (see \autoref{rmk:GeometryOptVolFrac}) and $\ppa(\xi^*)a=a$ for all $\xi^*\in S_{\cA}(a)$.
In contrast, when the data are incompatible, any choice of $\xi^*\in S_{\cA}(a)$ leads, via \eqref{eq:TildeWellsInQuasiconvexificationPfcA}, to states $\ta_0$ and $\ta_1$
that differ from $a_0$ and $a_1$, as illustrated in \autoref{fig:tildewells}.
In this way, the approach compa\-tibilizes the data $(F,\cK)$ giving rise to the compatible data $(F,\tcK)$ with $\tcK:=\{\ta_0,\ta_1\}$ satisfying
\begin{equation}\label{eq:CompatibleApproxAverage}
F=(1-\ttheta_{\cA}(F,\cK))\ta_0+\ttheta_{\cA}(F,\cK)\ta_1 \text{ and }\ta_1-\ta_0=\ppa(\xi^*)a\in\Lambda_\cA.
\end{equation}
In addition, since $\xi^*\in S_{\cA}(a)$, it follows from \eqref{eq:OptimalCompatibleApprox2} that 
$\ta_1-\ta_0$ is a wave cone approximation of the difference of the wells $a=a_1-a_0$ in the sense that
\begin{equation}\label{eq:OptimalCompatibleApprox}
  \dist^2(a, \Lambda_{\cA}) = \min_{\xi\in\Sdm} \abs{a-\ppa(\xi)a}^2 =\abs{a-\ppa(\xi^*)a}^2.
\end{equation}
\end{rmk}

\begin{figure}
\centering
\tikzset{every picture/.style={line width=0.75pt}}      
\begin{tikzpicture}[x=0.75pt,y=0.75pt,yscale=-1,xscale=1]
\draw  [dash pattern={on 4.5pt off 4.5pt}]  (290.1,29.8) -- (261.1,117.8) ;
\draw [color={rgb, 255:red, 0; green, 0; blue, 0 }  ,draw opacity=1 ]   (193.89,121.33) -- (444.89,121.33) ;
\draw  [color={rgb, 255:red, 0; green, 174; blue, 179 }  ,draw opacity=1 ][fill={rgb, 255:red, 0; green, 174; blue, 179 }  ,fill opacity=1 ] (187.93,121.33) .. controls (187.93,118.04) and (190.6,115.37) .. (193.89,115.37) .. controls (197.18,115.37) and (199.85,118.04) .. (199.85,121.33) .. controls (199.85,124.62) and (197.18,127.29) .. (193.89,127.29) .. controls (190.6,127.29) and (187.93,124.62) .. (187.93,121.33) -- cycle ;
\draw  [color={rgb, 255:red, 0; green, 174; blue, 179 }  ,draw opacity=1 ][fill={rgb, 255:red, 0; green, 174; blue, 179 }  ,fill opacity=1 ] (438.93,121.33) .. controls (438.93,118.04) and (441.6,115.37) .. (444.89,115.37) .. controls (448.18,115.37) and (450.85,118.04) .. (450.85,121.33) .. controls (450.85,124.62) and (448.18,127.29) .. (444.89,127.29) .. controls (441.6,127.29) and (438.93,124.62) .. (438.93,121.33) -- cycle ;
\draw  [color={rgb, 255:red, 0; green, 174; blue, 179 }  ,draw opacity=1 ][fill={rgb, 255:red, 0; green, 174; blue, 179 }  ,fill opacity=1 ] (253.93,121.33) .. controls (253.93,118.04) and (256.6,115.37) .. (259.89,115.37) .. controls (263.18,115.37) and (265.85,118.04) .. (265.85,121.33) .. controls (265.85,124.62) and (263.18,127.29) .. (259.89,127.29) .. controls (256.6,127.29) and (253.93,124.62) .. (253.93,121.33) -- cycle ;
\draw [color={rgb, 255:red, 0; green, 0; blue, 0 }  ,draw opacity=1 ]   (218.73,44.08) -- (467.96,14.36) ;
\draw  [color={rgb, 255:red, 0; green, 174; blue, 179 }  ,draw opacity=1 ][fill={rgb, 255:red, 0; green, 174; blue, 179 }  ,fill opacity=1 ] (212.81,44.78) .. controls (212.42,41.51) and (214.75,38.55) .. (218.02,38.16) .. controls (221.29,37.77) and (224.25,40.1) .. (224.64,43.37) .. controls (225.03,46.64) and (222.7,49.6) .. (219.43,49.99) .. controls (216.17,50.38) and (213.2,48.05) .. (212.81,44.78) -- cycle ;
\draw  [color={rgb, 255:red, 0; green, 174; blue, 179 }  ,draw opacity=1 ][fill={rgb, 255:red, 0; green, 174; blue, 179 }  ,fill opacity=1 ] (462.05,15.06) .. controls (461.66,11.8) and (463.99,8.83) .. (467.26,8.44) .. controls (470.52,8.05) and (473.49,10.38) .. (473.88,13.65) .. controls (474.27,16.92) and (471.94,19.88) .. (468.67,20.27) .. controls (465.4,20.66) and (462.44,18.33) .. (462.05,15.06) -- cycle ;
\draw  [color={rgb, 255:red, 0; green, 174; blue, 179 }  ,draw opacity=1 ][fill={rgb, 255:red, 0; green, 174; blue, 179 }  ,fill opacity=1 ] (283.93,35.19) .. controls (283.93,32.32) and (286.25,30) .. (289.12,30) .. controls (291.98,30) and (294.3,32.32) .. (294.3,35.19) .. controls (294.3,38.05) and (291.98,40.37) .. (289.12,40.37) .. controls (286.25,40.37) and (283.93,38.05) .. (283.93,35.19) -- cycle ;

\draw (293.89,10) node [anchor=north west][inner sep=0.75pt]  [font=\Large] [align=left] {$F$};
\draw (190,20) node [anchor=north west][inner sep=0.75pt]  [font=\Large,rotate=-359] [align=left] {$\ta_0$};
\draw (482,9) node [anchor=north west][inner sep=0.75pt]  [font=\Large,rotate=-359] [align=left] {$\ta_1$};
\draw (241.89,132) node [anchor=north west][inner sep=0.75pt]  [font=\Large] [align=left] {$a_{\ttheta_{\cA}(F,\,\cK)}$};
\draw (172,132) node [anchor=north west][inner sep=0.75pt]  [font=\Large] [align=left] {$a_0$};
\draw (448,132) node [anchor=north west][inner sep=0.75pt]  [font=\Large] [align=left] {$a_1$};
\end{tikzpicture}
\caption{The compatible approximation $(F,\tcK)$ of the data $(F,\cK)$.}
\label{fig:tildewells}
\end{figure}
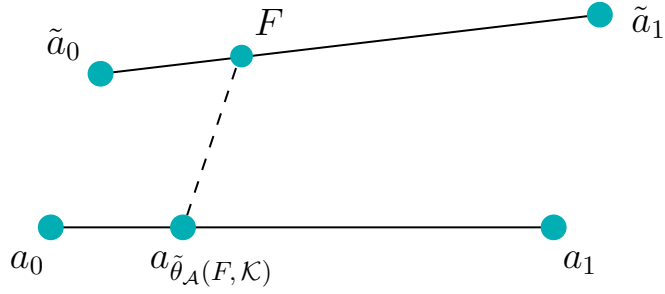

Based on \autoref{rmk:compatibleApprox}, we introduce the following notion.
\begin{definition}[Compatible approximation]\label{def:compatibleApprox}
Let $\cA$ be a differential operator as in \eqref{eq:diffopera}.
Suppose that $\cA$ has constant rank and spanning wave cone; see \autoref{def:constantrank} and \autoref{def:SpanWaveCone}.
Let $\ppa$ be the compatibility projection; see \autoref{def:compatibilityProjection}.
Given $F\in X$ and $\cK=\{a_0,a_1\}\subset X$, let $\ttheta_{\cA}(F,\cK)$ be given as in \autoref{prop:optvolfrac}.
For $a:=a_1-a_0$, let $\xi^*\in S_\cA(a)$; see \autoref{def:optLaminationDirections}. \smallbreak
We call the data $(F,\tcK)$ a compatible approximation of $(F,\cK)$ if $\tcK=\{\ta_0,\ta_1\}\subset X $ with $\ta_0$ and $\ta_1$ given by \eqref{eq:TildeWellsInQuasiconvexificationPfcA}.
\end{definition}

\section{Lower bounds for singularly perturbed two-well energies}\label{sec:LowerBounds}
This section is arranged into two parts.
In \autoref{subsec:GeneralLowerScaling}, a rather general result is presented, which provides a tool to deduce lower scaling bounds of $\cA$-free two-well energies.
Subsequently, we specialize to the model settings $\cA\in\{\curl,\Div,\ccurl\}$ in \autoref{subsec:ApplicationOfGeneralLowerScaling}, making the introduced notions explicit while deriving the lower bounds asserted in \autoref{thm:incompgrad} and \autoref{thm:incompsyma}.

\subsection{General result on incompatible \texorpdfstring{$\cA$}{A}-free two-well energies}\label{subsec:GeneralLowerScaling}
We systematically derive lower scaling bounds of zeroth-order-corrected singularly perturbed incompatible $\cA$-free two-well energies.
We will see that the lower bound critically depends on the maximal vanishing order (on the unit sphere) of a specific function, which is introduced in \autoref{def:LcAa}.
Let us therefore recall the definition of the maximal vanishing order.
\begin{definition}[Maximal vanishing order on the unit sphere, {\cite[Definition 1.3]{ruland2}}]\label{def:maxvanishorderdef}
Let $d\in\NN$. Let $p\in C^\infty(S^{d-1};\RR)$ be a nonnegative function. Let $S$ denote the zero set of $p$. We define the maximal vanishing order $L[p]$ of $p$ as
\begin{equation}\label{eq:defMaxVanOrder}
    L[p]:=\min\left\{l\in\NN: \inf_{\xi\in S^{d-1}\setminus S}\frac{p(\xi)}{\dist^{2l}(\xi,S)}>0\right\}
\end{equation}
if $S\subsetneq S^{d-1}$ and the set on the right-hand side of \eqref{eq:defMaxVanOrder} is nonempty. 
\end{definition}
We focus primarily on the maximal vanishing order of the following function.
\begin{definition}\label{def:LcAa}
Let $\cA$ be a differential operator as in \eqref{eq:diffopera} satisfying the constant rank property; see \eqref{def:constantrank}.
Let $a\in X $. We define the nonegative function $p_{\cA,a}\in C^\infty(S^{d-1})$ by setting
\begin{equation}
  p_{\cA,a}(\xi):=\abs{a-\ppa(\xi)a}_\cA^2-h_\cA(a),\quad\xi\in S^{d-1}.
\end{equation}
If the maximal vanishing order of $p_{\cA,a}$ exists, we denote it by $L_\cA(a):=L[p_{\cA,a}]$.
\end{definition}
\begin{rmk}\label{rmk:ZeroOfpAndOtherFormOfp}
Using \eqref{eq:compquanteqNotExtremal} and \eqref{eq:compquanteq}, we can write 
\begin{equation}\label{eq:pcAaForumla2}
  p_{\cA,a}(\xi) = g_\cA(a)-\abs{\ppa(\xi)a}_\cA^2\quad\forall\xi\in S^{d-1}.
\end{equation}
Hence, the zero set of $p_{\cA,a}$ is given by $S_\cA(a)$; see \eqref{eq:compquant} and \eqref{eq:optLaminationDirections}.
\end{rmk}
We now come to the main result of \autoref{subsec:GeneralLowerScaling}.
\begin{breaktheorem}[Lower scaling bounds for $\cA$-free two-well energies]\label{thm:incompa}
Let $\cA$ be a differential operator as in \eqref{eq:diffopera}.
Suppose that $\cA$ has constant rank and spanning wave cone; see \autoref{def:constantrank} and \autoref{def:SpanWaveCone}.
Denote by $\ppa$ the compatibility projection; see \autoref{def:compatibilityProjection}.
Let $\cK=\{a_0,a_1\} \subset  X $ and $F\in X $.
For $a:=a_1-a_0$, let $S_\cA(a)$ be as in \autoref{def:optLaminationDirections}.
Given a bounded Lipschitz domain $\Omega\subset \RR^d$ with $d\geq2$, let $E^\cA_\eps(F,\cK)$ be as in \eqref{eq:MinSingPertAfreeEnergy}.
Then, the optimal volume fraction $\ttheta_{\cA}=\ttheta_{\cA}(F,\cK)\in[0,1]$ (see \autoref{prop:optvolfrac}) determines the energy scaling in the following way:
\begin{enumerate}
\item(Pure phase) If $\ttheta_\cA \in \{0,1\} $, it holds that
\begin{equation}\label{eq:zeroscaling2}
    E^\cA_\eps(F,\cK)-E^\cA_0(F,\cK) = 0 \quad \forall \eps> 0,
\end{equation}
as is observed by considering the minimizer $(u,\chi)$ of \eqref{eq:MinSingPertAfreeEnergy} given by the constant maps $u = F$ and $\chi = (1-\ttheta_{\cA})a_0 + \ttheta_{\cA} a_1$.

\item(Phase mixing) If $\ttheta_\cA\in(0,1)$, assume that $S_\cA(a)$ is contained in a finite union of linear subspaces of $\RR^d$ of dimension at most $d-1$.
Suppose that $L=L_\cA(a)\in\NN$; see \autoref{def:LcAa}.\\
Then, there exist $c=c(a,\cA,\Omega,d,\ttheta_\cA)>0$ and $\eps_0= \eps_0(a,\cA,\Omega,d, \ttheta_\cA)>0$
such that
\begin{equation}\label{eq:afree23scaling}
    E^\cA_\eps(F,\cK)-E^\cA_0(F,\cK)
    \geq c\eps^{\nicefrac{2L}{(2L+1)}} \quad \forall \eps\in(0,\eps_0).
\end{equation}
\end{enumerate}
\end{breaktheorem}

The proof of \autoref{thm:incompa} is based on the following result.
\begin{proposition}[{\cite[Proposition 3.3]{ruland2}}]\label{prop:lowerscaling33}
  Let $d, L\in\NN$ with $d\geq 2$. Given a bounded Lipschitz domain $\Omega\subset\RR^d$, let $W\subset\RR^d,\, W\ne \{0\}$ be a union of finitely many 
  linear subspaces of dimension at most $d-1$. Let $\theta \in[0,1]$. For $f\in BV(\Omega; \{-\theta, 1-\theta, 0\})$ with $f\in\{-\theta,  1-\theta\}$ in $\Omega$
  and $f=0$ outside $\Omega$, we consider the energies
  \begin{equation}\label{eq:proplowerscaling33energies}
    \tilde{E}_{el}(f):= \int_{\RR^d}\big\lvert{\hat{f}(\xi)}\big\rvert^2 \frac{\dist^{2L}(\xi,W)}{\abs{\xi}^{2L}}\, d\xi,\qquad \tilde{E}_{surf}(f):=\norm{\nabla f}_{TV(\Omega)}.
  \end{equation}
  Then, there exists $c=c(\Omega,d,L,W)>0$ such that
  \begin{equation}\label{eq:proplowerscaling33lowerbound}
    \tilde{E}_{el}(f)+\eps\tilde{E}_{surf}(f)+\eps\per(\Omega) \geq c \norm{f}_{L^2(\Omega)}^2 \eps^{\nicefrac{2L}{(2L+1)}}\quad\forall\eps\in(0,1).
  \end{equation}
\end{proposition} 
Equipped with \autoref{prop:lowerscaling33}, we prove the main result of this section.
\begin{proof}[Proof of \autoref{thm:incompa}]
The proof is organized into two steps.
\Step{1: Pure phase}
Suppose that $\ttheta_\cA\in\{0,1\}$.
Consider the constant maps $u = F$ and $\chi = (1-\ttheta_{\cA})a_0 + \ttheta_{\cA} a_1$.
Recalling \autoref{thm:qwdom} and \autoref{prop:optvolfrac}, we find
\begin{equation}
  E^\cA_\eps(F,\cK)\leq E^\cA_\eps(u,\chi)=E^\cA_0(u,\chi) = \abs{\Omega}\big\lvert F-a_{\ttheta_\cA} \big\rvert _\cA^2=E^\cA_0(F,\cK)\quad\forall\eps\geq0.
\end{equation}
For all $\eps\geq0$, the reverse inequality $E^\cA_\eps(F,\cK)\geq E^\cA_0(F,\cK)$ is immediate from the definitions of the energies.
\Step{2: Phase mixing}
Suppose that $\ttheta_\cA\in(0,1)$.
Let $\chi\in BV(\Omega; \cK)$.
Throughout this proof, we will use the notation
\begin{equation}
  E^\cA_{el}(\chi;F):=\inf_{u\in\cD_F^\cA(\Omega)}E^\cA_{el}(u,\chi)\text{ and } E^\cA_{\eps}(\chi;F):=\inf_{u\in\cD_F^\cA(\Omega)}E^\cA_{\eps}(u,\chi) \text{ for }\eps\geq0.
\end{equation}
Now, let $u\in\cD_F^\cA(\Omega)$. By \autoref{lem:raitamean},
it holds that $\fint_\Omega u\,dx = F$. In addition, there exists $\theta\in[0,1]$ such that
$\fint_\Omega \chi \,dx= a_\theta $.
Therefore, using that $\int_\Omega \lvert v-\overline{v}\rvert^2\, dx=\int_\Omega \lvert v\rvert^2\, dx - \abs{\Omega}{\overline{v}}^2$ for all $v\in L^2(\Omega)$
with $\overline{v}=\fint_\Omega v\,dx$, it follows by Parseval's theorem that
\begin{equation}\label{eq:prekeyestimate}
  \begin{aligned}
    E^\cA_{el}(u,\chi)&= \int_\Omega\abs{ u -\chi}_\cA^2\, dx= \int_\Omega\abs{(u-F) -(\chi-a_\theta)}_\cA^2\, dx + \abs{\Omega}\abs{F-a_\theta}_\cA^2\\
    &= \int_{\RR^d}\abs{w -\tilde{\chi}}_\cA^2\, dx + \abs{\Omega}\abs{F-a_\theta}_\cA^2
    = \int_{\RR^d}\big\lvert\widehat{w} -\widehat{\tilde{\chi}}\big\rvert _\cA^2\, d\xi + \abs{\Omega}\abs{F-a_\theta}_\cA^2,
  \end{aligned}
\end{equation}
where, for $x\in\RR^d$, we define the functions
\begin{equation}\label{eq:intermediate1f}
  w(x):=u(x)-F,\qquad \tilde{\chi}(x):=\chi-a_\theta =  f(x)a, \qquad f(x):=
  \begin{cases}
    -\theta&\chi(x)=a_0,\\
    1-\theta &\chi(x)=a_1,\\
    0 &x\in\RR^d\setminus\Omega.
  \end{cases}
\end{equation}
Note that under the relation $w=u-F$, admissible maps $u\in\cD_F^\cA(\Omega)$ are in one-to-one correspondence with $w\in\cD_0^\cA(\Omega)$.
Relaxing the boundary constraint from $w\in \cD^\cA_0(\Omega)$ to $w\in L^2(\RR^d; X )$ with $\cA w=0$ and then minimizing $\hat{w}(\xi)$
at each frequency in \eqref{eq:prekeyestimate} yields
\begin{equation}\label{eq:softboundaryconstraint}
\begin{aligned}
    E^\cA_{el}(\chi; F) -  \abs{\Omega}\abs{F-a_\theta}_\cA^2 &\geq \int_{\RR^d}\abs{\widehat{\tilde{\chi}}_\cA(\xi)
    - \ppaCC(\xi)[\widehat{\tilde{\chi}}(\xi)] }_\cA^2\, d\xi
    = \int_{\RR^d} \lvert\hat{f}(\xi)\rvert^2 \abs{a - \ppa(\xi)a}_\cA^2\, d\xi\\
    &= \abs{\Omega}\theta(1-\theta)h_\cA(a)
    + \int_{\RR^d}\lvert\hat{f}(\xi)\rvert^2 \left( \abs{a - \ppa(\xi)a}_\cA^2 -h_\cA(a) \right) d\xi,
\end{aligned}
\end{equation}
where we have omitted some details that were elaborated on in the proof of \autoref{thm:qtw};
see, in particular, \eqref{eq:cAu}, \eqref{eq:uchi}, \eqref{eq:uchi2} and \eqref{eq:fhatsquared}.
Comparing this to the formula for $Q^\cA_\theta W$ in \autoref{thm:qtw}, we find
\begin{equation}\label{eq:keyestimate}
  E^\cA_{el}(\chi; F) -\abs{\Omega} Q^\cA_\theta W(F) \geq
  \int_{\RR^d}\lvert\hat{f}(\xi)\rvert^2 \left( \abs{a - \ppa(\xi)a}_\cA^2 -h_\cA(a) \right)\, d\xi.
\end{equation}
Now, let $L=L_\cA(a)\in\NN$ be the maximal vanishing order of $p_{\cA,a}$.
As pointed out in \autoref{rmk:ZeroOfpAndOtherFormOfp}, the zero set of $p$ is $S_\cA(a)$.
Therefore, the definition of the maximal vanishing order implies that there exists $c_L=c_L(a,\cA)>0$ such that
\begin{equation}
\abs{a - \ppa(\xi)a}_\cA^2 -h_\cA(a)=p_{\cA,a}\Big(\frac{\xi}{\abs{\xi}}\Big)\geq c_L \dist^{2L}\Big(\frac{\xi}{\abs{\xi}},S_\cA(a)\Big) \geq c_L \frac{\dist^{2L}(\xi,\RR S_\cA(a))}{\abs{\xi}^{2L}}
\end{equation}
for all $\xi\in\RR^d\setminus\{0\}$, where we used zero-homogeneity of $\ppa$ and denote by $\RR S_\cA(a)\subset\RR^d$ the linear cone generated by $S_\cA(a)$.\medbreak

By assumption, there exists a set $W\subset\RR^d$, which contains $S_\cA(a)$ and is a finite union of linear subspaces of $\RR^d$ of dimensions at most $d-1$.
Now, the linearity of the subspaces yields $\RR S_\cA(a)\subset W$. Together with \eqref{eq:keyestimate}, this implies
\begin{equation}\label{eq:intermediate2}
  E^\cA_{el}(\chi; F) -\abs{\Omega} Q^\cA_\theta W(F) \geq 
  c_L \int_{\RR^d}\lvert\hat{f}(\xi)\rvert^2  \frac{\dist^{2L}(\xi,W)}{\abs{\xi}^{2L}} \, d\xi =: \tilde{E}_{el}(f).
\end{equation}
Since $S_\cA(a)\subset S^{d-1}$ is nonempty (\autoref{prop:smoothCompProj}), we have $W\ne\{0\}$
and can apply \autoref{prop:lowerscaling33}. This yields a constant $c_1=c_1(a,\cA, \Omega, d, L, W)=c_1(a,\cA, \Omega, d)>0$ such that
\begin{equation}\label{eq:proplowerscaling33lowerboundApplied}
\tilde{E}_{el}(f)+\eps\tilde{E}_{surf}(f)+\eps\per(\Omega) \geq c_1 \norm{f}_{L^2(\Omega)}^2 \eps^{\nicefrac{2L}{(2L+1)}} \quad\forall\eps\in(0,1).
\end{equation}
Because $\del_j \chi = (\del_j f) a$ in $\cM(\RR^d; X )$ for all $j\in\{1,\dots,d\}$,
the surface energies from \eqref{eq:surfaceenergy} and \eqref{eq:proplowerscaling33energies} are equivalent:
\begin{equation}\label{eq:surfaceenergiesequivalent}
E_{surf}(\chi)=\abs{a}\tilde{E}_{surf}(f).
\end{equation}
Again, using that $\norm{f}_{L^2(\Omega)}^2=\abs{\Omega}\theta(1-\theta)$ (see \eqref{eq:fhatsquared}), we obtain
\begin{equation}\label{eq:intermediate3}
  E^\cA_{el}(\chi; F) -\abs{\Omega} Q^\cA_\theta W(F) +\eps E_{surf}(\chi)+\eps\per(\Omega) \geq 
 c_2 \,\theta(1-\theta) \eps^{\nicefrac{2L}{(2L+1)}} \quad\forall\eps\in(0,1)
\end{equation}
for some constant $c_2=c_2(a,\cA, \Omega, d)>0$.
Recalling \autoref{thm:qwdom} and the definition of the optimal volume fraction \eqref{eq:optvolfrac}, it follows that
\begin{equation}
    E^\cA_{\eps}(\chi; F) - E^\cA_0(F,\cK) +\eps\per(\Omega) \geq \abs{\Omega} \big(Q^\cA_\theta W(F) - Q^\cA_{\ttheta_\cA} W(F) \big) + c_2 \,\theta(1-\theta) \eps^{\nicefrac{2L}{(2L+1)}} \quad\forall\eps\in(0,1).
\end{equation}
By \eqref{eq:leadingCoefficient}, we know that $\theta\mapsto Q^\cA_\theta W(F)$ is quadratic with leading coefficient $g_\cA(a)>0$ and minimum in $\ttheta_\cA\in(0,1)$.
We infer that
\begin{equation}\label{eq:intermediate4}
    E^\cA_{\eps}(\chi; F) - E^\cA_0(F,\cK) +\eps\per(\Omega) \geq \abs{\Omega} g_\cA(a)(\theta-\ttheta_\cA)^2 + c_2 \,\theta(1-\theta) \eps^{\nicefrac{2L}{(2L+1)}} \quad\forall\eps\in(0,1).
\end{equation}
Now, pick any $\eps_0=\eps_0(a,\cA,\Omega, d,\ttheta_\cA)\in(0,1)$ with
\begin{equation}
    \eps_0^{\nicefrac{2L}{(2L+1)}}\leq \frac{(\min\{\ttheta_\cA,1-\ttheta_\cA\})^2 \abs{\Omega}g_\cA(a)}{4 c_2\ttheta_\cA(1-\ttheta_\cA)}.  
\end{equation}
A quick computation shows that there exists a constant $c_3=c_3(a,\cA, \Omega, d,\ttheta_\cA)>0$ such that
\begin{equation}
\min_{\theta\in\RR}\abs{\Omega} g_\cA(a)(\theta-\ttheta_\cA)^2 + c_2 \,\theta(1-\theta) \eps^{\nicefrac{2L}{(2L+1)}} \geq c_3 \eps^{\nicefrac{2L}{(2L+1)}}\quad\forall\eps\in(0,\eps_0).
\end{equation}
Combining this estimate with \eqref{eq:intermediate4} and minimizing over $\chi\in BV(\Omega;\cK)$ proves
\begin{equation}\label{eq:intermediate5}
    E^\cA_{\eps}(F,\cK) - E^\cA_0(F,\cK) +\eps\per(\Omega) \geq c_3 \eps^{\nicefrac{2L}{(2L+1)}}\quad\forall\eps\in(0,\eps_0).
\end{equation}
By possibly reducing the constants $\eps_0>0$ and $c_3>0$, still depending on the same parameters, we can absorb the perimeter:
\begin{equation}
 E^\cA_{\eps}(\chi; F) - E^\cA_0(F,\cK)  \geq c_3 \eps^{\nicefrac{2L}{(2L+1)}}\quad\forall\eps\in(0,\eps_0).
\end{equation}
This shows \eqref{eq:afree23scaling} and completes the proof.
\end{proof}
\begin{rmk}\label{rmk:FourierEquicomp}
In the Fourier representation of the elastic energy \eqref{eq:softboundaryconstraint}, the function $p_{\cA,a}$ arises very naturally.
As mentioned in the introduction, \autoref{thm:incompa} fails to provide a lower scaling bound in the peculiar situation when $S_{\cA}(a)=\Sdm$.
In this case, the Fourier multiplier $p_{\cA,a}$ vanishes everywhere in $\Sdm$ and the lower bound \eqref{eq:softboundaryconstraint} degenerates.
In particular, the maximal vanishing order $L_\cA(a)$ is not defined, and we do not control the elastic energy as in \eqref{eq:intermediate2}.
We note that, in deducing the lower bound \eqref{eq:softboundaryconstraint}, the only step where the elastic energy may have been estimated too coarsely
occurs when the hard boundary constraint $w\in \cD^\cA_0(\Omega)$ is relaxed to the softer condition $w\in L^2(\RR^d; X )$ with $\cA w=0$.
Therefore, to effectively control the elastic energy when $S_{\cA}(a)=\Sdm$, a more refined treatment of the boundary condition is necessary.\medbreak
\end{rmk}
\begin{rmk}
Although \autoref{thm:incompa} does not yield a lower scaling bound for $E^{\cA}_\eps(F,\cK)-E^{\cA}_0(F,\cK)$ when $S_{\cA}(a)=\Sdm$, 
it is possible to obtain an $\eps$-lower scaling bound:
In the setting of \autoref{thm:incompa}, suppose that $\ttheta_{\cA}(F,\cK)\in(0,1)$. Then,
estimating the surface energy using the Poincaré inequality for $BV$-functions \cite[Remark 3.50]{ambrosio}, there exists $c=c(a, \Omega,d,\ttheta_{\cA})>0$ such that
\begin{equation}\label{eq:EpsLowerBdPeculiar}
  c\eps\leq E^{\cA}_\eps (F,\cK)-E^{\cA}_0(F,\cK) \quad \forall \eps\geq 0.
\end{equation}
However, as this lower scaling bound essentially neglects the energy contribution of the elastic energy, it may not be optimal. We plan to investigate this in future work \cite{MPF26}.
\end{rmk}

In the present work, we do not attempt to improve the lower bound \eqref{eq:EpsLowerBdPeculiar} when the Fourier-based technique fails.
We therefore exclude the following states in our analysis of lower bounds.
\begin{definition}[Equicompatible states]\label{def:equicompatibleStates}
Let $\cA$ be a differential operator as in \eqref{eq:diffopera}. For $a\in X$, let $S_{\cA}(a)$ be given as in \autoref{def:optLaminationDirections}.
We define the set of equicompatible states by
\begin{equation}\label{eq:equicompatibleStates}
  E_\cA:=\{a\in X :S_\cA(a)=S^{d-1}\}.
\end{equation}
\end{definition}

In the next section, we apply \autoref{thm:incompa} for $\cA\in\{\curl,\Div,\ccurl\}$.
As it is rather cumbersome to compute the maximal vanishing order by hand, we next provide a lemma that ensures that the assumptions of 
\autoref{thm:incompa}~\textcolor{blue}{(ii)} are satisfied when the compatibility projection has a specific structure.
\begin{lemma}\label{lem:LCP}
Let $\cA$ be a differential operator as in \eqref{eq:diffopera}. Let $E_\cA$ be the set of equicompatible states as given in \autoref{def:equicompatibleStates}.
Let $a\in X \setminus E_\cA$. Denote by $S_\cA(a)$ the set of optimal lamination directions; see \autoref{def:optLaminationDirections}.
Suppose that there exist a dimension $l\in\NN$, a linear map $m\in\Lin(\RR^d;\RR^{l})$ and a constant $c\in\RR$ such that one of the following holds:
\begin{align}
    (i)&\hspace{1cm} \abs{a-\ppa(\xi)a}_\cA^2 =  \abs{ m \xi}^2+c \quad\forall \xi\in S^{d-1},\label{eq:LCPa} \\
    (ii)&\hspace{1.68cm} \abs{\ppa(\xi)a}_\cA^2 = \abs{ m \xi}^2+c \quad\forall \xi\in S^{d-1}.\label{eq:LCPb}
\end{align}
Then, $S_\cA(a)$ is contained in a linear subspace of $\RR^d$ of dimension at most $d-1$ and the maximal vanishing order of $p_{\cA,a}$ is $L_\cA(a)=1$; see \autoref{def:LcAa}.
\end{lemma}
This lemma is inspired by the result \cite[Lemma 3.7]{ruland} on lower scaling bounds of $E^{\cA}_\eps(F,\cK)$ in the \textit{compatible} setting for
first order linear differential operators $\cA$ of the form \eqref{eq:diffopera}. The proof in \cite{ruland} makes use of the lower bound (see \cite[Proof of Lemma 3.1]{ruland})
\begin{equation}\label{eq:pseudoinverseEst}
p_{\cA,a}(\xi) =  \abs{a - \ppa(\xi)a}_\cA^2 -h_\cA(a) = \abs{a - \ppa(\xi)a}_\cA^2 \geq c \lvert{\AA(\xi)a}\rvert_Y^2\quad\forall \xi\in\Sdm,\, a\in\Lambda_\cA
\end{equation}
and the linear structure of $\xi\in\RR^d\mapsto \AA(\xi)a$ for fixed $a\in X$.
The conditions \eqref{eq:LCPa} and \eqref{eq:LCPb} are reminiscent of this linear structure and allow us to generalize the argument from \cite[Lemma~3.7]{ruland} to the incompatible setting.
Note that the estimate \eqref{eq:pseudoinverseEst} relies on the fact that $h_\cA(a)=0$ for all compatible states $a\in\Lambda_\cA$.
\begin{proof}[Proof of \autoref{lem:LCP}]
The proof is organized into two steps. First, we derive the lemma under assumption \eqref{eq:LCPa}. In the second step, we provide the details for the setting \eqref{eq:LCPb}.
\Step{1: Proof for the setting \eqref{eq:LCPa}}
Let $m\in\Lin(\RR^d;\RR^{l})$ and $c\in\RR$ be such that \eqref{eq:LCPa} is satisfied.
We make use of the identity:
\begin{equation}\label{eq:AdjointDef}
     \abs{m \xi}^2 = (m^Tm\xi,\xi)\quad\forall\xi\in\RR^d.
\end{equation}
Since $m^Tm$ is a positive semidefinite, symmetric operator on $\RR^d$, the spectral theorem provides a spectral decomposition of $m^Tm$.
Let $\lambda_{min}$ denote the smallest eigenvalue of $m^Tm$ and $U_{min}:=E(m^Tm, \lambda_{min})$ the associated eigenspace.
Due to \eqref{eq:AdjointDef}, we find that
\begin{equation}\label{eq:formulaLCPCompQuant}
h_\cA(a)=\min_{\xi\in S^{d-1}}\abs{a-\ppa(\xi)a}_\cA^2 = \min_{\xi\in S^{d-1}} \abs{m \xi}^2+c = \lambda_{min}+c
\end{equation}
and that this minimum is attained in
\begin{equation}\label{eq:LCPOptiLamDir}
    S_\cA(a)=\argmin_{\xi\in S^{d-1}}\abs{a-\ppa(\xi)a}_\cA^2 =S^{d-1}\cap U_{min},
\end{equation}
which is the zero set of the function $p_{\cA,a}$.
Since we assumed that $a$ is not equicompatible, it follows that $ U_{min}$ is a linear subspace of dimension at most $d-1$, which contains $S_\cA(a)$.\medbreak

We now turn towards proving $L_\cA(a)=1$.
As $\lambda_{min}$ is the smallest eigenvalue of $m^Tm$, there exists $\delta=\delta(m)>0$ such that 
\begin{equation}
  \abs{m\xi'}^2\geq\abs{m\xi''}^2+\delta=\lambda_{min}+\delta \quad \forall\xi'\in S^{d-1}\cap U_{min}^\perp, \,\xi''\in S^{d-1}\cap U_{min}
\end{equation}
and therefore it holds that
\begin{equation}\label{eq:usefulest}
  \abs{m\xi'}^2  \geq \left(\lambda_{min}+\delta\right)\abs{\xi'}^2\quad\forall\xi'\in U_{min}^\perp.
\end{equation}
Another important observation is the orthogonality:
\begin{equation}\label{eq:frequortho}
  (m\xi', m\xi'')=  (\xi', \lambda_{min}\xi'') = 0\quad\forall\xi'\in U_{min}^\perp, \,\xi''\in U_{min}
\end{equation}
For $\xi\in\RR^d$, we use the orthogonal decomposition to write $\xi=\xi'+\xi''$ with $\xi'\in U_{min}^\perp,\, \xi''\in U_{min}$.
Taking both \eqref{eq:usefulest} and \eqref{eq:frequortho} into account, a brief computation shows that
\begin{equation}\label{eq:estforrul}
  \abs{m\xi}^2= \abs{m\xi'}^2+2(m\xi', m\xi'')+\abs{m\xi''}^2
  = \abs{m\xi'}^2+\lambda_{min}\abs{\xi''}^2\geq \lambda_{min}\abs{\xi}^2 + \delta\abs{\xi'}^2 \quad\forall \xi\in \RR^d.
\end{equation}
Together with \eqref{eq:LCPa} and \eqref{eq:formulaLCPCompQuant}, we obtain the bound 
\begin{equation}\label{eq:pcAaLowerBound}
p_{\cA,a}(\xi)= \abs{a-\ppa(\xi)a}_\cA^2-h_\cA(a) = \abs{m\xi}^2 - \lambda_{min} \geq \delta\abs{\xi'}^2\quad\forall \xi\in S^{d-1}.
\end{equation}
As we will show next, we have the scaling $\dist^2(\xi,S_\cA(a))\sim \abs{\xi'}^2$ for $\xi\in\Sdm$ close to $S_\cA(a)$, which 
together with \eqref{eq:pcAaLowerBound} yields $L_\cA(a)=1$.
More precisely, we will now prove the scaling:
\begin{equation}\label{eq:LCPSphereDistanceProjection}
    \abs{\xi'}^2\leq \dist^2(\xi,S_\cA(a))\leq 2\abs{\xi'}^2 \quad \forall \xi\in S^{d-1}: \abs{\xi'}^2\leq\tfrac{3}{4}.
\end{equation}
To this end, let $\xi\in S^{d-1}$ with $\abs{\xi'}^2\leq\tfrac{3}{4}$. Then, there exists $\eta\in S_\cA(a)$ such that $\dist^2(\xi,S_\cA(a))= \abs{\xi-\eta}^2$.
Since $\xi''\ne 0$, it follows by \eqref{eq:LCPOptiLamDir} that $\eta$ is unique and explicitly given by $\eta = {\xi''}/{\abs{\xi''}}$.
Using orthogonality, we obtain
\begin{equation}
    \dist^2(\xi,S_\cA(a))=\abs{\xi'+\xi''-\eta}^2 = \abs{\xi'}^2 +  \abs{\xi''-\eta}^2   = \abs{\xi'}^2 +  \big(1-\abs{\xi''}\big)^2.
\end{equation}
Next, we employ $1 = \abs{\xi}^2=\abs{\xi'}^2+ \abs{\xi''}^2$ to infer that
\begin{equation}
\begin{aligned}
        \dist^2(\xi,S_\cA(a)) = \abs{\xi'}^2 + \big(1-\abs{\xi''}\big)^2 &= 2-2\abs{\xi''}= 2\left(1-\sqrt{1-\abs{\xi'}^2} \,\right)\\
        &= - 2\int_0^1 \frac{d}{dt}\sqrt{1-t\abs{\xi'}^2}\,dt =  \int_0^1 \frac{\abs{\xi'}^2}{\sqrt{1-t\abs{\xi'}^2}}\,dt.
\end{aligned}
\end{equation}
Since $\abs{\xi'}^2\leq\tfrac{3}{4}$, we have $1\leq (1-t\abs{\xi'}^2)^{-\frac{1}{2}}\leq2$ for all $t\in(0,1)$,
which proves \eqref{eq:LCPSphereDistanceProjection} and completes the proof for the setting \eqref{eq:LCPa}.\medbreak
\Step{2: Proof for the setting \eqref{eq:LCPb}}
It remains to prove the lemma under assumption \eqref{eq:LCPb}.
Let $m\in\Lin(\RR^d;\RR^{l})$ and $c\in\RR$ be such that \eqref{eq:LCPb} is satisfied.
Moreover, let $\lambda_{max}$ denote the largest eigenvalue of $m^Tm$ and $U_{max}:=E(m^Tm, \lambda_{max})$ the associated eigenspace.
Again using \eqref{eq:AdjointDef}, we compute
\begin{equation}\label{eq:LCPbgcA}
g_\cA(a)=\max_{\xi\in S^{d-1}}\abs{\ppa(\xi)a}_\cA^2 =  \max_{\xi\in S^{d-1}} \abs{ m \xi}^2+c = \lambda_{max} + c
\end{equation}
and the set of maximizers
\begin{equation}\label{eq:LCPbgcA2}
S_\cA(a)=\argmax_{\xi\in S^{d-1}}\abs{\ppa(\xi)a}_\cA^2 =S^{d-1}\cap U_{max}.
\end{equation}
In particular, $U_{max}$ is a linear subspace of $\RR^d$ which contains $S_\cA(a)$. By assumption $a$ is not equicompatible and so $U_{max}$ has dimension at most $d-1$.\medbreak
Next, using \autoref{rmk:ZeroOfpAndOtherFormOfp} and \eqref{eq:LCPb}, we rewrite
\begin{equation}
    p_{\cA,a}(\xi)=g_\cA(a)-\abs{\ppa(\xi)a}_\cA^2 = \lambda_{max} - \abs{ m \xi}^2\quad\forall\xi\in S^{d-1}.
\end{equation}
Furthermore, there exists a constant $\delta=\delta(m)>0$ such that
\begin{equation}
  \abs{m\xi'}^2 +\delta  \leq\abs{m\xi''}^2=\lambda_{max} \quad \forall\xi'\in S^{d-1}\cap U_{max}^\perp, \,\xi''\in S^{d-1}\cap U_{max}.
\end{equation}
Denoting by $\xi'$ the orthogonal projection of $\xi$ onto $U_{max}^\perp$ and repeating the arguments from the first part of the proof, we obtain the estimate
\begin{equation}
p_{\cA,a}(\xi)= \lambda_{max}-\abs{m\xi}^2 \geq \delta\abs{\xi'}^2\quad\forall \xi\in S^{d-1}.
\end{equation}
Due to \eqref{eq:LCPSphereDistanceProjection}, this implies $L_\cA(a)=1$ and concludes the proof.
\end{proof}

\subsection{Application of the general result to the model problems}\label{subsec:ApplicationOfGeneralLowerScaling}
Our next goal is to apply \autoref{thm:incompa} for $\cA\in\{\curl, \Div, \ccurl\}$.
\autoref{subsec:ApplicationOfGeneralLowerScaling} is thus divided into three parts -- one for each differential operator.\\
In \autoref{subsubsec:Grad}, we characterize for the differential operator $\curl$ the objects that were introduced to study incompatible $\cA$-free two-well problems.
This is complemented by the corresponding analysis for the divergence operator in \autoref{subsubsec:Div}.
We then apply \autoref{lem:LCP} and \autoref{thm:incompa} for $\cA\in\{\curl, \Div\}$ and prove \autoref{thm:incompgrad}~\textcolor{blue}{(i+ii)}.
In \autoref{subsubsec:CC}, we carry out the analogous calculations for the differential operator $\ccurl$, which allows us to prove \autoref{thm:incompsyma}~\textcolor{blue}{(i+ii)}.
\subsubsection{Two-well problem for the gradient}\label{subsubsec:Grad}
To apply \autoref{thm:incompa} for the differential operator $\curl$, we need to determine the maximal vanishing order $L_{\curl}(a)$ for $a\in\rdd$; see \autoref{def:LcAa}.
Towards this goal, the following lemma characterizes the compatibility projection.
\begin{lemma}\label{lem:CompProjCurl}
Let $d\geq2$. Consider the differential operator $\curl$; see \eqref{eq:curl}.
For $\xi\in\rrdmz$, let $V_{\curl}(\xi)$ and $\ppc(\xi)$ be given as in \eqref{eq:xicompatiblestates} and \autoref{def:compatibilityProjection}, respectively.
Then, it holds that
\begin{equation}\label{eq:xicompatiblestatesCurl}
    V_{\curl}(\xi)=\{b\otimes\xi:b\in\RR^d\}\quad\forall\xi\in\rrdmz.
\end{equation}
The operator $\curl$ has constant rank and spanning wave cone; see \autoref{def:constantrank} and \autoref{def:SpanWaveCone}.
Moreover, we have
\begin{equation}\label{eq:compProjCurl}
\ppc(\xi)a = (a\xi)\otimes \xi \; \text{ and } \; \abs{\ppc(\xi)a}^2 = \abs{a\xi}^2\quad\forall\xi\in\Sdm,\,a\in\rdd.
\end{equation}
\end{lemma}
\begin{rmk}
As $\ppc$ is zero-homogeneous, we only computed $\ppc(\xi)a$ for $\xi\in\Sdm$. By normalization, formula \eqref{eq:compProjCurl} extends to all $\xi\in\rrdmz$.
\end{rmk}
\begin{proof}[Proof of \autoref{lem:CompProjCurl}]
Formula \eqref{eq:xicompatiblestatesCurl} is well-known; see, for instance, \cite[Corollary 8.31]{rindler}.
We observe that $\curl$ has constant rank, since $V_{\curl}(\xi)$ is the kernel of the symbol of $\curl$ and it is of dimension $d$ for all $\xi\in\rrdmz$.
The wave cone contains the matrices $e_i\otimes e_j$ for all $1\leq i,j\leq d$, which form a basis of $\rdd$. Hence, the $\curl$ operator has spanning wave cone.\medbreak

It remains to prove \eqref{eq:compProjCurl}. To this end, let $\xi\in\Sdm$ and $a\in\rdd$.
By \autoref{def:compatibilityProjection}, we have $\ppc(\xi)a\in V_{\curl}(\xi)$ and $(a-\ppc(\xi)a) \perp  V_{\curl}(\xi)$ in $\rdd$.
In particular, there exists $\tb\in\RR^d$ such that
\begin{equation}\label{eq:perpCondForCurl}
  \ppc(\xi)a=\tb\otimes\xi\text{ and }
  (\tb\otimes\xi,  b\otimes\xi) = (a,  b\otimes\xi)\quad \forall b \in\RR^d.
\end{equation}
Using that 
\begin{equation}\label{eq:MbotimesXi}
  (M, b\otimes\xi)= (M\xi,b) \quad\forall M\in\rdd,\, b\in\RR^d,
\end{equation}
we obtain from \eqref{eq:perpCondForCurl} that
\begin{equation}\label{eq:perpCondForCurl2}
 ((\tb\otimes\xi)\xi,  b) = (a\xi,  b)\quad \forall b \in\RR^d.
\end{equation}
Since $\abs{\xi}=1$, we find $((\tb\otimes\xi)\xi,  b) = (\tb,b)$, 
which together with \eqref{eq:perpCondForCurl2} shows that $\tb=a\xi$ and proves the first formula in \eqref{eq:compProjCurl}. 
The second formula in \eqref{eq:compProjCurl} follows from the identity $\abs{b\otimes\xi}= \abs{b}\abs{\xi}$, which holds for all $b,\xi \in\RR^d$.
\end{proof}
A direct consequence of \autoref{lem:LCP} and \autoref{lem:CompProjCurl} is that $L_{\curl}(a)=1$ for any $a\in\rdd$ that is not equicompatible.
To see this, note that \eqref{eq:LCPb} is satisfied due to \eqref{eq:compProjCurl}.
It remains to characterize the set of equicompatible states $E_{\curl}$, which is addressed in the following lemma.
\begin{lemma}\label{lem:EquicompatibleStatesCurl}
Let $d\geq2$. Consider the differential operator $\curl$; see \eqref{eq:curl}.
Let $a\in\rdd$. Let $g_{\curl}(a)$ and $S_{\curl}(a)$ be as in \autoref{def:compquant} and \autoref{def:optLaminationDirections}, respectively.
Moreover, let $\lambda_{max}(a^Ta)$ be the largest eigenvalue of $a^Ta$ and $E(a^Ta;\lambda_{max}(a^Ta))$ the associated eigenspace.
Then, it holds that 
\begin{equation}\label{eq:CompQuantAndOptLamCurl}
g_{\curl}(a) = \lambda_{max}(a^Ta) \; \text{ and } \; 
S_{\curl}(a) = S^{d-1}\cap E(a^Ta;\lambda_{max}(a^Ta)).
\end{equation}
In particular, the set of equicompatible states, introduced in \autoref{def:equicompatibleStates}, is given by
\begin{equation}\label{eq:equicompatibleStatesCurl}
    E_{\curl}=\RR O(d).
\end{equation}
\end{lemma}
\begin{proof}
As \eqref{eq:compProjCurl} corresponds to \eqref{eq:LCPb} with $m=a$ and $c=0$,
both formulae in \eqref{eq:CompQuantAndOptLamCurl} follow by the same argument as in the second step of the proof of \autoref{lem:LCP}; see \eqref{eq:LCPbgcA} and \eqref{eq:LCPbgcA2}.
Using \eqref{eq:CompQuantAndOptLamCurl}, it holds that $S_{\curl}(a)=\Sdm$ if and only if $a^Ta=\lambda_{max}(a^Ta) I_d$, which proves \eqref{eq:equicompatibleStatesCurl}.
\end{proof}
We are now ready to prove \autoref{thm:incompgrad}~\textcolor{blue}{(i+ii)} for $\cA=\curl$.
We however postpone this proof to the end of \autoref{subsubsec:Div}, where we simultaneously prove the theorem for $\cA=\Div$.
    
\subsubsection{Two-well problem for the divergence}\label{subsubsec:Div}
Next, we study the situation for the divergence operator.
\begin{lemma}\label{lem:CompProjDiv}
Let $d\geq2$. Consider the differential operator $\Div$; see \eqref{eq:Div}.
For $\xi\in\rrdmz$, let $V_{\Div}(\xi)$ and $\ppd(\xi)$ be given as in \eqref{eq:xicompatiblestates} and \autoref{def:compatibilityProjection}, respectively.
Then, we have
\begin{equation}\label{eq:xicompatiblestatesDiv}
    V_{\Div}(\xi)=\{a\in\rdd:a\xi=0\}\quad\forall\xi\in\rrdmz.
\end{equation}
The divergence operator has constant rank and spanning wave cone; see \autoref{def:constantrank} and \autoref{def:SpanWaveCone}.
Moreover, it holds that
\begin{equation}\label{eq:compProjDiv}
\ppd(\xi)a = a- (a\xi)\otimes \xi \; \text{ and } \; \abs{a-\ppd(\xi)a}^2 = \abs{a\xi}^2\quad\forall\xi\in\Sdm.
\end{equation}
\end{lemma}
\begin{proof}
Let $\AA$ be the symbol of $\cA=\Div$ given as in \eqref{eq:symbol}. It holds that
\begin{equation}
  \AA(\xi)a=a\xi\quad\forall\xi\in \rrdmz,\,a\in\rdd,
\end{equation}
which proves \eqref{eq:xicompatiblestatesDiv}.\medbreak

Since $e_i\otimes e_j\in V_{\Div}(e_k)$ for all $1\leq i,j,k\leq d$ with $j\ne k$,
the wave cone contains the matrices $e_i\otimes e_j$ for all  $1\leq i,j\leq d$, which form a basis of $\rdd$. Therefore, the divergence operator has spanning wave cone. 
For all $\xi\in\rrdmz$ and $a\in\rdd$, we obtain from \eqref{eq:MbotimesXi} and \eqref{eq:xicompatiblestatesDiv} that
\begin{equation}
  a\in V_{\Div}(\xi) \iff \forall b\in\RR^d: a\perp b\otimes\xi.
\end{equation}
With \eqref{eq:xicompatiblestatesCurl}, this implies 
\begin{equation}\label{eq:VDivIsVCurlPerp}
  V_{\Div}(\xi)=V_{\curl}(\xi)^{\perp} \quad \forall \xi\in\rrdmz.
\end{equation}
In particular, it holds that $\dim V_{\Div}(\xi)=d(d-1)$ for all $\xi\in \rrdmz$ and thus the divergence operator has constant rank.\medbreak

Now, \eqref{eq:VDivIsVCurlPerp} yields the orthogonal decomposition $\RR^d=V_{\curl}(\xi)\oplus V_{\Div}(\xi)$, which proves
\begin{equation}\label{eq:PpcPlusPpdIsId}
  \ppc(\xi)a+\ppd(\xi)a=a \quad\forall a\in\rdd,\,\xi\in \Sdm.
\end{equation}
We conclude the proof by observing that \eqref{eq:compProjCurl} and \eqref{eq:PpcPlusPpdIsId} imply \eqref{eq:compProjDiv}.
\end{proof}
By \autoref{lem:LCP} and \autoref{lem:CompProjDiv}, it follows that $L_{\Div}(a)=1$ for any $a\in\rdd$ that is not equicompatible. Note that
\eqref{eq:LCPa} is satisfied due to \eqref{eq:compProjDiv}.
It remains to characterize the set $E_{\Div}$, which is the content of the following lemma.
\begin{lemma}\label{lem:EquicompatibleStatesDiv}
Let $d\geq2$. Consider the differential operator $\Div$; see \eqref{eq:Div}.
Let $a\in\rdd$. Let $h_{\Div}(a)$ and $S_{\Div}(a)$ be as in \autoref{def:compquant} and \autoref{def:optLaminationDirections}, respectively.
Moreover, let $\lambda_{min}(a^Ta)$ be the smallest eigenvalue of $a^Ta$ and $E(a^Ta;\lambda_{min}(a^Ta))$ the associated eigenspace.
Then, it holds that 
\begin{equation}\label{eq:CompQuantAndOptLamDiv}
h_{\Div}(a) = \lambda_{min}(a^Ta) \; \text{ and } \; 
S_{\Div}(a) = S^{d-1}\cap E(a^Ta;\lambda_{min}(a^Ta)).
\end{equation}
In particular, the set of equicompatible states, introduced in \autoref{def:equicompatibleStates}, is given by
\begin{equation}\label{eq:equicompatibleStatesDiv}
    E_{\Div}=\RR O(d).
\end{equation}
\end{lemma}
\begin{proof}
The formulae in \eqref{eq:CompQuantAndOptLamDiv} follow from the same argument as in first step of the proof of \autoref{lem:LCP} since 
\eqref{eq:compProjDiv} corresponds to \eqref{eq:LCPa} with $m=a$ and $c=0$; see, in particular, \eqref{eq:formulaLCPCompQuant} and \eqref{eq:LCPOptiLamDir}.
By \eqref{eq:CompQuantAndOptLamDiv}, it follows that $S_{\curl}(a)=\Sdm$ if and only if $a^Ta=\lambda_{min}(a^Ta) I_d$. This proves \eqref{eq:equicompatibleStatesDiv} and completes the proof.
\end{proof}
We are now ready to prove the first and second statements of \autoref{thm:incompgrad}.
\begin{proof}[Proof of \autoref{thm:incompgrad}~\textcolor{blue}{(i+ii)}]
Let $\cA\in\{\curl,\Div\}$. By \autoref{lem:CompProjCurl} and \autoref{lem:CompProjDiv}, we know that $\cA$ has constant rank and spanning wave cone.
Hence, the requirements of \autoref{thm:incompa}~\textcolor{blue}{(i)} are satisfied, and 
\autoref{thm:incompgrad}~\textcolor{blue}{(i)} is immediate.\medbreak
To apply \autoref{thm:incompa}~\textcolor{blue}{(ii)}, we need to verify some additional conditions.
As in the assumptions of \autoref{thm:incompgrad}~\textcolor{blue}{(ii)}, suppose that $a=a_1-a_0\in\rdd\setminus \RR O(d)$.
By \autoref{lem:EquicompatibleStatesCurl} and \autoref{lem:EquicompatibleStatesDiv}, we have $a\not\in E_\cA$.
We can apply \autoref{lem:LCP} due to \eqref{eq:compProjCurl} and \eqref{eq:compProjDiv}. This proves that 
$S_\cA(a)$ is contained in a linear subspace of $\RR^d$ of dimension at most $d-1$ and that $L_\cA(a)=1$.
We conclude the proof, observing that \autoref{thm:incompa}~\textcolor{blue}{(ii)} yields \autoref{thm:incompgrad}~\textcolor{blue}{(ii)}.
\end{proof}

\subsubsection{Two-well problem for the symmetrized gradient}\label{subsubsec:CC}
In contrast, to the first order differential operators $\cA\in\{\curl,\Div\}$, more work is required to verify the conditions of \autoref{thm:incompa}~\textcolor{blue}{(ii)}
for $\cA=\ccurl$. The reason for this is that \autoref{lem:LCP} is not applicable in this setting, which makes it necessary 
that we determine the maximal vanishing order $L_{\cc}(a)$ for $a\in\rddsym$ by hand. Towards this goal, we characterize the compatibility projection in the next lemma.

\begin{lemma}\label{lem:CompProjCC}
Let $d\geq2$. Consider the differential operator $\ccurl$; see \eqref{eq:curlcurl}.
For $\xi\in\rrdmz$, let $V_{\cc}(\xi)$ and $\ppcc(\xi)$ be given as in \eqref{eq:xicompatiblestates} and \autoref{def:compatibilityProjection}, respectively.
Then, it holds that
\begin{equation}\label{eq:xicompatiblestatesCC}
    V_{\cc}(\xi)=\{b\odot\xi:b\in\RR^d\}\quad\forall\xi\in\rrdmz.
\end{equation}
The differential operator $\ccurl$ has constant rank and spanning wave cone; see \autoref{def:constantrank} and \autoref{def:SpanWaveCone}.
Moreover, we have
\begin{equation}\label{eq:CompProjCurlCurlFormula}
\ppcc(\xi)a = G_{\xi}(a\xi) \odot \xi \; \text{ and } \; \abs{\ppcc(\xi)a}^2=2\abs{a\xi}^2 -  (a\xi,\xi)^2\quad\forall\xi\in\Sdm,
\end{equation}
where $G_\xi\in\Lin(\RR^d)$ is defined, for $\xi\in\Sdm$, by
\begin{equation}
        G_\xi(v):= 2v-(\xi, v)\xi, \quad v\in \RR^d.
\end{equation}
\end{lemma}
To derive \autoref{lem:CompProjCC}, the need the following formulae.
\begin{lemma}\label{lem:symprodform}
Let $b,\xi\in\RR^d$ and $a\in \rddsym$, then
\begin{equation}
  (a,b\odot\xi) = (a\xi, b), \qquad \qquad \abs{b\odot \xi}^2 = \tfrac{1}{2} \left(\abs{b}^2\abs{\xi}^2+ (b,\xi)^2\right).
\end{equation}
\end{lemma}
\begin{proof}
Using \eqref{eq:MbotimesXi}, we compute
\begin{equation}
  (a,b\odot\xi) = \tfrac{1}{2}(a, b\otimes \xi + \xi\otimes b) =  \tfrac{1}{2}\left[(a\xi, b ) + (ab,\xi ) \right] = (a\xi,b),
 \end{equation}
where we used that $a^T=a$. This proves the first formula, which we now apply for $a=b\odot\xi$.
This yields
\begin{equation}
\abs{b\odot \xi}^2 = ((b\odot\xi)\xi, b) = \tfrac{1}{2}((b\otimes \xi + \xi\otimes b)\xi, b)=  \tfrac{1}{2} \left(\abs{b}^2\abs{\xi}^2+ (b,\xi)^2\right)
\end{equation} 
and completes the proof.
\end{proof}
We are now ready to prove \autoref{lem:CompProjCC}.
\begin{proof}[Proof of \autoref{lem:CompProjCC}]
For $\xi\in\rrdmz$, it is well-known that $V_{\cc}(\xi)=\{b\odot\xi:b\in\RR^d\}$; see, for instance, \cite[Lemma 2.4]{ruland2}.
It follows that $\ccurl$ has constant rank and spanning wave cone in $\rddsym$.\medbreak

It remains prove \eqref{eq:CompProjCurlCurlFormula}. To this end, let $a\in\rddsym$ and $\xi\in S^{d-1}$. 
Since $\ppcc(\xi)a$ is defined as the orthogonal projection of $a$ onto $V_{\cc}(\xi)$ in $\RR^{d\times d}_{sym}$,
there exists a unique $\tb\in\RR^d$ with $\ppcc(\xi)a = \tb \odot \xi$. Our goal is to derive an explicit formula for $\tb$.
Using $\abs{\xi}=1$ and \autoref{lem:symprodform}, we find
\begin{equation}
    (a\xi,b) = (a,b\odot \xi) = (\tb \odot \xi, b\odot\xi) =  ((\tb \odot \xi)\xi, b) = \tfrac{1}{2}(\tb+(\xi, \tb)\xi, b )\quad\forall b\in\RR^d.
\end{equation}
Given a linear subspace $U$ of $\RR^d$, we denote by $\pi_U$ the orthogonal projection onto $U$ in $\RR^d$.
Then, the above identity implies 
\begin{equation}
    a\xi = \tfrac{1}{2}(\tb+(\xi, \tb)\xi) = (\pi_{\spn\{\xi\}} + \tfrac{1}{2}\pi_{\spn\{\xi\}^\perp})\tb.
\end{equation}
It follows that
\begin{equation}
    \ppcc(\xi)a = [(\pi_{\spn\{\xi\}} + \tfrac{1}{2}\pi_{\spn\{\xi\}^\perp})^{-1}(a\xi)] \odot\xi = G_{\xi}(a\xi) \odot \xi, 
\end{equation}
where we used that for any $v\in\RR^d$, we have
\begin{equation}\label{eq:gxieq}
\begin{aligned}
        (\pi_{\spn\{\xi\}} + \tfrac{1}{2}\pi_{\spn\{\xi\}^\perp})^{-1} v &= (\pi_{\spn\{\xi\}} + 2\pi_{\spn\{\xi\}^\perp})v \\
        &= (\xi, v)\xi + 2 (v - (\xi, v)\xi) = 2v-(\xi, v)\xi=G_\xi(v).
\end{aligned}
\end{equation}
With this in hand, we again apply \autoref{lem:symprodform} to compute 
\begin{multline}
    \abs{\ppcc(\xi)a}^2 = \tfrac{1}{2} \left(\abs{G_\xi(a\xi)}^2+(G_\xi(a\xi),\xi)^2\right)
    = \tfrac{1}{2} \left(\abs{2a\xi - (\xi,a\xi)\xi}^2 +(2a\xi - (\xi,a\xi)\xi,\xi)^2\right) \\
    = \tfrac{1}{2} \left(\abs{2a\xi}^2 - 2 (2a\xi, (\xi,a\xi)\xi) + \abs{(\xi,a\xi)\xi}^2 +(\xi,a\xi)^2\right) 
    = 2\abs{a\xi}^2 -  (a\xi,\xi)^2,
\end{multline}
which concludes the proof.
\end{proof}

Formula \eqref{eq:CompProjCurlCurlFormula} enables us to make explicit the compatibility quantifiers and optimal lamination directions in the following proposition.
\begin{proposition}\label{prop:optlaminationdirectioncurlcurl}
Let $d\geq2$. Consider the differential operator $\cA=\ccurl$; see \eqref{eq:curlcurl}.
Given $a\in\rddsym$, let $g_{\cc}(a)$ be the
compatibility quantifier and $S_{\cc}(a)$ the set of optimal lamination directions as defined in \autoref{def:compquant} and \autoref{def:optLaminationDirections}, respectively.
Denote by $\lambda_{-}=\lambda_{-}(a),\,\lambda_{+}=\lambda_{+}(a)$ the smallest and largest eigenvalues of $a$, respectively.
Moreover, let $U_{-}=E(a; \lambda_{-})$ and $U_{+}=E(a; \lambda_{+})$ be the respective eigenspaces. 
Then, the following holds:
\begin{enumerate}
  \item If $a$ is positive semidefinite (i.e. $\lambda_-\geq0$), we have
  \begin{equation}\label{eq:posdefiniteS0}
    g_{\cc}(a) = \lambda_{+}^2, \qquad\qquad S_{\cc}(a) = U_{+}\cap S^{d-1}.
  \end{equation}
  \item If $a$ is negative semidefinite (i.e. $\lambda_+\leq0$), we have
  \begin{equation}\label{eq:negdefiniteS0}
    g_{\cc}(a) = \lambda_{-}^2, \qquad\qquad S_{\cc}(a) = U_{-}\cap S^{d-1}.
  \end{equation}
  \item If $a$ is indefinite (i.e. $\lambda_-<0<\lambda_+$), we have 
  \begin{equation}\label{eq:indefiniteS0}
    g_{\cc}(a) = \lambda_{-}^2 + \lambda_{+}^2, \qquad \qquad S_{\cc}(a) = r_{-}\left(U_{-}\cap S^{d-1}\right) + r_{+}\left(U_{+}\cap S^{d-1}\right),
  \end{equation}
  where we are using the Minkowski sum and $r_{-},\,r_{+}\in(0,1)$ are determined by
  \begin{equation}\label{eq:formularmrp}
    r_{-}^2 = \frac{-\lambda_{-}}{\lambda_{+}-\lambda_{-}},\qquad\qquad r_{+}^2 = \frac{\lambda_{+}}{\lambda_{+}-\lambda_{-}}. 
  \end{equation}
\end{enumerate}
In particular, the set of equicompatible states, introduced in \autoref{def:equicompatibleStates}, is given by
\begin{equation}\label{eq:curlcurlequicomp}
  E_{\cc}=\RR I_d.
\end{equation}
\end{proposition}
\begin{rmk}\label{rmk:OptLamCCurlInTwoDim}
Let $d=2$ and $a\in\rddsym\setminus\RR I_2$.
An important consequence of \autoref{prop:optlaminationdirectioncurlcurl} is that $S_{\cc}(a)$ is contained in a union of up to two one-dimensional linear subspaces of $\RR^2$.\smallbreak
Indeed, if $a$ is positive semidefinite, we have $S_{\cc}(a)\subset U_+$ and $\dim U_+=1$.
Similarly, if $a$ is negative semidefinite, it holds that $S_{\cc}(a)\subset U_-$ with $\dim U_-=1$.
In contrast, if $a$ is indefinite, both $U_-$ and $U_+$ are one-dimensional and $S_{\cc}(a)$ consists of two pairs of antipodal points. In particular, the set $S_{\cc}(a)$ is contained in a union of two one-dimensional linear subspaces of $\RR^2$.
\end{rmk}
\begin{rmk}\label{rmk:OptLamCCurlInHighDim}
Note that $S_{\cc}(a)$ is not necessarily contained in a finite union of lower-dimensional linear subspaces of $\RR^d$ if $d>2$.
For instance, take $d=3$ and consider the indefinite matrix $a=e_1\otimes e_1+ e_2\otimes e_2 - e_3\otimes e_3\in\rddsymd{3}$. The set of optimal lamination directions consists of two circles shifted along the $\xi_3$-axis:
\begin{equation}
  S_{\cc}(a) = \big\{(\xi_1,\xi_2,\xi_3) \in S^2:(\xi_1,\xi_2)\in \tfrac{1}{\sqrt{2}} S^1 \text{ and } \xi_3 = \pm \tfrac{1}{\sqrt{2}} \big\}
\end{equation}
which cannot be covered by finitely many two-dimensional linear subspaces.
In particular, \autoref{thm:incompa}~\textcolor{blue}{(ii)} is not applicable, which is one of the reasons we restrict to two dimensions in \autoref{thm:incompsyma}~\textcolor{blue}{(ii)}.
\end{rmk}
\begin{proof}[Proof of \autoref{prop:optlaminationdirectioncurlcurl}]
The proof is organized into three steps. 
In the first step, we associate to a given $a\in \rddsym$ the polynomial $q:\RR^d\to\RR$ which agrees with the map $\xi\mapsto\abs{\ppc(\xi)a}^2$ on $\Sdm$.
Then, our strategy is to determine $g_{\cc}(a)$ and $S_{\cc}(a)$, by maximizing $q|_{\Sdm}$.
To this end, we derive properties of critical points of $q|_{\Sdm}$ in the second step.
We conclude the proof in the third step, where we characterize the maximum and the set of maximizers.
\Step{1: A related maximization problem}
Let $a\in \rddsym$. Since $g_{\cc}(a)=\max_{\xi\in \Sdm}\abs{\ppcc(\xi)a}^2$ and $S_{\cc}(a)=\argmax_{\xi\in \Sdm}\abs{\ppcc(\xi)a}^2$, we need to solve a variational problem,
which reformulate as follows. 
First, we associate to $a$ the polynomial $q:\RR^d\to\RR$, defined by
\begin{equation}\label{eq:p0}
  q(\xi):=2\lvert{a\xi}\rvert^2 -  (a\xi,\xi)^2,\quad\xi\in\RR^d.
\end{equation}
By \autoref{lem:CompProjCC}, we have
\begin{equation}\label{eq:p0identity}
  q(\xi)=\abs{\ppcc(\xi)a}^2 \quad\forall\xi\in S^{d-1}.
\end{equation}
It follows that $g_{\cc}(a)=\max_{S^{d-1}}q$ and $S_{\cc}(a)=\argmax_{S^{d-1}}q$.
We therefore maximize the polynomial $q$ on the unit sphere instead of the map $\xi\in\Sdm\mapsto\abs{\ppcc(\xi)a}^2$.
Note that identity \eqref{eq:p0identity} does not hold for general $\xi \in \RR^d \setminus \{0\}$, since $\ppcc$ is zero-homogeneous, whereas $q$ is not.
\Step{2: Properties of critical points}
Next, we derive properties of critical points of $q|_{\Sdm}$. To this end, let $\xi^*\in\Sdm$ be a critical point; that is, 
for any tangential vectors $v\in T_{\xi^*} S^{d-1}=\spn\{\xi^*\}^\perp$, it holds that
\begin{equation}\label{eq:XiStarCritical}
\nabla q(\xi^*)\cdot v =0.
\end{equation}
This implies $\nabla q(\xi^*)=t\xi^*$ for some $t\in\RR$. 
Using symmetry of $a$, we compute the gradient of $q$ at $\xi\in\RR^d$:
\begin{equation}\label{eq:gradhess}
  \nabla q(\xi) = 4a^T a\xi - 2(a\xi,\xi)(a+a^T)\xi = 4(a^2\xi-(a\xi,\xi)a\xi).
\end{equation}
It follows that $\xi^*$ is an eigenvector of $a^2-(a\xi^*,\xi^*)a$.
The spectral theorem provides an orthogonal decomposition $\RR^d=\oplus_{j=1}^J U_j$ into eigenspaces $U_j=E(a,\lambda_j)$ of $a$ with $1\leq J\leq d$ such that 
\begin{equation}\label{eq:orthodecomp}
  a = \textstyle\sum_{j=1}^{J} \lambda_j \pi_{U_j}\quad\text{and}\quad \lambda_1<\dots<\lambda_J.
\end{equation}
\textbf{Claim: }\textit{We claim that $\xi^*\in U_i$ for some $i\in\{1,\dots,J\}$ or 
$\xi^*\in (U_i\oplus U_j)\setminus (U_i\cup U_j)$ for distinct $i,\,j\in\{1,\dots,J\}$. In the latter case, it holds that}
\begin{equation}\label{eq:lilj}
  \lambda_i+\lambda_j = (a\xi^*,\xi^*).
\end{equation}
\textit{Proof of the claim.}
First, note that $\xi^*$ is an eigenvector of
\begin{equation}\label{eq:diagShapeAsquaredMinusa}
  a^2-(a\xi^*,\xi^*) a = \textstyle\sum_{j=1}^{J} (\lambda_j^2 - (a\xi^*,\xi^*) \lambda_j )\pi_{U_j}.
\end{equation}
Since \eqref{eq:diagShapeAsquaredMinusa} is in diagonal form, we know that for any $i\in\{1,\dots,J\}$,
\begin{equation}\label{eq:muI}
  \mu_i:=\lambda_i^2 - (a\xi^*,\xi^*) \lambda_i\text{  is an eigenvalue of $a^2-(a\xi^*,\xi^*) a$.}
\end{equation}
If there is an index $i\in\{1,\dots,J\}$ such that $\mu_i\ne\mu_j$ for all $j\in\{1,\dots,J\}$, then the eigenspace of $a^2-(a\xi^*,\xi^*) a$ associated with $\mu_i$ is $U_i$.
For two distinct indices $i,\,j\in\{1,\dots,J\}$, completing the square in \eqref{eq:muI} shows that
\begin{equation}\label{eq:squarelilj}
\begin{aligned}
    \mu_i=\mu_j \quad&\iff \quad \bigg(\lambda_i- \frac{(a\xi^* , \xi^*)}{2}\bigg)^2 = \bigg(\lambda_j - \frac{(a\xi^*,\xi^*)}{2}\bigg)^2\\
    &\iff\hspace{2.2cm} \lambda_i+\lambda_j = (a\xi^*,\xi^*),
\end{aligned}
\end{equation}
where for the last equivalence we used that $\lambda_i, \lambda_j$ are distinct and both have the same distance to ${(a\xi^*,\xi^*)}/{2}$ and so their average must be equal to ${(a\xi^*,\xi^*)}/{2}$.
In case \eqref{eq:squarelilj} is satisfied, the eigenspace of $a^2-(a\xi^*,\xi^*) a$ associated with $\mu_i$ is $U_i\oplus U_j$. 
To see this, note that \eqref{eq:squarelilj} proves that $\mu_i = \mu_j=\mu_k$ for pairwise distinct $i,\,j,\,k\in\{1,\dots,J\}$ cannot occur, since
then $\lambda_i,\lambda_j$ and $\lambda_k$ all have the same distance to ${(a\xi^*,\xi^*)}/{2}$, which implies that $\lambda_i,\lambda_j$
and $\lambda_k$ are not distinct and contradicts \eqref{eq:orthodecomp}.\\
Therefore, if $\xi^*$ is an eigenvector of $a^2-(a\xi^*,\xi^*) a$ with $\xi^*\in (U_i\oplus U_j)\setminus (U_i\cup U_j)$ for distinct $i,\,j\in\{1,\dots,J\}$,
it follows that $\mu_i=\mu_j$ and we obtain \eqref{eq:lilj}. This completes the proof of the claim.\medbreak

Our next goal is to determine possible value of $q(\xi^*)$, for which we distinguish two cases.
\Case{1}
If the critical point $\xi^*\in U_i$ for some $i\in\{1,\dots,J\}$, then it follows from $\xi^*\in S^{d-1}$ that
\begin{equation}\label{eq:pzerodefinite}
  q(\xi^*) = 2\lvert{a\xi^*}\rvert^2 -  (a\xi^*,\xi^*)^2 = \lambda_i^2.
\end{equation}
\Case{2}
However, if the critical point $\xi^*\in (U_i\oplus U_j)\setminus (U_i\cup U_j)$ for distinct $i,\,j\in\{1,\dots,J\}$,
there exist two \textit{orthogonal} vectors $\xi_i\in U_i\cap S^{d-1},\, \xi_j\in U_j\cap S^{d-1}$ and $r_i,r_j\in(0,1)$ such that
\begin{equation}\label{eq:eq_ri2rj2is1}
  \xi^*= r_i \xi_i + r_j \xi_j\quad\text{and}\quad r_i^2+r_j^2 =1.
\end{equation} 
Moreover, condition \eqref{eq:lilj} must be satisfied,
which together with \eqref{eq:eq_ri2rj2is1} yields the constraint
\begin{equation}\label{eq:axixi3}
  \lambda_i+\lambda_j = r_i^2\lambda_i + r_j^2 \lambda_j,
\end{equation}
where we used that $\xi_i$ and $\xi_j$ are orthogonal.
By \eqref{eq:orthodecomp}, constraint \eqref{eq:axixi3} implies that $\lambda_i$ and $\lambda_j$ are nonzero and have different signs.
Therefore, the second case does not occur if $a$ is positive semidefinite or negative semidefinite. 
Without loss of generality, we may assume 
\begin{equation}\label{eq:lambdailesslambdaj}
  \lambda_i < 0 < \lambda_j.
\end{equation}
Then, the constraints \eqref{eq:eq_ri2rj2is1} and \eqref{eq:axixi3} translate to
\begin{equation}\label{eq:rirjformula}
  r_i^2 = \frac{-\lambda_i}{\lambda_j-\lambda_i},\qquad\qquad r_j^2 = \frac{\lambda_j}{\lambda_j-\lambda_i}. 
\end{equation}
With this, we plug $\xi^*=r_i \xi_i + r_j \xi_j$ into $q$ and use $(a-b)(a^2+ab+b^2)=a^3-b^3$ to compute
\begin{equation}\label{eq:pzeroindefinite}
  q(\xi^*)  =  2(\lambda_i^2 r_i^2+\lambda_j^2 r_j^2)- (\lambda_i+\lambda_j)^2 = \lambda_i^2+\lambda_j^2.
\end{equation}

\Step{3: Solving the variational problem}
In order to characterize maximizers of $q|_{S^{d-1}}$, we need to rule out non-maximizing critical points.
If $a$ is positive semidefinite, it follows that any maximizer $\xi^*$ lies in $U_i$ for some $i\in\{1,\dots,J\}$
and by \eqref{eq:pzerodefinite}, we obtain
\begin{equation}
  g_{\cc}(a)=\max_{S^{d-1}} q = \max_{1\leq i\leq J} \lambda_i^2 = \lambda_+^2,
\end{equation}
which is attained in 
\begin{equation}
  S_{\cc}(a) = \argmax_{S^{d-1}} q = U_+\cap\Sdm.
\end{equation}
This proves \eqref{eq:posdefiniteS0}. If $a$ is negative semidefinite, we obtain \eqref{eq:negdefiniteS0} by the same argument.\\
In contrast, if $a$ is indefinite, maximizing \eqref{eq:pzeroindefinite} for $i,j\in\{1,\dots,J\}$ subject to \eqref{eq:lambdailesslambdaj} yields 
\begin{equation}
  g_{\cc}(a)=\max_{S^{d-1}} q = \lambda_-^2 + \lambda_+^2,
\end{equation}
which, by \eqref{eq:pzeroindefinite}, is attained in 
\begin{equation}
  S_{\cc}(a) = \argmax_{S^{d-1}} q = r_-(U_-\cap\Sdm) + r_+(U_+\cap\Sdm).
\end{equation}
This observation proves \eqref{eq:indefiniteS0}. 
Finally, the characterization \eqref{eq:curlcurlequicomp} of $E_{\cc}$ follows from \eqref{eq:posdefiniteS0} and \eqref{eq:negdefiniteS0}.
Note that indefinite matrices $a\in\rddsym$ are not equicompatible, since 
\eqref{eq:indefiniteS0} implies
\begin{equation}
  \dim S_{\cc}(a)=\dim U_- + \dim U_+ - 2 < d-1,
\end{equation}
which proves that $S_{\cc}(a)\subsetneq \Sdm$.
\end{proof}

Given $a\in\rddsym$, we obtain from \autoref{lem:CompProjCC} and \autoref{prop:optlaminationdirectioncurlcurl} explicit formulae for $p_{\cc,a}$ and its zero set $S_{\cc}(a)$.\\
In order to apply \autoref{thm:incompa}~\textcolor{blue}{(ii)} in the case $d=2$, we need to determine the maximal vanishing order $L_{\cc}(a)$ of $p_{\cc,a}$, which is
the focus of the next proposition.
We emphasize that all other conditions of \autoref{thm:incompa}~\textcolor{blue}{(ii)} have already been verified for $d=2$ and $\cA=\ccurl$ if $a\not\in\RR I_2$.
\begin{proposition}\label{prop:maxVanOrderCC}
Let $d=2$. Consider the differential operator $\cA=\ccurl$; see \eqref{eq:curlcurl}.
Given $a\in\RR^{2\times2}_{sym}\setminus\RR I_2$, let $L_{\cc}(a)$ be given as in \autoref{def:LcAa}.
Then, it holds that
\begin{equation}\label{eq:maxVanOrderCC}
L_{\cc}(a) = 
      \begin{cases}
        2 &\text{if } \rank a=1,\\
        1 &\text{if } \rank a=2.
      \end{cases}
\end{equation}
\end{proposition}
\begin{proof}
We need to show that the maximal vanishing order of the function $p_{\cc,a}$, introduced in \autoref{def:LcAa}, is given by the right-hand side of \eqref{eq:maxVanOrderCC}.
To do so, we use the formulae \eqref{eq:pcAaForumla2} and \eqref{eq:CompProjCurlCurlFormula}, which yield
\begin{equation}\label{eq:pcAaForumla3}
  p_{\cc,a}(\xi) = g_{\cc}(a)- \big(2\lvert{a\xi}\rvert^2 - (a\xi,\xi)^2\big)\quad\forall\xi\in S^{d-1},
\end{equation}
where $g_{\cc}(a)$ is explicitly given by \autoref{prop:optlaminationdirectioncurlcurl}.\\
Before we enter the proof, note that the spectral theorem and the assumption $a\in\RR^{2\times2}_{sym}\setminus\RR I_2$ imply 
that there is an orthonormal basis of eigenvectors $(\xi_-,\xi_+)\subset\RR^2$ of $a$ such that
\begin{equation}\label{eq:VanOrderCCSemDefFormOfa}
  a=\lambda_-\xi_-\otimes\xi_-+\lambda_+\xi_+\otimes\xi_+ \text{ with eigenvalues } \lambda_-<\lambda_+.
\end{equation}
The proof is now organized into two steps. First, we study the situation for (positive and negative) semidefinite matrices $a$. Then, we consider the indefinite setting.
\Step{1: The maximal vanishing order for semidefinite matrices}
Suppose that $a$ is semidefinite.
Without loss of generality, we may assume that $a$ is positive semidefinite. 
In case $a$ is negative semidefinite, consider $-a$ instead and note that $g_{\cc}(-a)=g_{\cc}(a)$, which by \eqref{eq:pcAaForumla3} yields $p_{\cc,-a}=p_{\cc,a}$.\medbreak

Due to \autoref{prop:optlaminationdirectioncurlcurl} and \eqref{eq:VanOrderCCSemDefFormOfa}, we have $g_{\cc}(a)=\lambda_+^2$ and $S_{\cc}(a)=\{\pm\xi_+\}$.
The maximal vanishing order is determined by the local behavior of $p_{\cc,a}$ close to its zero set $S_{\cc}(a)$.
We now parametrize points $\xi\in S^1$ close to $S_{\cc}(a)$ by 
\begin{equation}\label{eq:ParametrizeClosetoSccSemidef}
  \xi=\frac{1}{\sqrt{1+\rho^2}}(\xi_0+\rho \nu) \text{ with } \xi_0\in\{\pm\xi_+\},\, \nu\in\{\pm\xi_-\} \text{ and } \rho\in(0,1),
\end{equation}
where we are using that $\xi_-$ is a unit vector spanning the tangent space of $S^1$ at both $\pm\xi_+$.
Applying \eqref{eq:LCPSphereDistanceProjection}, we see that, for all $\xi\in S^1$ as in \eqref{eq:ParametrizeClosetoSccSemidef}, there holds
\begin{equation}\label{eq:distanceCircleScaling}
  \dist^2(\xi,S_{\cc}(a)) \sim \rho^2.
\end{equation}
Now, let $\xi\in S^1$ be parametrized as in \eqref{eq:ParametrizeClosetoSccSemidef}. To determine $p_{\cc,a}$ using \eqref{eq:pcAaForumla3}, we calculate
\begin{align}\label{eq:absaxi}
  \abs{a\xi}^2 = \frac{1}{1+\rho^2}\left(\abs{a\xi_0}^2 +2\rho(a\xi_0, a\nu) + \rho^2\abs{a\nu}^2\right)
  = \frac{1}{1+\rho^2} \left(\lambda^2_{+}+\rho^2\lambda_-^2\right)
\end{align}
by observing $(a\xi_0, a\nu)=\lambda_+\lambda_-(\xi_0,\nu)=0$. Similarly, we find
\begin{align}\label{eq:axixi}
  (a\xi,\xi) = \frac{1}{1+\rho^2} \left( (a\xi_0,\xi_0) + 2 \rho(a\xi_0,\nu) +\rho^2 (a\nu,\nu) \right)
  = \frac{1}{1+\rho^2} \left( \lambda_{+} +\rho^2 \lambda_- \right),
\end{align}
which yields
\begin{multline}
2\lvert{a\xi}\rvert^2 - (a\xi,\xi)^2 
  = \frac{2}{1+\rho^2} \left(\lambda^2_{+}+\rho^2\lambda_-^2  \right)
  - \frac{1}{(1+\rho^2)^2}\left(  \lambda_{+}^2 +2\lambda_{+}\rho^2  \lambda_- + \rho^4 \lambda_-^2\right)\\
  = \frac{1+2\rho^2}{(1+\rho^2)^2}\lambda_{+}^2 + \frac{2\rho^2}{(1+\rho^2)^2}\left(\lambda_-^2-\lambda_{+}\lambda_- \right)
  +\frac{\rho^4}{(1+\rho^2)^2}\lambda_-^2.
\end{multline}
Since $g_{\cc}(a)=\lambda_+^2$, we infer from \eqref{eq:pcAaForumla3} that
\begin{equation}\label{eq:vanishorderkeyestimate}
    p_{\cc,a}(\xi)= \frac{2\rho^2}{(1+\rho^2)^2} \left(\lambda_+  - \lambda_- \right)\lambda_- +\frac{\rho^4}{(1+\rho^2)^2}\left(\lambda_+^2 -\lambda_-^2\right).
\end{equation}
As $a$ is positive semidefinite, we have $0\leq \lambda_-<\lambda_+$.
If $\rank a=1$, there holds that $\lambda_-=0$ and $p_{\cc,a}(\xi)\sim\rho^4$. 
In that case, it follows by \autoref{def:maxvanishorderdef} and \eqref{eq:distanceCircleScaling} that $L_{\cc}(a)=2$.
In contrast, if $\rank a=2$, we have $\lambda_->0$ and $p_{\cc,a}(\xi)\sim\rho^2$, which proves $L_{\cc}(a)=1$.

\Step{2: The maximal vanishing order for indefinite matrices}
Next, we assume that $a$ is indefinite. Proceeding to use the notation introduced in \eqref{eq:VanOrderCCSemDefFormOfa},
we have $\lambda_-<0<\lambda_+$.
By \autoref{prop:optlaminationdirectioncurlcurl}, it holds that
\begin{equation}\label{eq:gccsccrecall}
  g_{\cc}(a) = \lambda_{-}^2 + \lambda_{+}^2, \qquad \qquad  S_{\cc}(a)= \{r_- x + r_+ y: x=\pm\xi_-\text{ and }y=\pm\xi_+\},
\end{equation}
where $r_{-},\,r_{+}\in(0,1)$ are determined by
\begin{equation}\label{eq:rminusrplus}
  r_{-}^2 = \frac{-\lambda_{-}}{\lambda_{+}-\lambda_{-}},\qquad\qquad r_{+}^2 = \frac{\lambda_{+}}{\lambda_{+}-\lambda_{-}}. 
\end{equation}
Note that the tangent space of $S^1$ at $\xi_0=r_- x + r_+ y\in S_{\cc}(a)$ is spanned by the orthogonal unit vector $\nu=r_- y - r_+ x$.
As $S_{\cc}(a)$ only consists of four points, there exists $\rho_0\in(0,1)$ such that points $\xi\in S^1$ in proximity of $S_{\cc}(a)$ can be uniquely parametrized by
\begin{equation}\label{eq:ParametrizeClosetoSccIndef}
  \xi=\frac{1}{\sqrt{1+\rho^2}}(\xi_0+\rho \nu) \text{ with } \xi_0=r_- x + r_+ y\in S_{\cc}(a),\, \nu = r_- y - r_+ x \text{ and } \rho\in(0,\rho_0).
\end{equation}
As in the first step of the proof, we compute $p_{\cc,a}$ near $S_{\cc}(a)$ using formula \eqref{eq:pcAaForumla3}.
Let $\xi\in S^1$ be as in \eqref{eq:ParametrizeClosetoSccIndef}.
From \eqref{eq:absaxi} and \eqref{eq:axixi}, we know 
\begin{equation}\label{eq:axi0axi0xi0}
  \begin{aligned}
    \abs{a\xi}^2 &= \frac{1}{1+\rho^2}\left(\abs{a\xi_0}^2 +2\rho(a\xi_0, a\nu) + \rho^2\abs{a\nu}^2\right),\\
    (a\xi,\xi) &= \frac{1}{1+\rho^2} \Big( (a\xi_0,\xi_0) + 2 \rho(a\xi_0,\nu) +\rho^2 (a\nu,\nu) \Big).
  \end{aligned}
\end{equation}
By \eqref{eq:rminusrplus} and \eqref{eq:ParametrizeClosetoSccIndef}, it follows that
\begin{equation}\label{eq:axixi2}
  \begin{aligned}
    \abs{a\xi_0}^2 &= \lambda_-^2 r_-^2 + \lambda_+^2 r_+^2 = \frac{\lambda_+^3-\lambda_-^3}{\lambda_{+}-\lambda_{-}} =\lambda_+^2 + \lambda_+ \lambda_- + \lambda_-^2,\\
    (a\xi_0,\xi_0) &= \lambda_- r_-^2 +\lambda_+ r_+^2 = \frac{\lambda_+^2-\lambda_-^2}{\lambda_{+}-\lambda_{-}}= \lambda_-+\lambda_+.
  \end{aligned}
\end{equation}
With \eqref{eq:axi0axi0xi0} and \eqref{eq:axixi2}, we obtain
\begin{equation}
  \begin{aligned}
    2\lvert{a\xi}\rvert^2 - &(a\xi,\xi)^2 \\
    &= \frac{2}{1+\rho^2}\left(\lambda_+^2 + \lambda_+ \lambda_- + \lambda_-^2 +2\rho(a\xi_0, a\nu) + \rho^2\abs{a\nu}^2\right)\\
    &\hspace{40pt}- \frac{1}{(1+\rho^2)^2} \Big( \lambda_-+\lambda_+ + 2 \rho(a\xi_0,\nu) +\rho^2 (a\nu,\nu) \Big)^2\\
    &= \frac{2}{1+\rho^2}\left(\lambda_+^2 + \lambda_+ \lambda_- + \lambda_-^2 \right)
    - \frac{1}{(1+\rho^2)^2} \left( \lambda_-+\lambda_+  \right)^2 + \frac{2}{1+\rho^2}\left(2\rho(a\xi_0, a\nu) + \rho^2\abs{a\nu}^2\right)\\
    &\hspace{40pt}- \frac{1}{(1+\rho^2)^2} \left(   2\left( \lambda_-+\lambda_+  \right)\left[ 2 \rho(a\xi_0,\nu) +\rho^2 (a\nu,\nu) \right]    +\left[ 2 \rho(a\xi_0,\nu) +\rho^2 (a\nu,\nu) \right]^2\right).
  \end{aligned}
\end{equation}
Bringing the fractions to a common denominator gives
\begin{equation}
\begin{aligned}
  2\lvert{a\xi}\rvert^2 - &(a\xi,\xi)^2 \\
  &= \frac{1}{(1+\rho^2)^2}\Big[ 2\left(\lambda_+^2 + \lambda_+ \lambda_- + \lambda_-^2 \right)- \left( \lambda_-+\lambda_+  \right)^2\Big]
    +\frac{2\rho^2}{(1+\rho^2)^2} (\lambda_+^2 + \lambda_+ \lambda_- + \lambda_-^2 )\\
    &\hspace{40pt}+ \frac{1}{(1+\rho^2)^2}\left(4\rho(a\xi_0, a\nu) + 2\rho^2\abs{a\nu}^2\right) +\frac{1}{(1+\rho^2)^2}\left(4\rho^3(a\xi_0, a\nu) + 2\rho^4\abs{a\nu}^2\right)  \\
    &\hspace{40pt}- \frac{1}{(1+\rho^2)^2} \left(   2\left( \lambda_-+\lambda_+  \right)\left[ 2 \rho(a\xi_0,\nu) +\rho^2 (a\nu,\nu) \right]    +\left[ 2 \rho(a\xi_0,\nu) +\rho^2 (a\nu,\nu) \right]^2\right).
\end{aligned}
\end{equation}
The square in the last expression is
\begin{equation}
  \left[ 2 \rho(a\xi_0,\nu) +\rho^2 (a\nu,\nu) \right]^2 = 4\rho^2(a\xi_0,\nu)^2 + 4\rho^3(a\xi_0,\nu)(a\nu,\nu) + \rho^4(a\nu,\nu)^2.
\end{equation}
Plugging this in and sorting by order of $\rho$, we find
\begin{equation}
\begin{aligned}
  2\lvert{a\xi}\rvert^2 - &(a\xi,\xi)^2 \\
  &= \frac{1}{(1+\rho^2)^2}(\lambda_+^2 + \lambda_-^2)+ \frac{\rho}{(1+\rho^2)^2}\Big(4(a\xi_0, a\nu)-  4\left( \lambda_-+\lambda_+  \right)  (a\xi_0,\nu)\Big) \\
    &\hspace{40pt} +\frac{\rho^2}{(1+\rho^2)^2} \Big(2(\lambda_+^2 + \lambda_+ \lambda_- + \lambda_-^2 ) + 2\abs{a\nu}^2-  2\left( \lambda_-+\lambda_+  \right)  (a\nu,\nu) -4(a\xi_0,\nu)^2\Big)\\
    &\hspace{40pt}+\frac{\rho^3}{(1+\rho^2)^2}\Big(4(a\xi_0, a\nu) - 4(a\xi_0,\nu)(a\nu,\nu) \Big)  + \frac{\rho^4}{(1+\rho^2)^2} \left( 2\abs{a\nu}^2  - (a\nu,\nu)^2\right).
\end{aligned}
\end{equation}
By \eqref{eq:pcAaForumla3} and \eqref{eq:gccsccrecall}, it follows that
\begin{equation}\label{eq:finalorderformula}
\begin{aligned}
    p_{\cc,a}(\xi) &=\frac{2\rho^2+\rho^4}{(1+\rho^2)^2}(\lambda_+^2 + \lambda_-^2)
    + \frac{4\rho}{(1+\rho^2)^2}\Big( \left( \lambda_-+\lambda_+  \right)(a\xi_0,\nu)    -   (a\xi_0, a\nu)\Big) \\
    &\hspace{40pt}+\frac{2\rho^2}{(1+\rho^2)^2} \Big(   \left( \lambda_-+\lambda_+  \right)  (a\nu,\nu) +2(a\xi_0,\nu)^2
                                          -(\lambda_+^2 + \lambda_+ \lambda_- + \lambda_-^2 )  - \abs{a\nu}^2  \Big)\\
    &\hspace{40pt}+\frac{4\rho^3}{(1+\rho^2)^2}\Big(    (a\xi_0,\nu)(a\nu,\nu)   -    (a\xi_0, a\nu)  \Big)  
    + \frac{\rho^4}{(1+\rho^2)^2} \left( (a\nu,\nu)^2- 2\abs{a\nu}^2 \right).
\end{aligned}
\end{equation}
Since $\xi_0$ is given as in \eqref{eq:ParametrizeClosetoSccIndef}, a quick computation shows that
\begin{equation}\label{eq:eigvec}
\begin{aligned}
    a^2\xi_0-(\lambda_-+\lambda_+)a\xi_0 &= \lambda_-^2r_-x+\lambda_+^2r_+y    -    (\lambda_-+\lambda_+)(\lambda_- r_-x+\lambda_+r_+y)\\
    &= -\lambda_-\lambda_+ \xi_0,
\end{aligned}
\end{equation}
Together with $\nu\perp\xi_0$, this implies that the first order term in \eqref{eq:finalorderformula} vanishes
\begin{equation}\label{eq:firstorderterm}
  \left( \lambda_-+\lambda_+  \right)(a\xi_0,\nu) - (a\xi_0, a\nu) = \Big((\lambda_-+\lambda_+)a  \xi_0- a^2 \xi_0,\,\nu\Big) = \lambda_-\lambda_+ (\xi_0,\nu)=0.
\end{equation}
Next, note that the second order term in \eqref{eq:finalorderformula} is
\begin{equation}\label{eq:secondorderterm}
\begin{aligned}
  2(\lambda_+^2 + \lambda_-^2) + 2\left( \lambda_-+\lambda_+  \right)  (a\nu,\nu)  +4(a\xi_0,\nu)^2
  -2(\lambda_+^2 + \lambda_+ \lambda_- + \lambda_-^2 )  - 2\abs{a\nu}^2\\
  =2\left( \lambda_-+\lambda_+  \right) (a\nu,\nu)     +4 (a\xi_0,\nu)^2 
  -2\lambda_+ \lambda_-  - 2\abs{a\nu}^2.
\end{aligned}
\end{equation}
We use \eqref{eq:rminusrplus} and the explicit form of $\xi_0$ and $\nu$ in \eqref{eq:ParametrizeClosetoSccIndef} to compute
\begin{equation}
\begin{aligned}
    &(a\nu,\nu) = \lambda_+r_-^2 +\lambda_-r_+^2=0,\\
    &(a\xi_0,\nu) = \lambda_+r_-r_+ -\lambda_-r_-r_+ =\sqrt{-\lambda_-\lambda_+},\\
    &\abs{a\nu}^2 = \lambda_+^2r_-^2 +\lambda_-^2r_+^2  =-\lambda_-\lambda_+.
\end{aligned}
\end{equation}
For all $\xi_0$ and $\nu$ as in \eqref{eq:ParametrizeClosetoSccIndef}, we obtain 
\begin{equation}\label{eq:IndefEstVanOrder}
\left( \lambda_-+\lambda_+  \right) (a\nu,\nu)     +2 (a\xi_0,\nu)^2      -\lambda_+ \lambda_-  - \abs{a\nu}^2 = -2\lambda_-\lambda_+  >0.
\end{equation}
As a consequence of \eqref{eq:finalorderformula}, \eqref{eq:secondorderterm} and \eqref{eq:IndefEstVanOrder}, it holds that $L_{\cc}(a)=1$ if $a$ is indefinite.
Here, we again used that $\dist^2(\xi,S_{\cc}(a))\sim\rho^2$ for all $\xi\in  S^1$ as in \eqref{eq:ParametrizeClosetoSccIndef}.
\end{proof}
With \autoref{prop:maxVanOrderCC}, we are ready to prove the first and second statements of \autoref{thm:incompsyma}.
\begin{proof}[Proof of \autoref{thm:incompsyma}~\textcolor{blue}{(i+ii)}]
Let $d=2$.
By the \autoref{lem:CompProjCC}, we know that the differential operator $\ccurl$ has constant rank and spanning wave cone. Hence, \autoref{thm:incompsyma}~\textcolor{blue}{(i)}
follows by applying \autoref{thm:incompa}~\textcolor{blue}{(i)} with $\cA=\ccurl$.\medbreak

In order to prove \autoref{thm:incompsyma}~\textcolor{blue}{(ii)}, let $a=a_1-a_0\in\RR^{2\times2}_{sym}$. Suppose that $a\not\in E_{\cc}=\RR I_2$; see \autoref{prop:optlaminationdirectioncurlcurl}.
Then, the requirements for \autoref{thm:incompa}~\textcolor{blue}{(ii)} are satisfied because $S_{\cc}(a)$ is contained in a union of up to two one-dimensional subspaces of $\RR^2$
(see \autoref{prop:optlaminationdirectioncurlcurl} and \autoref{rmk:OptLamCCurlInTwoDim}) and the maximal vanishing order is characterized as in \autoref{prop:maxVanOrderCC}.
The lower bound \autoref{thm:incompsyma}~\textcolor{blue}{(ii)} follows.
\end{proof}

\pagebreak
\section{Upper bounds for the singularly perturbed two-well energies for the gradient and the divergence}\label{sec:UpperBoundGradDiv}
In \autoref{subsec:ApplicationOfGeneralLowerScaling}, we established $\eps^{\nicefrac{2}{3}}$-lower scaling bounds for the zeroth-order-corrected singularly perturbed two-well energies
for the gradient and the divergence.
The goal of this section is to prove \autoref{thm:incompgrad}~\textcolor{blue}{(iii)} by deriving the matching upper bounds.\smallbreak

Our approach proceeds via an intermediate result, \autoref{prop:a4}, which serves as a simplified and more tractable formulation of the theorem.
The argument is then divided into two parts:
In \autoref{subsec:ChangeOfVar}, we begin by introducing \autoref{prop:a4} and show how it implies \autoref{thm:incompgrad}~\textcolor{blue}{(iii)} by means of a suitable change of variables. The remainder of the argument, presented in \autoref{subsec:branchingConstruction}, is devoted to the proof of \autoref{prop:a4} and relies on a branching construction.
\subsection{Reduction to a simplified setting}\label{subsec:ChangeOfVar}
In this section, we show that for the proof of \autoref{thm:incompgrad}~\textcolor{blue}{(iii)}, we may assume without loss of generality that $\cA=\curl$ and that the coordinate direction $e_1\in\RR^2$ is an optimal lamination direction.
Under these assumptions, \autoref{thm:incompgrad}~\textcolor{blue}{(iii)} reduces to the following simplified setting.

\begin{breakproposition}[Upper scaling bound of the two-well energy for the gradient]\label{prop:a4}
Let $\cK=\{a_0,a_1\}\subset\RR^{2\times2}$ and $F\in\RR^{2\times2}$.
For the differential operator $\cA=\curl$ given by \eqref{eq:curl},
consider the energy $E^{\curl}_\eps(F,\cK)$ in the domain $Q=(0,1)^2$ given by \eqref{eq:MinSingPertAfreeEnergy}.
For $a:=a_1-a_0$, let $S_{\curl}(a)$ be the set of optimal lamination directions as in \autoref{def:optLaminationDirections}.
Let $\ttheta_{\curl}=\ttheta_{\curl}(F,\cK)$ be given as in \autoref{prop:optvolfrac}.
Suppose that $\ttheta_{\curl}\in(0,1)$ and that $e_1\in S_{\curl}(a)$.\\
Then, for all $\eps\in(0,1)$, there exist a vector field $v_\eps\in W^{1,\infty}_0(Q;\RR^2)$ and a phase arrangement
$\chi_\eps\in BV(Q; \cK)$ such that, for $u_\eps := F + \nabla v_\eps\in\cD_F^{\curl}(Q)$, the following energy estimate holds:
\begin{equation}\label{eq:a4Estimate}
  E^{\curl}_\eps(u_\eps,\chi_\eps) - E^{\curl}_0(F,\cK) \leq C\eps^{\nicefrac{2}{3}}\quad\forall \eps\in (0,1)
\end{equation}
for some constant $C=C(a,\ttheta_{\curl})>0$.
\end{breakproposition}
We postpone the proof of \autoref{prop:a4} to \autoref{subsec:branchingConstruction}.
Assuming the validity of \autoref{prop:a4}, we now prove \autoref{thm:incompgrad}~\textcolor{blue}{(iii)}.
\begin{proof}[Proof of \autoref{thm:incompgrad}~\textcolor{blue}{(iii)}]
We begin by recollecting the assumptions of \autoref{thm:incompgrad}~\textcolor{blue}{(iii)}:
Let $\cK=\{a_0,a_1\}\subset\RR^{2\times 2}$ and $F\in\RR^{2\times 2}$. Denote by $a:=a_1-a_0$.
Let $\cA\in\{\curl,\Div\}$ and $\xi^*\in S_{\cA}(a)$.
Suppose that $\ttheta_{\cA}=\ttheta_{\cA}(F,\cK)\in(0,1)$ and that $\Omega\subset\RR^2$ is a rotated unit square with two faces normal to $\xi^*$.\medbreak

The proof is now organized into two steps. In the first step, we show the $\eps^{\nicefrac{2}{3}}$-upper bound for the differential operator $\curl$. In the second step, we examine the upper bound for the divergence.
\Step{1: The upper bound for the gradient two-well energy}
Suppose that $\cA=\curl$. Using a change of variables, we will derive an equivalent gradient two-well problem with data $(F',\cK')$ in the standard unit square and optimal lamination direction $e_1$. This allows us to apply \autoref{prop:a4} and prove the energy scaling for the original problem with data $(F,\cK)$ in $\Omega$.\medbreak
Let $R\in SO(2)$ be a rotation with 
\begin{equation}\label{eq:Re1xiStar}
  Re_1 =\xi^*.
\end{equation}
Since $E^{\curl}_\eps(F,\cK;\Omega)$ is invariant under translations of $\Omega$ for all $\eps\geq 0$, we may assume without loss of generality that
\begin{equation}\label{eq:ChangeofVarDomain}
  \Omega = R Q,\text{ where } Q=(0,1)^2.
\end{equation}
Now, let $F':=FR$ and $\cK':=\{a_0',a_1'\}$ with $a_j':=a_j R$. Let $a':=a_1'-a_0'$.
By \autoref{lem:CompProjCurl}, we find
\begin{equation}
\abs{\ppc(\xi)a'}^2=\abs{a'\xi}^2 = \abs{aR\xi}^2 \quad\forall\xi\in S^1,
\end{equation}
which, recalling \autoref{def:compquant} and \autoref{def:optLaminationDirections}, implies
\begin{equation}\label{eq:ChangeofVarCompQuantOptLam}
  h_{\curl}(a') = h_{\curl}(a) \text{ and } R S_{\curl}(a')=S_{\curl}(a).
\end{equation}
Since $\xi^*\in S_{\curl}(a)$, we obtain $e_1\in S_{\curl}(a')$ from \eqref{eq:Re1xiStar}.
Another consequence of \eqref{eq:ChangeofVarCompQuantOptLam} is that
\begin{equation}\label{eq:ChangeofVarQuasiConv}
    \abs{F'-a'_\theta}^2+\theta(1-\theta)h_{\curl}(a') = \abs{F-a_\theta}^2+\theta(1-\theta)h_{\curl}(a)\quad\forall\theta\in[0,1],
\end{equation}
where $a_\theta'$ and $a_\theta$ are defined as in \eqref{eq:atheta}.
Here, we used that the Frobenius norm has the property
\begin{equation}\label{eq:FrobeniusNormProp}
    \abs{AR}=\abs{A} \quad\forall A\in\rddd{2},\, R\in SO(2).
\end{equation}
Using \autoref{thm:qwdom}, \autoref{prop:optvolfrac}, \eqref{eq:ChangeofVarDomain} and \eqref{eq:ChangeofVarQuasiConv}, it follows that
\begin{equation}\label{eq:ChangeofVarOptVolFrac}
  E^{\curl}_0(F',\cK';Q)=E^{\curl}_0(F,\cK;\Omega), \hspace{1.5cm}  \ttheta_{\curl}(F',\cK') = \ttheta_{\curl}(F,\cK) \in(0,1),
\end{equation}
where we indicate the domain in our notation of the energies.\medbreak
Applying \autoref{prop:a4} for the data $(F',\cK')$, we obtain, for each $\eps\in(0,1)$, a vector field $v'_\eps\in W^{1,\infty}_0(Q;\RR^2)$ and a phase arrangement
$\chi'_\eps\in BV(Q; \cK')$ such that $u'_\eps:=F'+\nabla v'_\eps\in\cD_{F'}^{\curl}(Q)$ satisfies
\begin{equation}\label{eq:ChangeofVarScaling}
  E^{\curl}_\eps(u'_\eps,\chi'_\eps;Q) - E^{\curl}_0(F',\cK';Q) \leq C\eps^{\nicefrac{2}{3}}\quad\forall \eps\in (0,1)
\end{equation}
for some constant $C=C(a',\ttheta_{\curl}(F',\cK'))=C(a',\ttheta_{\curl}(F,\cK))>0$.\medbreak

For any $\eps\in(0,1)$, we now define $v_\eps\in W^{1,\infty}_0(\Omega;\RR^d)$ and $\chi_\eps\in BV(\Omega; \cK)$ by a change of variables, setting, for $x\in\Omega$,
\begin{equation}\label{eq:NonObjectiveVectorCurl}
    v_\eps(x):=v'_\eps(R^Tx),\hspace{2cm} \chi_\eps(x):=\chi'_\eps(R^Tx) R^T.
\end{equation}
We introduce $u_\eps:=F+\nabla v_\eps\in\cD_F^{\curl}(\Omega)$ for $\eps\in(0,1)$ and note that
\begin{equation}\label{eq:NonObjectiveTensorCurl}
  u_\eps(x)=u'_\eps(R^Tx)R^T\quad\forall x\in\Omega.
\end{equation}
It follows that
\begin{equation}\label{eq:ChangeofVarScaling2}
    \begin{aligned}
      E^{\curl}_\eps(u_\eps,\chi_\eps;\Omega)&= \int_\Omega \abs{u_\eps(x)-\chi_\eps(x)}^2\,dx + \eps \norm{\nabla \chi_\eps}_{TV(\Omega)}\\
      &= \int_Q \abs{u'_\eps(y)R^T-\chi'_\eps(y)R^T}^2\,dy + \eps \norm{\nabla \chi_\eps'}_{TV(Q)} = E^{\curl}_\eps(u'_\eps,\chi'_\eps;Q)\quad\forall \eps\in (0,1),
    \end{aligned}
\end{equation}
where we again used \eqref{eq:FrobeniusNormProp}.
This together with \eqref{eq:ChangeofVarOptVolFrac} and \eqref{eq:ChangeofVarScaling} implies
\begin{equation}\label{eq:ChangeofVarScaling3}
  E^{\curl}_\eps(u_\eps,\chi_\eps;\Omega) - E^{\curl}_0(F,\cK;\Omega) \leq C\eps^{\nicefrac{2}{3}}\quad\forall \eps\in (0,1).
\end{equation}
Since $a'$ only depends on $a$ and $\xi^*$, and any choice of $\xi^*\in S_{\curl}(a)$ provides an upper bound of this form,
it is justified to write $C=C(a,\ttheta_{\curl}(F,\cK))$.
This completes the proof for the gradient two-well energy.

\Step{2: The upper bound for the divergence two-well energy}
Suppose that $\cA=\Div$. Let $S\in SO(2)$ be defined by $S(x_1,x_2):=(-x_2,x_1)$ in $\RR^2$.\\
The strategy of this proof is to translate the two-well problem for the divergence into one for the gradient, using that $\Div(u)=\curl(u')$ for all $u,u'\in C^\infty(\Omega;\RR^{2\times 2})$ with $u(x)=u'(x)S$ in $\Omega$. In contrast to the first step of this proof, this corresponds to a change of variables in the target space of a matrix field $u\in\cD_F^{\Div}(\Omega)$ instead of a transformation of the domain $\Omega$.\medbreak

We recall the setting from the beginning of the proof, in particular, that $\xi^*\in S_{\Div}(a)$ and that $\Omega\subset\RR^2$ is a rotated unit square with two faces normal to $\xi^*$.
Denote by $\xi^*_\perp:=S\xi^*$.
Let $t_{min}=\lambda_{min}(a^Ta)$ and $t_{max}=\lambda_{max}(a^Ta)$ be the smallest and largest
eigenvalues of $a^Ta\in\RR^{2\times2}_{sym}$, respectively.
By \autoref{lem:EquicompatibleStatesDiv}, it holds that
\begin{equation}\label{eq:ChangeOfVarCompQuatDiv}
  h_{\Div}(a) = t_{min} \text{ and } S_{\Div}(a) = S^{d-1}\cap E(a^Ta;t_{min}).
\end{equation}
Since $\xi^*\in S_{\Div}(a)$, it follows that
\begin{equation}\label{eq:ChangeofVarSpectralDecomp}
  a^Ta = t_{min}\,\xi^*\otimes \xi^* + t_{max}\,\xi^*_\perp\otimes \xi^*_\perp.
\end{equation}
We now introduce a gradient two-well problem with data $(F',\cK')$, which is equivalent to the divergence problem with data $(F,\cK)$.
Let $F':=FS^T$ and $\cK':=\{a_0',a_1'\}$ with $a_j':=a_jS^T$. 
Setting $a':=a_1'-a_0'$, we compute
\begin{equation}\label{eq:aTaWithPrime}
  \begin{aligned}
    a'^{\,T}a' = S (a^T a) S^T
    &= t_{min}\,(S\xi^*)\otimes (S\xi^*) + t_{max}\,(S\xi^*_\perp)\otimes (S\xi^*_\perp)\\
    &= t_{min}\,\xi^*_\perp\otimes \xi^*_\perp + t_{max}\,\xi^*\otimes \xi^*.
  \end{aligned}
\end{equation}
Now, \autoref{lem:EquicompatibleStatesCurl} yields
\begin{equation}
  g_{\curl}(a') = t_{max}  \text{ and } S_{\curl}(a') = S^{d-1}\cap E(a'^{\,T}a';t_{max}).
\end{equation}
Since $\lvert a'\rvert^2=t_{min} + t_{max}$, we obtain from \eqref{eq:compquanteq} that
\begin{equation}\label{eq:ChangeofVarCompQuant2}
  h_{\curl}(a') = \lvert a' \rvert^2-g_{\curl}(a') = t_{min} = h_{\Div}(a).
\end{equation}
As in the first step of the proof, this shows that
\begin{equation}\label{eq:ChangeofVarOptVolFrac2}
  E^{\curl}_0(F',\cK')=E^{\Div}_0(F,\cK), \hspace{1.5cm}  \ttheta_{\curl}(F',\cK') = \ttheta_{\Div}(F,\cK) \in(0,1),
\end{equation}
where we consider the domain $\Omega$ for both energies.\\
Due to \eqref{eq:ChangeofVarSpectralDecomp} and \eqref{eq:aTaWithPrime}, we have $ E(a'^{\,T}a';t_{max}) = E(a^Ta;t_{min})$, which implies that
\begin{equation}\label{eq:ChangeofVarOptLam2}
  S_{\curl}(a') = S_{\Div}(a).
\end{equation}
In particular, we get $\xi^*\in S_{\curl}(a')$.
As we have already proved \autoref{thm:incompgrad}~\textcolor{blue}{(iii)} for the differential operator $\curl$ in the first step, we can apply 
it for the data $(F',\cK')$ in the domain $\Omega$. We obtain, for each $\eps\in(0,1)$, a map $u'_\eps\in\cD_{F'}^{\curl}(\Omega)$ and phase arrangement $\chi'_\eps\in BV(\Omega;\cK')$ with
\begin{equation}\label{eq:ChangeofVarScaling4}
  E^{\curl}_\eps(u'_\eps,\chi'_\eps) - E^{\curl}_0(F',\cK') \leq C\eps^{\nicefrac{2}{3}}\quad\forall \eps\in (0,1)
\end{equation}
for some constant $C=C(a',\ttheta_{\curl}(F',\cK'))>0$.\medbreak

We translate this into the setting of the divergence two-well problem by a change of variables in the target space of $u'_\eps$ and $\chi'_\eps$.
For any $\eps\in(0,1)$, we define the maps $u_\eps\in \cD_{F}^{\Div}(\Omega)$ and $\chi_\eps\in BV(\Omega; \cK)$ by setting, for $x\in\Omega$,
\begin{equation}
    u_\eps(x):=u'_\eps(x)S,\hspace{2cm} \chi_\eps(x):=\chi'_\eps(x)S.
\end{equation}
Using \eqref{eq:FrobeniusNormProp}, \eqref{eq:ChangeofVarOptVolFrac2} and \eqref{eq:ChangeofVarScaling4}, it follows as in the first step of the proof that
\begin{equation}\label{eq:ChangeofVarScaling5}
  E^{\Div}_\eps(u_\eps,\chi_\eps) - E^{\Div}_0(F,\cK) \leq C\eps^{\nicefrac{2}{3}}\quad\forall \eps\in (0,1).
\end{equation}
Since $a'$ only depends on $a$ and $\ttheta_{\curl}(F',\cK')=\ttheta_{\Div}(F,\cK)$, we may write $C=C(a,\ttheta_{\Div}(F,\cK))$, which completes the proof.
\end{proof}

\subsection{A branching construction for the incompatible setting}\label{subsec:branchingConstruction}
In this section, we prove \autoref{prop:a4}. We adapt the branching construction developed in \cite{km92, km94} to account for incompatible wells.
The approach adheres to the same recipe:
We begin with a unit cell construction (\autoref{lem:aa1}) that satisfies the boundary conditions on the lateral sides and enables frequency-doubling of the oscillations in the vertical direction. We then turn to a cut-off argument (\autoref{lem:aa2}).
Finally, we combine coarse-scale oscillations in the interior with fine-scale oscillations near the boundary (\autoref{prop:a3})
by iterating the unit-cell construction towards the top and bottom boundaries, where the cut-off is carried out.  \medbreak

We outline the approach to prove \autoref{prop:a4}, which is based on the construction anticipated in \autoref{rmk:compatibleApprox}.
Suppose that $F\in\RR^{2\times2}$ and $K=\{a_0,a_1\}\subset\RR^{2\times2}$ are incompatible; that is, they fail to satisfy \eqref{eq:compatiblewells} or \eqref{eq:compatiblebdrydata}
for $\cA=\curl$.
Then, the key idea is to use a compatible approximation $(F,\tcK)$ of the data $(F,\cK)$ (see \autoref{def:compatibleApprox}),
which, roughly speaking, allows us to split the singularly perturbed energy into two parts:
\begin{equation}\label{eq:EnergySplit}
  E^{\curl}_\eps(F,\cK) = E^{\curl}_0(F,\cK) +  E^{\curl}_\eps(F,\tcK)\quad\forall\eps\in(0,1).
\end{equation}
This is achieved using the upper bound construction for the compatible data $(F,\tcK)$, which also provides the bound 
\begin{equation}\label{eq:EnergySplit2}
  E^{\curl}_\eps(F,\tcK)\lesssim\eps^{\nicefrac{2}{3}}\quad\forall\eps\in(0,1)
\end{equation}
and together with \eqref{eq:EnergySplit} implies \autoref{prop:a4}.\medbreak
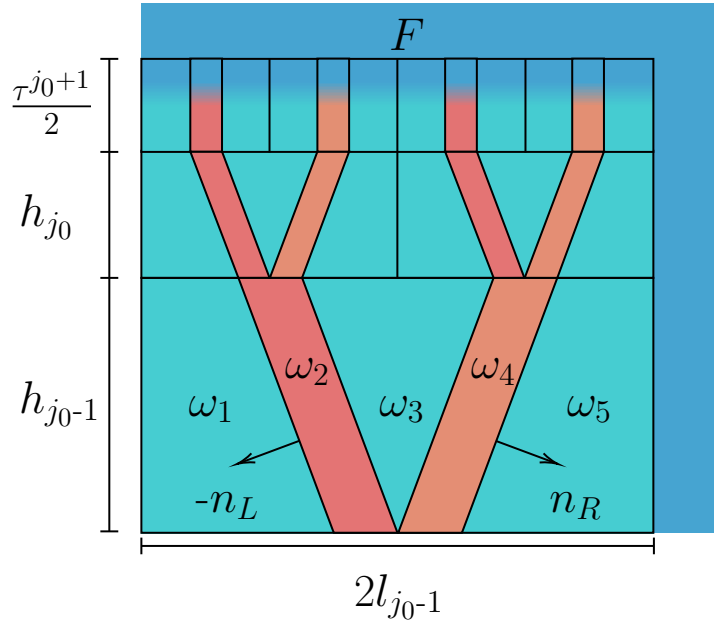
\begin{figure}
\begin{center}

  
\tikzset {_f2blwfhco/.code = {\pgfsetadditionalshadetransform{ \pgftransformshift{\pgfpoint{0 bp } { 0 bp }  }  \pgftransformrotate{-90 }  \pgftransformscale{2 }  }}}
\pgfdeclarehorizontalshading{_txxjm6u8d}{150bp}{rgb(0bp)=(0.27,0.64,0.82);
rgb(43.15383093697684bp)=(0.27,0.64,0.82);
rgb(49.70238276890346bp)=(0.27,0.8,0.82);
rgb(100bp)=(0.27,0.8,0.82)}

  
\tikzset {_ktc4321xs/.code = {\pgfsetadditionalshadetransform{ \pgftransformshift{\pgfpoint{0 bp } { 0 bp }  }  \pgftransformrotate{-90 }  \pgftransformscale{2 }  }}}
\pgfdeclarehorizontalshading{_48wupjldo}{150bp}{rgb(0bp)=(0.27,0.64,0.82);
rgb(43.360304151262554bp)=(0.27,0.64,0.82);
rgb(49.64192526681082bp)=(0.89,0.45,0.45);
rgb(100bp)=(0.89,0.45,0.45)}

  
\tikzset {_ehj43ygjy/.code = {\pgfsetadditionalshadetransform{ \pgftransformshift{\pgfpoint{0 bp } { 0 bp }  }  \pgftransformrotate{-90 }  \pgftransformscale{2 }  }}}
\pgfdeclarehorizontalshading{_4db7pcu1y}{150bp}{rgb(0bp)=(0.27,0.64,0.82);
rgb(43.360304151262554bp)=(0.27,0.64,0.82);
rgb(49.64192526681082bp)=(0.89,0.56,0.45);
rgb(100bp)=(0.89,0.56,0.45)}

  
\tikzset {_907yh3hwx/.code = {\pgfsetadditionalshadetransform{ \pgftransformshift{\pgfpoint{0 bp } { 0 bp }  }  \pgftransformrotate{-90 }  \pgftransformscale{2 }  }}}
\pgfdeclarehorizontalshading{_w1t22xsku}{150bp}{rgb(0bp)=(0.27,0.64,0.82);
rgb(43.360304151262554bp)=(0.27,0.64,0.82);
rgb(49.64192526681082bp)=(0.89,0.56,0.45);
rgb(100bp)=(0.89,0.56,0.45)}

  
\tikzset {_rrize0cjp/.code = {\pgfsetadditionalshadetransform{ \pgftransformshift{\pgfpoint{0 bp } { 0 bp }  }  \pgftransformrotate{-90 }  \pgftransformscale{2 }  }}}
\pgfdeclarehorizontalshading{_gk27dep2a}{150bp}{rgb(0bp)=(0.27,0.64,0.82);
rgb(43.360304151262554bp)=(0.27,0.64,0.82);
rgb(49.64192526681082bp)=(0.89,0.45,0.45);
rgb(100bp)=(0.89,0.45,0.45)}
\tikzset{every picture/.style={line width=0.75pt}}       

\begin{tikzpicture}[x=0.75pt,y=0.75pt,yscale=-0.75,xscale=0.8]
\draw  [draw opacity=0][fill={rgb, 255:red, 70; green, 164; blue, 209 }  ,fill opacity=1 ] (172.2,8.67) -- (529.23,8.67) -- (529.23,363.87) -- (172.2,363.87) -- cycle ;
\draw  [draw opacity=0][fill={rgb, 255:red, 70; green, 205; blue, 209 }  ,fill opacity=1 ] (172.2,108.44) -- (492,108.44) -- (492,364.15) -- (172.2,364.15) -- cycle ;
\draw   (172.2,192.93) -- (491.66,192.93) -- (491.66,363.87) -- (172.2,363.87) -- cycle ;
\draw    (492,372.7) -- (172.2,372.7) ;
\draw [shift={(172.2,372.7)}, rotate = 360] [color={rgb, 255:red, 0; green, 0; blue, 0 }  ][line width=0.75]    (0,5.59) -- (0,-5.59)   ;
\draw [shift={(492,372.7)}, rotate = 360] [color={rgb, 255:red, 0; green, 0; blue, 0 }  ][line width=0.75]    (0,5.59) -- (0,-5.59)   ;
\draw    (152.2,362.93) -- (152.2,192.93) ;
\draw [shift={(152.2,192.93)}, rotate = 90] [color={rgb, 255:red, 0; green, 0; blue, 0 }  ][line width=0.75]    (0,5.59) -- (0,-5.59)   ;
\draw [shift={(152.2,362.93)}, rotate = 90] [color={rgb, 255:red, 0; green, 0; blue, 0 }  ][line width=0.75]    (0,5.59) -- (0,-5.59)   ;
\draw    (392,301.7) -- (430.21,316.99) ;
\draw [shift={(432.07,317.73)}, rotate = 201.81] [color={rgb, 255:red, 0; green, 0; blue, 0 }  ][line width=0.75]    (10.93,-3.29) .. controls (6.95,-1.4) and (3.31,-0.3) .. (0,0) .. controls (3.31,0.3) and (6.95,1.4) .. (10.93,3.29)   ;
\draw    (331.93,107.83) -- (331.93,192.8) ;
\draw   (172.2,107.7) -- (491.66,107.7) -- (491.66,192.93) -- (172.2,192.93) -- cycle ;
\draw    (272,301.7) -- (233.79,316.99) ;
\draw [shift={(231.93,317.73)}, rotate = 338.19] [color={rgb, 255:red, 0; green, 0; blue, 0 }  ][line width=0.75]    (10.93,-3.29) .. controls (6.95,-1.4) and (3.31,-0.3) .. (0,0) .. controls (3.31,0.3) and (6.95,1.4) .. (10.93,3.29)   ;
\path  [shading=_txxjm6u8d,_f2blwfhco] (172.2,46.04) -- (491.67,46.04) -- (491.67,108.44) -- (172.2,108.44) -- cycle ; 
 \draw   (172.2,46.04) -- (491.67,46.04) -- (491.67,108.44) -- (172.2,108.44) -- cycle ; 

\draw    (152.2,108.44) -- (152.2,46.04) ;
\draw [shift={(152.2,46.04)}, rotate = 90] [color={rgb, 255:red, 0; green, 0; blue, 0 }  ][line width=0.75]    (0,5.59) -- (0,-5.59)   ;
\draw [shift={(152.2,108.44)}, rotate = 90] [color={rgb, 255:red, 0; green, 0; blue, 0 }  ][line width=0.75]    (0,5.59) -- (0,-5.59)   ;
\draw    (331.93,46.31) -- (331.93,108.17) ;
\draw    (152.2,108.44) -- (152.2,192.93) ;
\draw  [fill={rgb, 255:red, 228; green, 142; blue, 116 }  ,fill opacity=1 ] (391.98,192.9) -- (431.73,192.9) -- (372.22,363.87) -- (332.47,363.87) -- cycle ;
\draw    (411.86,45.8) -- (411.86,108.8) ;
\draw  [fill={rgb, 255:red, 228; green, 116; blue, 116 }  ,fill opacity=1 ] (272.48,192.87) -- (232.73,192.87) -- (292.24,363.84) -- (331.99,363.84) -- cycle ;
\draw  [fill={rgb, 255:red, 228; green, 142; blue, 116 }  ,fill opacity=1 ] (282.03,108.25) -- (301.83,108.25) -- (272.48,192.87) -- (252.39,192.87) -- cycle ;
\draw  [fill={rgb, 255:red, 228; green, 116; blue, 116 }  ,fill opacity=1 ] (222.52,108.23) -- (202.73,108.23) -- (232.73,192.87) -- (252.16,192.85) -- cycle ;
\draw  [fill={rgb, 255:red, 228; green, 142; blue, 116 }  ,fill opacity=1 ] (441.03,108.25) -- (460.83,108.25) -- (431.73,192.9) -- (411.39,192.87) -- cycle ;
\draw  [fill={rgb, 255:red, 228; green, 116; blue, 116 }  ,fill opacity=1 ] (381.52,108.23) -- (361.73,108.23) -- (391.98,192.9) -- (411.16,192.85) -- cycle ;
\draw    (252.16,45.85) -- (252.16,108.85) ;
\path  [shading=_48wupjldo,_ktc4321xs] (202.73,45.87) -- (222.52,45.87) -- (222.52,108.23) -- (202.73,108.23) -- cycle ; 
 \draw   (202.73,45.87) -- (222.52,45.87) -- (222.52,108.23) -- (202.73,108.23) -- cycle ; 

\path  [shading=_4db7pcu1y,_ehj43ygjy] (282.03,45.88) -- (301.83,45.88) -- (301.83,108.25) -- (282.03,108.25) -- cycle ; 
 \draw   (282.03,45.88) -- (301.83,45.88) -- (301.83,108.25) -- (282.03,108.25) -- cycle ; 

\path  [shading=_w1t22xsku,_907yh3hwx] (441.03,45.88) -- (460.83,45.88) -- (460.83,108.25) -- (441.03,108.25) -- cycle ; 
 \draw   (441.03,45.88) -- (460.83,45.88) -- (460.83,108.25) -- (441.03,108.25) -- cycle ; 

\path  [shading=_gk27dep2a,_rrize0cjp] (361.73,45.87) -- (381.52,45.87) -- (381.52,108.23) -- (361.73,108.23) -- cycle ; 
 \draw   (361.73,45.87) -- (381.52,45.87) -- (381.52,108.23) -- (361.73,108.23) -- cycle ; 

\pgfmathsetmacro{\yShift}{20}
 
 \draw (94,260) node [anchor=north west][inner sep=0.75pt]  [font=\LARGE]  {$h_{j_{0}\text{\mbox{-}} 1}$};
 \draw (303,389.4) node [anchor=north west][inner sep=0.75pt]  [font=\LARGE]  {$2l_{j_{0}\text{\mbox{-}} 1}$};
 \draw (425,335) node [anchor=north west][inner sep=0.75pt]  [font=\LARGE]  {$n_{R}$};
 \draw (204,335) node [anchor=north west][inner sep=0.75pt]  [font=\LARGE]  {-$n_{L}$};
 \draw (200,250+\yShift) node [anchor=north west][inner sep=0.75pt]  [font=\LARGE]  {$\omega _{1}$};
\draw (260,225+\yShift) node [anchor=north west][inner sep=0.75pt]  [font=\LARGE]  {$\omega _{2}$};
\draw (376 ,225+\yShift) node [anchor=north west][inner sep=0.75pt]  [font=\LARGE]  {$\omega _{4}$};
\draw (319,250+\yShift) node [anchor=north west][inner sep=0.75pt]  [font=\LARGE]  {$\omega _{3}$};
\draw (436,250+\yShift) node [anchor=north west][inner sep=0.75pt]  [font=\LARGE]  {$\omega _{5}$};
\draw (88,52) node [anchor=north west][inner sep=0.75pt]  [font=\LARGE]  {$\tfrac{\tau ^{j_{0} +1}}{2}$};
\draw (325,17) node [anchor=north west][inner sep=0.75pt]  [font=\LARGE]  {$F$};
\draw (94,135) node [anchor=north west][inner sep=0.75pt]  [font=\LARGE]  {$h_{j_{0}}$};

\end{tikzpicture}

\end{center}
\caption{The self-similar construction in \autoref{prop:a3}.
The unit cells from \autoref{lem:aa1} are iterated with decreasing size towards the top boundary.
The boundary layer corresponds to the cut-off cells from \autoref{lem:aa2}.
The partition of $\omega$ from the proof of \autoref{lem:aa1} is indicated here.}
\label{fig:branching}
\end{figure}

We now make rigorous the idea \eqref{eq:EnergySplit} of splitting the incompatible energy into the excess energy and a compatible remainder.
\begin{lemma}[Unit cell construction]\label{lem:aa1}
Assume the hypotheses of \autoref{prop:a4}.
In addition, let $b\in\RR^d$ be such that $\ppc(e_1)a=b\otimes e_1$.
Let $\omega=(-l,l)\times(0,h)$ with $0<l\leq h\leq 1$. \\
Then, there exist a vector field $v\in W^{1,\infty}(\omega;\RR^2)$ and a phase arrangement $\chi\in BV(\omega;\cK)$
such that, defining $u=F+\nabla v$, the localized energy can be estimated by
\begin{equation}\label{eq:aa1est}
  E^{\curl}_\eps(u,\chi;\omega) := \int_\omega \abs{u-\chi}^2\,dx + \eps \norm{\nabla \chi}_{TV(\omega)} \leq 
  2lh E^{\curl}_0(F,\cK) + C\left(\frac{l^3}{h} + \eps h\right) \quad\forall\eps>0
\end{equation}
for some constant $C=C(a, \ttheta_{\curl})>0$. Moreover, the vector field has the following boundary values:
\begin{enumerate}
  \item (Lateral sides) For all $x_2\in[0,h]$, we have 
  \begin{equation}\label{eq:sidestrace}
    v(-l,x_2)= v(l,x_2)=0.
  \end{equation}
  \item (Bottom and top side) For all $x_1\in[-l,l]$, it holds that 
  \begin{equation}\label{eq:TopBottomTrace}
  v(x_1,0)= \varphi_l(x_1)b, \hspace{2cm} v(x_1,h)= \varphi_{{l}/{2}}(x_1-\tfrac{l}{2})b,
  \end{equation}
  where, for $r>0$ and $t\in\RR$, we define the $2r$-periodic function $\varphi_r:\RR\to\RR$ by setting
  \begin{equation}\label{eq:varphirBoundaryData}
    \varphi_r(t):=r\varphi\left(\tfrac{t}{r}\right),\text{ where }
    \varphi(t):=
    \begin{cases}
    -\ttheta_{\curl}(1+t) &  t\in[-1,-\ttheta_{\curl})+2\ZZ,\\
    (1-\ttheta_{\curl})t &   t\in[-\ttheta_{\curl},\ttheta_{\curl})+2\ZZ,\\
    -\ttheta_{\curl}(t-1) &  t\in[\ttheta_{\curl},1)+2\ZZ.
    \end{cases}
  \end{equation}
\end{enumerate}
\end{lemma}
In order to prove \autoref{lem:aa1}, we need the following formula.
\begin{lemma}\label{lem:importantequality}
Let $X$ be a real inner product space.
Let $v,w\in X $. Then, it holds that
\begin{equation}
  (1-\theta)\abs{v}^2+\theta\abs{w}^2 = \abs{(1-\theta)v + \theta w}^2 + \theta(1-\theta)\abs{v - w}^2\quad\forall\theta\in\RR.
\end{equation}
\end{lemma}
\begin{proof}
For any $\theta\in\RR$, we compute
\begin{equation}
    \begin{aligned}
        \abs{(1-\theta)v + \theta w}^2 &= (1-\theta)^2\abs{v}^2  +\theta^2\abs{w}^2+ 2(1-\theta)\theta(v,w)\\
        &= (1-\theta)\abs{v}^2+\theta\abs{w}^2 - (1-\theta)\theta\left(\abs{v}^2+\abs{w}^2\right) + 2(1-\theta)\theta(v,w)\\
        &=(1-\theta)\abs{v}^2+\theta\abs{w}^2- (1-\theta)\theta\abs{v-w}^2,
    \end{aligned}
\end{equation}
which yields the asserted formula.
\end{proof}
We now come to the unit cell construction.
\begin{proof}[Proof of \autoref{lem:aa1}]
The proof consists of three steps: In the first step, $\omega$ is partitioned. This partition is then used to define $u$ and $\chi$ in the second step.
Finally, the localized energy is estimated in the last step.\\
For brevity, we omit the subscript $\curl$ in the notation of the optimal volume fraction $\ttheta_{\curl}$.
\Step{1: Partition of $\omega$}
It will prove useful to define the parameter $\gamma$:
\begin{equation}
\gamma:=\frac{(1-\ttheta)l}{2h}.
\end{equation}
As indicated for the lower unit cell in \autoref{fig:branching}, we partition $\omega$ into $(\omega_i)_{i\leq5}$ by setting
\begin{equation}
    \begin{gathered}
      \omega_1 := \{x\in\omega: x_1\in (-l,\,  -\ttheta l - \gamma x_2)\},\hspace{2cm} \omega_2 := \{x\in\omega: x_1\in (-\ttheta l - \gamma x_2,\, -\gamma x_2)\},\\
      \hspace{-0.6cm}\omega_3 := \{x\in\omega: x_1\in (-\gamma x_2,\,\gamma x_2)\},\hspace{2.6cm} \omega_4 := \{x\in\omega: x_1\in (\gamma x_2,\, \ttheta l + \gamma x_2)\},\\
      \omega_5 := \{x\in\omega: x_1\in ( \ttheta l + \gamma x_2,\, l)\}.\label{eq:omegaparts}
    \end{gathered}
\end{equation}
The following volume proportions will play an important role in the proof:
\begin{equation}\label{eq:volumeproportions}
  \abs{\omega_2}=\abs{\omega_4}=\frac{\ttheta}{2}\abs{\omega},\qquad\qquad \abs{\omega_1\cup\omega_3\cup\omega_5} = (1-\ttheta)\abs{\omega}.
\end{equation}
We introduce the vectors $n_L:=e_1+\gamma e_2$ and $n_R:=e_1-\gamma e_2$, which are normal to the interfaces of the subdomains:
\begin{equation}
  n_L \perp \left(\overline{\omega_1}\cap\overline{\omega_2}\right), \,\left(\overline{\omega_2}\cap\overline{\omega_3}\right),\qquad\qquad
  n_R\perp\left(\overline{\omega_3}\cap\overline{\omega_4}\right),\, \left(\overline{\omega_4}\cap\overline{\omega_5}\right).
\end{equation}

\Step{2: Displacement gradient and phase arrangement}
We define the displacement gradient $u\in L^\infty(\omega;\rddd{2})$ by prescribing the values in each subdomain and only then argue that there exists
$v\in W^{1,\infty}(\omega;\RR^2)$ satisfying the boundary conditions \eqref{eq:sidestrace} and \eqref{eq:TopBottomTrace} with $F+\nabla v=u$ in $\omega$.\\
To this end, we use the matrices $\ta_0,\ta_1\in\rddd{2}$ from the compatible approximation discussed in \autoref{rmk:compatibleApprox}:
\begin{equation}\label{eq:TildeWellCurl}
    \ta_0 := F -\ttheta b\otimes e_1  \hspace{2cm}   \ta_1 := F +(1-\ttheta) b\otimes e_1,
\end{equation}
where we inserted $\ppc(e_1)a=b\otimes e_1$.
As we assumed the hypotheses of \autoref{prop:a4}, we have $e_1\in S_{\curl}(a)$, which yields
\begin{equation}\label{eq:EOneOptLamDir}
\abs{a-\ppc(e_1)a}^2=h_{\curl}(a).
\end{equation}
We define $E_2 := \gamma b\otimes e_2\in\rddd{2}$ and consider the following perturbations of $\ta_1$:
\begin{equation}\label{eq:perturbedtildewells}
  \ta_{1,L}:=\ta_1 + E_2,\qquad\qquad \ta_{1,R}:=\ta_1 - E_2.
\end{equation}
This allows us to introduce $u:\omega\to\rddd{2}$ by setting
\begin{equation}\label{eq:lema1mapu}
  u(x):= \begin{cases}
    \ta_0 &x\in\omega_1 \cup\omega_3\cup\omega_5,\\
    \ta_{1,L} &x\in\omega_2,\\
    \ta_{1,R} &x\in\omega_4,
  \end{cases}
\end{equation}
We verify the Hadamard jump conditions:
\begin{equation}
  \begin{aligned}
    &\text{At the interfaces } \overline{\omega_1}\cap\overline{\omega_2} \text{ and }\overline{\omega_2}\cap\overline{\omega_3}: 
    &&\ta_{1,L}-\ta_0 = b\otimes e_1 + E_2 = b\otimes n_L,\\
    &\text{At the interfaces } \overline{\omega_3}\cap\overline{\omega_4} \text{ and }\overline{\omega_4}\cap\overline{\omega_5}: 
    &&\ta_{1,R}-\ta_0 = b\otimes e_1 - E_2 = b\otimes n_R.
  \end{aligned}
\end{equation}
Therefore, the matrix field $u$ is $\curl$-free in $\omega$ and the term \textit{displacement gradient} justified.
Indeed, setting $v(-l,x_2)=0$ according to the boundary condition \eqref{eq:sidestrace} and integrating
$\del_1 v(\cdot,x_2)$ from $-l$ to $x_1$, we obtain $v\in W^{1,\infty}(\omega;\RR^2)$ given by
\begin{equation}\label{eq:gradOfVectorfieldBranchConstr}
  v(x)=\begin{cases}
    -\ttheta (l+x_1)b & x\in \omega_1,\\
    \big((1-\ttheta) x_1 + \gamma x_2\big)b & x\in \omega_2,\\
    -\ttheta x_1 b & x\in \omega_3,\\
    \big((1-\ttheta) x_1 - \gamma x_2\big)b & x\in \omega_4,\\
    -\ttheta (x_1-l)b & x\in \omega_5.
  \end{cases}
\end{equation}
From this, it follows that $F+\nabla v = u$ in $\omega$ since
\begin{align}
      \del_1 v(x) &= (u(x)-F)e_1 =\begin{cases}
        -\ttheta b &x\in\omega_1\cup\omega_3\cup\omega_5,\\
        (1-\ttheta) b &x\in\omega_2 \cup \omega_4,
      \end{cases}\label{eq:gradOfVectorfieldBranchConstr2}\\
        \del_2 v(x) &= (u(x)-F)e_2 =\begin{cases}
        0 &x\in\omega_1\cup\omega_3\cup\omega_5,\\
        \gamma b &x\in\omega_2,\\
        -\gamma b &x\in\omega_4.
      \end{cases}
\end{align}
The lateral boundary conditions \eqref{eq:sidestrace} are a consequence of \eqref{eq:gradOfVectorfieldBranchConstr}.
The top and bottom boundary conditions \eqref{eq:TopBottomTrace} can be verified by a computation or by observing that $v(-l,0)=v(-l,h)=0$
and that $\del_1 v(\cdot,x_2)$ matches with the derivative of the boundary data for both $x_2\in\{0,h\}$.\medbreak

We define the phase arrangements $\chi\in BV(\omega;\cK),\,\tchi\in BV(\omega;\tcK)$, where $\tcK:=\{\ta_0,\ta_1\}$, setting
\begin{equation}
    \chi(x):= \begin{cases}
      a_0 &x\in\omega_1\cup\omega_3\cup\omega_5,\\
      a_1 &x\in\omega_2\cup\omega_4,
    \end{cases}
    \hspace{1.5cm}
    \tilde{\chi}(x):= \begin{cases}
      \ta_0 &x\in\omega_1\cup\omega_3\cup\omega_5,\\
      \ta_1 &x\in\omega_2\cup\omega_4.
    \end{cases}
\end{equation}
\Step{3: Localized energy -- estimates \& extraction of the excess energy}
The \textit{localized elastic energy} is split into three parts
\begin{equation}\label{eq:stint}
  E^{\curl}_{el}(u,\chi;\omega):=\int_\omega \abs{u-\chi}^2\,dx  = \underbrace{\int_\omega \abs{u-\tchi}^2 \,dx}_{=E^{\curl}_{el}(u,\tchi;\,\omega)}
  + 2 \underbrace{\int_\omega (u-\tchi, \tchi-\chi) \,dx}_{=:I_2}+ \underbrace{\int_\omega \abs{\tchi-\chi}^2\,dx}_{=:I_3}.
\end{equation}
The first integral in this sum is the localized elastic energy $E^{\curl}_{el}(u,\tchi;\omega)$.
Since this is the localized elastic energy of the unit cell construction for the compatible two-well problem with data $(F,\tcK)$ (see, for instance, \cite[Lemma A.1]{ruland}),
we expect the following bound 
\begin{equation}
  E_{el}(u,\tchi;\omega) = \int_\omega \abs{u-\tchi}^2 \leq C \frac{l^3}{h}
\end{equation}
for some constant $C>0$. Indeed, this estimate is derived by
\begin{multline}\label{eq:estimateinlema1}
  E_{el}(u,\tchi;\omega) = \abs{\omega_2} \abs{\ta_{1,L}-\ta_1}^2 + \abs{\omega_4} \abs{\ta_{1,R}-\ta_1}^2 = {\ttheta}\abs{\omega}\abs{E_2}^2 \\
  = {2\ttheta lh}\gamma^2\abs{b\otimes e_2}^2
  = \frac{\ttheta(1-\ttheta)^2l^3}{2h}  \abs{b\otimes e_1}^2
  = \frac{\ttheta(1-\ttheta)^2l^3}{2h} \abs{\ppc(e_1)a}^2\\
  \leq \frac{\ttheta(1-\ttheta)^2l^3}{2h} \abs{a}^2
  =: C(a,\ttheta) \frac{l^3}{h},
\end{multline}
where we used \eqref{eq:volumeproportions} and that $b\otimes e_1=\ppc(e_1)a$.\medbreak

Note that the second and third integral in \eqref{eq:stint} arise due to incompatibility.
Indeed, if $(F,\cK)$ satisfy both compatibility conditions  \eqref{eq:compatiblewells} and \eqref{eq:compatiblebdrydata} for $\cA=\curl$,
it follows from \autoref{rmk:compatibleApprox} that $\tcK=\cK$ and in particular that $\tchi = \chi$ in $\omega$.\smallbreak
As we are studying the incompatible two-well problem,
we anticipate that these contributions are related to the excess energy $E^{\curl}_0(F,\cK)$.
The following computation shows that the excess energy can be recovered from the third integral $I_3$ in \eqref{eq:stint}:
\begin{equation}\label{eq:thirdint}
\begin{aligned}
    I_3& =\int_\omega \abs{\tchi-\chi}^2\,dx = \abs{\omega_1\cup\omega_3\cup\omega_5}\abs{\ta_0-a_0}^2 + \abs{\omega_2\cup\omega_4}\abs{\ta_1-a_1}^2\\
    &= (1-\ttheta) \abs{\omega} \abs{\ta_0-a_0}^2 + \ttheta\abs{\omega} \abs{\ta_1-a_1}^2\\
    &= \abs{\omega}\left( \big\lvert(1-\ttheta)(\ta_0-a_0) + \ttheta(\ta_1-a_1)\big\rvert ^2   + \ttheta(1-\ttheta) \big\lvert (\ta_0-a_0) -  (\ta_1-a_1) \big\rvert ^2 \right)\\
    &= \abs{\omega}\left( \big\lvert F-a_{\ttheta} \big\rvert^2   + \ttheta(1-\ttheta) \big\lvert a-\ppc(e_1)a \big\rvert ^2 \right) = \abs{\omega} E^{\curl}_0(F,\cK)=2lhE^{\curl}_0(F,\cK),
\end{aligned}
\end{equation}
where we used \eqref{eq:volumeproportions}, \autoref{lem:importantequality}, \eqref{eq:TildeWellCurl}, \eqref{eq:EOneOptLamDir} and \autoref{thm:qwdom} in this order.\\
The symmetry of $\omega_2$ and $\omega_4$ combined with the definition \eqref{eq:perturbedtildewells} of the oppositely perturbed states $a_{1,L},\,a_{1,R}$ guarantee that the
integral $I_2$ in \eqref{eq:stint} vanishes:
\begin{equation}\label{eq:vanishsecondintaa1}
  \begin{aligned}
    I_2=\int_\omega (u-\tchi, \tchi-\chi)\,dx &= \abs{\omega_2} (\ta_{1,L}-\ta_1, \ta_1-a_1) + \abs{\omega_4} (\ta_{1,R}-\ta_1, \ta_1-a_1)\\
    &= \frac{\ttheta}{2}\abs{\omega}(E_2-E_2, \ta_1-a_1)  = 0.
  \end{aligned}
\end{equation}
Using $l\leq h$, we estimate the surface energy by
\begin{equation}\label{eq:aa1surface}
  \begin{aligned}
    \frac{1}{\abs{a}}\norm{\nabla \chi}_{TV(\omega)} &= \per(\{\chi=a_0\};\omega) = \per(\omega_1\cup\omega_3\cup\omega_5;\omega) =  4 \sqrt{\gamma^2h^2+h^2}\\
    &=4 \,\sqrt{\frac{(1-\ttheta)^2}{4}l^2+h^2\,}   
    \leq 4h\,\sqrt{\frac{(1-\ttheta)^2}{4} +1\,} \leq 8 h.
  \end{aligned}
\end{equation}
Multiplying by $\lvert a\rvert$ and possibly enlarging the constant $C=C(a,\ttheta)$ such that $C\geq 8 \lvert a\rvert$, it follows from
\eqref{eq:stint}, \eqref{eq:estimateinlema1}, \eqref{eq:thirdint}, \eqref{eq:vanishsecondintaa1} and \eqref{eq:aa1surface} that
\begin{equation}\label{eq:aa1estInPF}
  E^{\curl}_\eps(u,\chi;\omega) \leq 2lh E^{\curl}_0(F,\cK) + C\left(\frac{l^3}{h} + \eps h\right) \quad\forall\eps>0.
\end{equation}
This shows \eqref{eq:aa1est} and completes the proof.
\end{proof}
In the branching construction (\autoref{prop:a3}), we will use a cut-off procedure for the unit cells in the boundary layer, illustrated in \autoref{fig:branching}.
The cut-off argument is carried out using linear interpolation and is addressed in the following lemma, where we also characterize the required energy. 
\begin{lemma}[Cut-off layer]\label{lem:aa2}
Assume the hypotheses of \autoref{prop:a4}.
In addition, let $b\in\RR^d$ be such that $\ppc(e_1)a=b\otimes e_1$.
Let $\omega=(-l,l)\times(0,h)$ with $0<l\leq 2h\leq 1$. \\
Then, there exists $v\in W^{1,\infty}(\omega;\RR^2)$ and a phase arrangement $\chi\in BV(\omega;\cK)$
such that, defining $u=F+\nabla v$, the localized energy can be estimated by
\begin{equation}\label{eq:aa1est2}
  E^{\curl}_\eps(u,\chi;\omega) = \int_\omega \abs{u-\chi}^2\,dx + \eps \norm{\nabla \chi}_{TV(\omega)} \leq 
  2lh E^{\curl}_0(F,\cK) + C\left(lh + \eps h\right) \quad\forall\eps>0
\end{equation}
for some constant $C=C(a, \ttheta_{\curl})>0$. Moreover, the vector field has the following boundary values:
\begin{enumerate}
  \item (Lateral sides) For all $x_2\in[0,h]$, we have 
  \begin{equation}\label{eq:sidestrace2}
    v(-l,x_2)= v(l,x_2)=0.
  \end{equation}
  \item (Bottom and top side) For all $x_1\in[-l,l]$, it holds that 
  \begin{equation}\label{eq:TopBottomTrace2}
  v(x_1,0)= \varphi_l(x_1)b, \hspace{2cm} v(x_1,h)=0,
  \end{equation}
  where $\varphi_l$ is given as in \eqref{eq:varphirBoundaryData}.
\end{enumerate}
\end{lemma}
\begin{proof}
The proof is organized into three steps: In the first step, we introduce the displacement map and the phase arrangements.
In the second step, we split the localized elastic energy as in the proof of \autoref{lem:aa1} and extract the excess energy.
Finally, we estimate the remaining part of the localized energy to prove \eqref{eq:aa1est2} in the third step.
\Step{1: Displacement map and phase arrangements}
We introduce the vector field $v\in  W^{1,\infty}(\omega;\RR^2)$ by setting
\begin{equation}\label{eq:DisplacementMapCutOffLayer}
  v(x):=\psi\left(\frac{x_2}{h}\right)\varphi_l(x_1) b,\quad x\in\omega,
\end{equation}
where the cut-off function $\psi:[0,\infty)\to\RR$ is given by
\begin{equation}\label{eq:CutOffFunction}
  \psi(t):= \begin{cases}
    1 & t\in[0,\frac{1}{2}],\\
    -4t+3 & t\in (\frac{1}{2}, \frac{3}{4}),\\
    0 & t\in[\frac{3}{4},\infty).
  \end{cases}
\end{equation}
Recalling \eqref{eq:varphirBoundaryData}, we see that the boundary conditions \eqref{eq:sidestrace2} and \eqref{eq:TopBottomTrace2} are satisfied.
In order to prove \eqref{eq:aa1est2}, we partition $\omega$ into the following two parts:
\begin{equation}
  \omega_I := [-\ttheta l, \ttheta l]\times(0,h),\qquad \omega_E :=\left((-l,-\ttheta l)\cup (\ttheta l,l)\right) \times(0,h),
\end{equation}
where we again omit the subscript $\curl$ of $\ttheta_{\curl}$ in this proof.
Next, we compute
\begin{equation}\label{eq:BoundaryLayerDisplacementGradient}
  u(x)=F+\nabla v(x) = F + \psi\left(\frac{x_2}{h}\right)\varphi_l'(x_1) b\otimes e_1 + \frac{1}{h}\psi'\left(\frac{x_2}{h}\right)\varphi_l(x_1) b\otimes e_2\quad\forall x\in\omega.
\end{equation}
While the cut-off function in \eqref{eq:DisplacementMapCutOffLayer} serves to interpolate
the bottom and top boundary values \eqref{eq:TopBottomTrace2} for any fixed $x_1\in(-l,l)$ along the $x_2$-direction, 
it will prove useful to disregard the cut-off function and also compute
\begin{equation}\label{eq:UsefulIdentityInBoundaryLayer}
  F + \nabla [\varphi_l(x_1) b]=F + \varphi_l'(x_1) b\otimes e_1 =\begin{cases}
    \ta_0 &  x\in\omega_E,\\
    \ta_1 &  x\in\omega_I,
  \end{cases}
\end{equation}
where $\ta_0,\,\ta_1\in\rddd{2}$ are given as in \eqref{eq:TildeWellCurl}.
This suggests that the displacement gradient $u$ is close to $a_0$ in $\omega_E$ and close to $a_1$ in $\omega_I$.
We therefore introduce the phase arrangements $\chi\in BV(\omega;\cK)$ and $\tchi\in BV(\omega;\tcK)$, where $\tcK:=\{\ta_0,\ta_1\}$, by setting
\begin{equation}\label{eq:BoundaryLayerPhaseArrangement}
  \chi(x):=\begin{cases}
    a_0 & x\in\omega_E,\\
    a_1 & x\in\omega_I,
  \end{cases}
  \hspace{1.5cm}
  \tchi(x):=\begin{cases}
    \ta_0 & x\in\omega_E,\\
    \ta_1 & x\in\omega_I.
  \end{cases}
\end{equation}
\Step{2: Localized energy --  extraction of the excess energy}
As in the proof of \autoref{lem:aa1}, we split the localized elastic energy
\begin{equation}\label{eq:stint3}
E^{\curl}_{el}(u,\chi;\omega)=\int_\omega \abs{u-\chi}^2\,dx  = \underbrace{\int_\omega \abs{u-\tchi}^2 \,dx}_{=E^{\curl}_{el}(u,\tchi;\,\omega)} 
+ 2 \underbrace{\int_\omega (u-\tchi, \tchi-\chi) \,dx}_{=:I_2}+ \underbrace{\int_\omega \abs{\tchi-\chi}^2\,dx}_{=:I_3}.
\end{equation}
Repeating the arguments from \eqref{eq:thirdint} and using the volume proportions of $\omega_I$ and $\omega_E$ relative to $\omega$, we obtain
\begin{equation}\label{eq:BoundaryLayerBaseEnergy}
  I_3 = \int_\omega \abs{\tchi-\chi}^2\,dx = 2lh E_0^{\curl}(F,\cK).
\end{equation}
Next, we claim that the cross-term $I_2$ from \eqref{eq:stint3} vanishes.
To prove this, it is useful to note that \eqref{eq:UsefulIdentityInBoundaryLayer} implies $\tchi(x) = F + \varphi_l'(x_1) b\otimes e_1$ for all $x\in\omega$, which together
with \eqref{eq:BoundaryLayerDisplacementGradient} yields
\begin{equation}\label{eq:uMinusTchi}
  u(x)-\tchi(x) = \left(\psi\left(\frac{x_2}{h}\right) -1\right)  \varphi_l'(x_1) b\otimes e_1 + \frac{1}{h}\psi'\left(\frac{x_2}{h}\right)\varphi_l(x_1) b\otimes e_2\quad\forall x\in\omega.
\end{equation}
We rewrite the integral $I_2$ from \eqref{eq:stint3} as
\begin{equation}\label{eq:SplitSecondIntegral}
  I_2 = \int_\omega (u-\tchi, \tchi-\chi) \,dx = \int_{\omega_E} (u-\tchi, \ta_0-a_0) \,dx + \int_{\omega_I} (u-\tchi, \ta_1-a_1) \,dx
\end{equation}
and apply Fubini's theorem to compute the integrals:\vspace{0.15cm}

\noindent\textit{(i) Integral in $\omega_E$:} Fix any $x_2\in[0,h]$. Using \eqref{eq:uMinusTchi}, we integrate $x_1\in (-l,-\ttheta l)\cup(\ttheta l,l)$:
\begin{multline}
  \int_{-l}^{-\ttheta l}  (u(x_1,x_2)-\tchi (x_1,x_2), \ta_0-a_0) \,dx_1  + \int_{\ttheta l}^{l}  (u(x_1,x_2)-\tchi (x_1,x_2), \ta_0-a_0) \,dx_1\\
  = \left(\psi\left(\frac{x_2}{h}\right) -1\right)(b\otimes e_1, \ta_0-a_0) \left(\int_{-l}^{-\ttheta l}  \varphi_l'(x_1) \,dx_1  + \int_{\ttheta l}^{l}  \varphi_l'(x_1) \,dx_1 \right) \\
  + \frac{1}{h}\psi'\left(\frac{x_2}{h}\right)(b\otimes e_2, \ta_0-a_0) \left(\int_{-l}^{-\ttheta l} \varphi_l(x_1) \,dx_1  + \int_{\ttheta l}^{l}  \varphi_l(x_1) \,dx_1 \right) .
\end{multline}
For all $x_1\in(-l,-\ttheta l)\cup(\ttheta l,l)$, we have $\varphi_l'(x_1)=-\ttheta$ and $\varphi_l(-x_1)=-\varphi_l(x_1)$. Therefore, it follows that
\begin{equation}\label{eq:MixedTermOmegaE}
  \int_{\omega_E} (u-\tchi, \ta_0-a_0) \,dx = - 2\ttheta(1-\ttheta)l (b\otimes e_1, \ta_0-a_0) \int_{0}^{h} \left(\psi\left(\frac{x_2}{h}\right) -1\right)  \,dx_2 .
\end{equation}
\noindent\textit{(ii) Integral in $\omega_I$:} We again fix any $x_2\in[0,h]$ and integrate $x_1\in(-\ttheta l,\ttheta l)$: 
\begin{multline}
\int_{-\ttheta l}^{\ttheta l}  (u(x_1,x_2)-\tchi (x_1,x_2), \ta_1-a_1) \,dx_1
  = \left(\psi\left(\frac{x_2}{h}\right) -1\right)  (b\otimes e_1, \ta_1-a_1) \int_{-\ttheta l}^{\ttheta l}  \varphi_l'(x_1) \,dx_1 \\
  + \frac{1}{h}\psi'\left(\frac{x_2}{h}\right) (b\otimes e_2, \ta_1-a_1) \int_{-\ttheta l}^{\ttheta l} \varphi_l(x_1) \,dx_1  .
\end{multline}
Now, using that $\varphi_l'(x_1)=1-\ttheta$ and $\varphi_l(-x_1)=-\varphi_l(x_1)$ for all $x_1\in(-\ttheta l, \ttheta l)$, we find
\begin{equation}\label{eq:MixedTermOmegaI}
  \int_{\omega_I} (u-\tchi, \ta_1-a_1) \,dx =  2\ttheta(1-\ttheta)l (b\otimes e_1, \ta_1-a_1) \int_{0}^{h} \left(\psi\left(\frac{x_2}{h}\right) -1\right)  \,dx_2 .
\end{equation}
Having computed both integrals, we combine \eqref{eq:SplitSecondIntegral}, \eqref{eq:MixedTermOmegaE} and \eqref{eq:MixedTermOmegaI} to infer that
\begin{equation}
  I_2  = 2\ttheta(1-\ttheta)l\Big( (b\otimes e_1, \ta_1-a_1) - (b\otimes e_1, \ta_0-a_0) \Big)
  \int_{0}^{h} \left(\psi\left(\frac{x_2}{h}\right) -1\right)  \,dx_2.
\end{equation}
By \eqref{eq:TildeWellCurl}, we have $\ta_1-\ta_0 =b\otimes e_1 =\ppc(e_1)a$. Since $(a-\ppc(e_1)a)\perp V_{\curl}(e_1)$ and $b\otimes e_1\in V_{\curl}(e_1)$, we obtain
\begin{equation}
  (b\otimes e_1, \ta_1-a_1) - (b\otimes e_1, \ta_0-a_0) = (b\otimes e_1, \ppc(e_1)a-a) = 0.
\end{equation}
In consequence, the cross-term in \eqref{eq:stint3} vanishes as claimed:
\begin{equation}\label{eq:BoundaryLayerMixedTerm}
  I_2  =0.
\end{equation}
\Step{3: Localized energy --  estimates} 
We now estimate the first term $E^{\curl}_{el}(u,\tchi;\,\omega)$ from \eqref{eq:stint3} as well as the localized surface energy to establish \eqref{eq:aa1est2}.
By \eqref{eq:varphirBoundaryData} and \eqref{eq:uMinusTchi}, it follows, for all $x\in\omega$, that
\begin{equation}\label{eq:BoundaryLayerEstimate}
\begin{aligned}
\abs{u(x)-\tchi(x)}^2
&\leq \Big\lvert \left(\psi\left(\frac{x_2}{h}\right) -1\right)  \varphi_l'(x_1) b\otimes e_1 + \frac{1}{h}\psi'\left(\frac{x_2}{h}\right)\varphi_l(x_1) b\otimes e_2 \Big\rvert^2\\
&\leq 2 \max\{\ttheta^2,(1-\ttheta)^2\} \abs{b\otimes e_1}^2 + \frac{2 l^2}{h^2} \psi'^2\left(\frac{x_2}{h}\right) \abs{b\otimes e_2}^2
\leq C\left(1 + \frac{l^2}{h^2} \psi'^2\left(\frac{x_2}{h}\right)\right)
\end{aligned}
\end{equation}
for some constant $C=C(a,\ttheta)>0$.
Integrating \eqref{eq:BoundaryLayerEstimate} over $\omega$, we obtain
\begin{equation}
\begin{aligned}
    E^{\curl}_{el}(u,\tchi;\,\omega) \leq C\left(2lh + \frac{l^3}{h^2} \int_0^h \psi'^2\left(\frac{x_2}{h}\right) \,dx_2 \right)
    \leq C\left(2lh + \frac{l^3}{h} \int_0^1 \psi'^2(t) \,dt \right).
\end{aligned}
\end{equation}
By possibly enlarging the constant $C$ (still depending on the same parameters) and using that $l\leq 2h$, it follows that
\begin{equation}\label{eq:BoundaryLayerEstimate2}
  E^{\curl}_{el}(u,\tchi;\,\omega) \leq C lh.
\end{equation}
Recalling \eqref{eq:BoundaryLayerPhaseArrangement} and arguing as in \eqref{eq:aa1surface}, we estimate
\begin{equation}\label{eq:BoundaryLayerSurfaceEnergy}
\norm{\nabla \chi}_{TV(\omega)}  \leq 2 h\lvert a\rvert.
\end{equation}
Requiring that $C\geq 2 \lvert a\rvert$, we infer from \eqref{eq:stint3}, \eqref{eq:BoundaryLayerBaseEnergy}, \eqref{eq:BoundaryLayerMixedTerm}, \eqref{eq:BoundaryLayerEstimate2}
and \eqref{eq:BoundaryLayerSurfaceEnergy} that
\begin{equation}
  E^{\curl}_\eps(u,\chi;\omega) \leq 2lh E_0^{\curl}(F,\cK) + C\left(lh + \eps h\right) \quad\forall\eps>0,
\end{equation}
which shows \eqref{eq:aa1est2} and concludes the proof.
\end{proof}
With \autoref{lem:aa1} and \autoref{lem:aa2} in hand, we come to the branching construction.
\begin{breakproposition}[Branching construction]\label{prop:a3}
Assume the hypotheses of \autoref{prop:a4}.
Then, for all $N\in\NN_{>1}$, there exists a vector field $v\in W^{1,\infty}_0(Q;\RR^2)$ and a phase arrangement
$\chi\in BV(Q; \cK)$ such that, for $u = F + \nabla v\in\cD_F^{\curl}(Q)$, the following energy estimate holds:
\begin{equation}\label{eq:PropA3Estimate}
  E^{\curl}_\eps(u,\chi) - E^{\curl}_0(F,\cK) \leq C\Big(\frac{1}{N^2}+\eps N\Big)\quad\forall \eps>0
\end{equation}
for some constant $C=C(a,\ttheta_{\curl})>0$.
\end{breakproposition}
\begin{proof}
The proof is organized into four steps: In the first step, we partition the upper half of the unit square into layers of unit cells. 
In the second step, we define $u$ and $\chi$ by \autoref{lem:aa1} in interior unit cells and by \autoref{lem:aa2} in boundary cells.
We then extend $u$ and $\chi$ symmetrically to the lower half square in the third step.
Finally, we estimate the singularly perturbed energy in the last step.

\Step{1: Partition of the upper half square}
We will first focus on the construction in the upper half square $Q^+:=(0,1)\times(\tfrac{1}{2}, 1)$.
As illustrated in \autoref{fig:branching}, the idea is to cover $Q^+$ by starting with a row of unit cells at the bottom and stacking two smaller unit cells, each half as wide, on top of each of the unit cells of the bottom row.
Iterating this procedure, we stack layers of increasingly narrow unit cells on top of each other and define the displacement map $v$ in each of the unit cells by \autoref{lem:aa1}.
In the topmost layer, we define $v$ in each unit cell using the cut-off construction \autoref{lem:aa2}.
In this way, we obtain a self-similar displacement map $v\in W^{1,\infty}(Q^+,\RR^2)$ which vanishes along the lateral and top boundary of $Q^+$.
Moreover, the number of oscillations of $v$ doubles in every layer as we approach the top.\medbreak

Let $N\in\NN_{>1}$. In the construction, this will be the number of oscillations in the bottom layer of $Q^+$.
Pick any $\tau\in(\tfrac{1}{4},\tfrac{1}{2})$.
The precise value of $\tau$ will play no further role in our analysis,
but it can be used to optimize the constant in \eqref{eq:PropA3Estimate} as detailed in \cite{chanPHD}.
In the $j$-th layer, $j\in\NN_0$, the unit cells have the dimensions:
\begin{equation}\label{eq:BranchingConstructionLengthsAndWidths}
  l_j := \frac{1}{2N2^j }, \hspace{2cm}  h_j:=\tau^j\frac{1-\tau}{2}.
\end{equation}
As we require $l_j\leq h_j$ in \autoref{lem:aa1}, we stop at $j_0:=\max \{j\in\NN_0:l_j\leq h_j\}$.
Note that $j_0$ exists since $N>1$ implies $l_0\leq h_0$.
Introducing
\begin{equation}
    y_j:=1-\frac{\tau^j}{2},\quad j\in\NN_0,
\end{equation}
we define the unit cells in the $j$-th layer, $j\in\{0,\dots,j_0\}$, by
\begin{equation}
    \omega_{j,k}:=((2k+1)l_j, y_j) + (-l_j,l_j)\times(0,h_j), \quad k\in\{0,\dots ,N2^j-1\}.
\end{equation}
Note that $y_{j+1}-y_j=h_j$ for all $j\in\{0,\dots,j_0-1\}$ and that the number of unit cells doubles in each layer.
The cut-off is carried out in the $j_0{+}1$-th layer, where we define the unit cells by
\begin{equation}
    \omega_{{j_0+1},k}:=((2k+1)l_{j_0+1}, y_{j_0+1}) + (-l_{j_0+1},l_{j_0+1})\times(0,\tfrac{\tau^{j_0+1}}{2}), \quad k\in\{0,\dots ,N2^{j_0+1}-1\}.
\end{equation}

\Step{2: Construction in the upper half square}
For the interior layers, that is, for $j\in\{0,\dots,j_0\}$, let $v_j$ and $\chi_j$ denote the maps from \autoref{lem:aa1} for $\omega=(-l_j,l_j)\times(0,h_j)$.
Similarly, for the boundary layer, let $v_{j_0+1}$ and $\chi_{j_0+1}$ denote the maps from \autoref{lem:aa2} for $\omega=(-l_{j_0+1},l_{j_0+1})\times(0,{\tau^{j_0+1}}/{2})$.
Using this, we define the maps $v$ and $\chi$ in $\omega_{j,k}$
for $j\in\{0,\dots,j_0+1\}$ and $k\in\{0,1,\dots , N2^j-1\}$ by
\begin{equation}\label{eq:blockwiseuptotop}
    \begin{aligned}
        v(x_1,x_2)&:=v_j(x_1 - (2k+1)l_j, x_2 - y_j),\\
        \chi(x_1,x_2)&:=\chi_j(x_1 - (2k+1)l_j, x_2 - y_j).
    \end{aligned}
\end{equation}
Due to the boundary values in \autoref{lem:aa1} and \autoref{lem:aa2}, the vector field $v$ is continuous along the interfaces of the unit cells, which implies that
$v\in W^{1,\infty}(Q^+;\RR^d)$. Another consequence is that $v$ satisfies the boundary conditions at the top and at the lateral faces of $Q^+$:
\begin{equation}\label{eq:HalfSquareBoundaryValue}
\begin{aligned}
            &(i)\phantom{i} \textit{ (Top side) } \hspace{0.9cm}\forall x_1\in[0,1]: v(x_1, 1)= 0,\\
            &(ii) \textit{ (Lateral sides) } \hspace{0.2cm}\forall x_2\in[\tfrac{1}{2}, 1]: v(0,x_2)=v(1,x_2) = 0.
\end{aligned}
\end{equation}
Similarly, we see that $\chi$ lies in $BV(Q^+;\cK)$.
Indeed, retracing the boundary values of $\chi$ in \autoref{lem:aa1} and \autoref{lem:aa2}, it follows that $\nabla\chi$ has no additional jumps along the interfaces of the unit cells:
\begin{equation}\label{eq:branchingConstructionSurfaceEnergyDecomp}
    \norm{\nabla\chi}_{TV(Q^+)} = \sum_{j=0}^{j_0+1} \sum_{k=0}^{N2^j-1} \norm{\nabla\chi}_{TV(\omega_{j,k})}.
\end{equation}
\Step{3: Mirroring}
We extend $v$ and $\chi$ to the full unit square by mirroring along the line $\{x_2=\tfrac{1}{2}\}$. For $x_1\in [-1,1]$ and $x_2\in[0,\tfrac{1}{2}]$, we thus set
\begin{equation}\label{eq:Mirroring}
        v(x_1,\tfrac{1}{2}-x_2):=v(x_1,\tfrac{1}{2}+x_2), \hspace{1.5cm} \chi(x_1,\tfrac{1}{2}-x_2):=\chi(x_1,\tfrac{1}{2}+x_2).
\end{equation}
This yields $v\in W_0^{1,\infty}(Q;\RR^2)$ and $\chi\in BV(Q;\cK)$. Here, the vector field $v$ vanishes along the boundary $\del Q$ due to \eqref{eq:HalfSquareBoundaryValue}.
Moreover, no additional phase interfaces are introduced along the line $\{x_2=\tfrac{1}{2}\}$:
\begin{equation}\label{eq:MirroredSurfaceEnergy}
    \norm{\nabla\chi}_{TV(Q)} = 2\norm{\nabla\chi}_{TV(Q^+)}.
\end{equation}
\Step{4: Estimates}
Let $\eps>0$. We now estimate the singularly perturbed two-well energy $E^{\curl}_\eps(u,\chi)$ for the displacement gradient $u=F+\nabla v\in\cD_F^{\curl}(Q)$.
Using \eqref{eq:branchingConstructionSurfaceEnergyDecomp}, \eqref{eq:Mirroring} and \eqref{eq:MirroredSurfaceEnergy}, it holds that
\begin{equation}\label{eq:BranchingFullDecomp}
E^{\curl}_\eps(u,\chi)= 2 E^{\curl}_\eps(u,\chi; Q^+) = 2\sum_{j=0}^{j_0+1} \sum_{k=0}^{N2^j-1}  E^{\curl}_{\eps}(u,\chi; \omega_{j,k}).
\end{equation}
By \autoref{lem:aa1}, there exists a constant $C=C(a,\ttheta_{\curl})>0$ such that
\begin{equation}\label{eq:InteriorLayerEst}
   E^{\curl}_{\eps}(u,\chi; \omega_{j,k}) \leq  \abs{\omega_{j,k}} E^{\curl}_0(F,\cK) + C\bigg(\frac{l_j^3}{h_j} + \eps h_j\bigg)
\end{equation}
for all $j\in\{0,\dots,j_0\}$ and $k\in\{0,\dots, N2^j-1\}$.
By possibly enlarging the constant $C=C(a,\ttheta_{\curl})$, it follows from \autoref{lem:aa2} together with $h_{j_0+1} < l_{j_0+1} \leq 2 h_{j_0+1}$ and $\tau^{j_0+1}\sim h_{j_0+1}$
that
\begin{equation}\label{eq:BoundaryLayerEst}
\begin{aligned}
     E^{\curl}_{\eps}(u,\chi; \omega_{j_0+1,k}) &\leq  \abs{\omega_{j_0+1,k}} E^{\curl}_0(F,\cK) + C\bigg(l_{j_0} \frac{\tau^{j_0+1}}{2} + \eps \frac{\tau^{j_0+1}}{2}\bigg)\\
     &\leq \abs{\omega_{j_0+1,k}} E^{\curl}_0(F,\cK) + C\bigg(\frac{l_{j_0+1}^3}{h_{j_0+1}} + \eps h_{j_0+1}\bigg)
\end{aligned}
\end{equation}
for all $k\in\{0,\dots, N2^{j_0+1}-1\}$.\medbreak
Recalling \eqref{eq:BranchingConstructionLengthsAndWidths} and combining
\eqref{eq:BranchingFullDecomp}, \eqref{eq:InteriorLayerEst} with \eqref{eq:BoundaryLayerEst}, we derive
\begin{multline}
  E^{\curl}_\eps(u,\chi) \leq 2\sum_{j=0}^{j_0+1} \sum_{k=0}^{N2^j-1} \bigg(\abs{\omega_{j,k}} E^{\curl}_0(F,\cK) + C\bigg(\frac{l_j^3}{h_j} + \eps h_j\bigg)\bigg)\\
  \leq E^{\curl}_0(F,\cK) + 2C \sum_{j=0}^{j_0+1} N2^j \bigg(\frac{l_j^3}{h_j} + \eps h_j\bigg)
  = E^{\curl}_0(F,\cK) + C \sum_{j=0}^{j_0+1}\bigg(\frac{l^2_j}{h_j} + \eps \frac{h_j}{l_j}\bigg)\\
  \leq E^{\curl}_0(F,\cK) + C \sum_{j=0}^{j_0+1} \bigg(\frac{1}{N^2} \Big(  \frac{1}{4\tau} \Big)^j + \eps N(2\tau)^j\bigg).
\end{multline}
Since $\tau\in(\tfrac{1}{4},\tfrac{1}{2})$, the geometric series converge and up to enlarging the constant $C=C(a,\ttheta_{\curl})$, we obtain
\begin{equation}
  E^{\curl}_\eps(u,\chi) \leq E^{\curl}_0(F,\cK) + C \Big(\frac{1}{N^2} + \eps N\Big),
\end{equation}
which proves \eqref{eq:PropA3Estimate}.
\end{proof}
We are now ready to prove \autoref{prop:a4}. As we have already seen in \autoref{subsec:ChangeOfVar}, this then completes the proof of the main result \autoref{thm:incompgrad}.
\begin{proof}[Proof of \autoref{prop:a4}]
Given $\eps\in(0,1)$, we apply \autoref{prop:a3} for $N:=\lceil \eps^{\nicefrac{-1}{3}} \rceil\in\NN_{>1}$.
Since $\eps^{\nicefrac{-1}{3}}\leq N\leq 2 \eps^{\nicefrac{-1}{3}}$, we obtain \eqref{eq:a4Estimate} from \eqref{eq:PropA3Estimate}. 
\end{proof}
\begin{rmk}\label{rmk:DisplGradComp}
While the data $(F,\cK)$ in \autoref{prop:a4} may be incompatible, this is not apparent on the level of the displacement gradients $(u_\eps)_\eps$ in the proof of \autoref{prop:a4}.
Indeed, let $(F,\tcK)$ be the compatible approximation of $(F,\cK)$ where $\tcK=\{\ta_0,\ta_1\}$ is defined by \eqref{eq:TildeWellCurl}.
Then the displacement gradients obtained from the upper bound construction for $E^{\curl}_\eps(F,\tcK)\lesssim \eps^{\nicefrac{2}{3}}$ are indistinguishable
from the displacement gradients from the upper bound construction for $E^{\curl}_\eps(F,\cK)-E^{\curl}_0(F,\cK)\lesssim \eps^{\nicefrac{2}{3}}$.\medbreak

When considering upper scaling bounds of zeroth-order-corrected singularly perturbed $\cA$-free two-well energies for more general differential operators,
it may therefore be useful to study the corresponding problem for a compatible approximation $(F,\tcK)$
with the aim of proving the excess energy extraction \eqref{eq:EnergySplit} and the compatible upper bound \eqref{eq:EnergySplit2}.
As we will see in \autoref{subsec:branchingConstruction2}, this approach is also successful for the $\eps^{\nicefrac{4}{5}}$-upper bound construction.
\end{rmk}

\section{Upper bounds for the singularly perturbed geometrically linear two-well energy}\label{sec:UpperBoundCC}
In this section, we derive upper scaling bounds for the zeroth-order-corrected singularly perturbed geometrically linear two-well energy
with the goal of proving \autoref{thm:incompsyma}~\textcolor{blue}{(iii)}.
The argument is divided into two parts:\smallbreak
In \autoref{subsec:TwoThridsUpperBoundCC}, we establish the following general upper scaling bound:
For any set of two distinct wells $\cK=\{a_0,a_1\}\subset\rddsymd{2}$ and boundary data $F\in\rddsymd{2}$ with $\ttheta_{\cc}(F,\cK)\in(0,1)$, we have
\begin{equation}\label{eq:TwoThridsUpperBoundCCIntro}
    E_\eps^{\cc}(F,\cK)-E_0^{\cc}(F,\cK) \leq C \eps^{\nicefrac{2}{3}}\quad\forall\eps\in(0,1).
\end{equation}
The proof is based on the branching construction from \autoref{sec:UpperBoundGradDiv}, adapted for $\cA=\ccurl$.\smallbreak

In \autoref{subsec:branchingConstruction2}, we refine estimate \eqref{eq:TwoThridsUpperBoundCCIntro} in the case
where the wells differ by a rank-one matrix, yielding an improved $\eps^{\nicefrac{4}{5}}$-upper bound.
This estimate follows by redoing the branching construction from \cite[Theorem 1.2]{cc15} for incompatible boundary data.

\subsection{An \texorpdfstring{$\eps^{\frac{2}{3}}$}{eps2/3}-upper bound via a gradient two-well problem}\label{subsec:TwoThridsUpperBoundCC}
We begin by presenting the key result (\autoref{prop:UpperBoundCC1}) of \autoref{subsec:TwoThridsUpperBoundCC}. 
Its proof requires some preparation.
To motivate our approach, we outline how the branching construction for incompatible data from \autoref{sec:UpperBoundGradDiv} needs to be modified for the setting $\cA=\ccurl$.
Rather than implementing this construction, we develop a shorter proof.
This approach is formalized in \autoref{lem:UpperBoundCC}, which we then use to prove \autoref{prop:UpperBoundCC1}.
\begin{breakproposition}[Upper bound of the geometrically linear two-well energy]\label{prop:UpperBoundCC1}
Let $\cK=\{a_0,a_1\}\subset\rddsymd{2}$ and $F\in\rddsymd{2}$.
Consider the differential operator $\cA=\ccurl$; see \eqref{eq:curlcurl}.
For $a:=a_1-a_0$, let $\xi^*\in S_{\cc}(a)$ be an optimal lamination direction; see \autoref{def:optLaminationDirections}.
Further, let $\Omega\subset\RR^2$ be a rotated unit square with two faces normal to $\xi^*$.
Consider the energy $E^{\cc}_\eps(F,\cK)$ given by \eqref{eq:MinSingPertAfreeEnergy}.
Let $\ttheta_{\cc}=\ttheta_{\cc}(F,\cK)$ be given as in \autoref{prop:optvolfrac}.
Now, suppose that $\ttheta_{\cc}\in(0,1)$.
Then, there exists a constant $C=C(a,\ttheta_{\cc})>0$ such that
\begin{equation}\label{eq:UpperBoundCC1}
  E^{\cc}_\eps(F,\cK) - E^{\cc}_0(F,\cK) \leq C\eps^{\nicefrac{2}{3}}\quad\forall \eps\in (0,1).
\end{equation}
\end{breakproposition}
As we argue in the following, it is possible to prove \autoref{prop:UpperBoundCC1} using a branching construction similar to the one in \autoref{sec:UpperBoundGradDiv}.\smallbreak
Specifically, assume the hypotheses of \autoref{prop:UpperBoundCC1} and
consider the compatible approximation $(F,\tcK)$ of the data $(F,\cK)$ for $\cA=\ccurl$; see \autoref{def:compatibleApprox}.
As there exists a unique $b\in\RR^d$ such that $\ppcc(\xi^*)a=b\odot \xi^*$, 
the elements of $\tcK=\{\ta_0,\ta_1\}\subset\rddsymd{2}$ are explicitly given by
\begin{equation}\label{eq:tildeWellsCCurl}
    \ta_0 = F-\ttheta_{\cc}\,b\odot\xi^*, \hspace{2cm}\ta_1 = F+(1-\ttheta_{\cc})b\odot\xi^*.
\end{equation}
Up to a change of variables (see \autoref{subsubsec:CCChangeOfVar} below),
we may assume without loss of generality that $\xi^*=e_1$ and $\Omega=(0,1)^2$. 
Then, repeating the arguments from \autoref{subsec:branchingConstruction}, it is possible to adapt the unit cell construction (\autoref{lem:aa1}),
the cut-off layer (\autoref{lem:aa2}) and the branching construction (\autoref{prop:a3}) to the setting of $\cA=\ccurl$.
This approach yields, for all $\eps\in(0,1)$, a linear strain $u_\eps := F + \nabla^{sym} v_\eps \in\cD_F^{\cc}(\Omega)$ and a phase arrangement $\chi_\eps\in BV(\Omega;\cK)$,
which allows us to extract the excess energy
\begin{equation}\label{eq:EnergySplit3}
  E^{\cc}_\eps(F,\cK) = E^{\cc}_0(F,\cK) +  E^{\cc}_\eps(F,\tcK)\quad\forall\eps\in(0,1)
\end{equation}
and simultaneously prove the upper bound
\begin{equation}\label{eq:EnergySplit4}
  E^{\cc}_\eps(F,\tcK)\lesssim\eps^{\nicefrac{2}{3}}\quad\forall\eps\in(0,1).
\end{equation}
A key observation is that $u_\eps$ oscillates between $\ta_0$ and the perturbations of $\ta_1$ (corresponding to \eqref{eq:perturbedtildewells} with
$E_2=\gamma b\odot e_2$) while $\chi_\eps$ mirrors this oscillatory behavior with $\chi_\eps(x)=a_0$ if $u_\eps(x)\approx \ta_0$ and $\chi_\eps(x)=a_1$ if $u_\eps(x)\approx \ta_1$.\medbreak

Rather than redoing the branching construction in the setting $\cA=\ccurl$ from scratch, this observation allows us to proceed via an equivalent argument:
The idea of this approach is to construct a set $\cK'=\{a_0',a_1'\}\subset\rddd{2}$ such that the symmetric parts of the displacement gradient and of the phase arrangement,
arising from the upper bound construction for
the \textit{gradient} two-well energy $E^{\curl}_\eps(F,\cK')$ from \autoref{sec:UpperBoundGradDiv}, yield the desired construction for the
\textit{geometrically linear} energy $E^{\cc}_\eps(F,\cK)$.
We now make rigorous this approach in the following lemma, which we then use to prove \autoref{prop:UpperBoundCC1}.
\begin{lemma}\label{lem:UpperBoundCC}
Assume the hypotheses of \autoref{prop:UpperBoundCC1}.
In addition, let $\cK'=\{a_0',a_1'\}\subset\rddd{2}$ be another set of wells.
For the differential operator $\cA=\curl$ from \eqref{eq:curl}, let $E_{0}^{\curl}(F,\cK')$ be given as in \eqref{eq:MinSingPertAfreeEnergy}.
For $a':=a'_1-a'_0$, let $S_{\curl}(a')$ be the set of optimal lamination directions; see \autoref{def:optLaminationDirections}.
Let $\ttheta_{\curl}(F,\cK')$ be given as in \autoref{prop:optvolfrac}.\\
Assume that the following conditions are satisfied:
\begin{tasks}[label=\upshape(C\arabic*)](2)
  \task$\xi^*\in S_{\curl}(a')$,\label{C1}
  \task$\ttheta_{\curl}(F,\cK')\in(0,1),$\label{C2}
  \task$\sym a_j'=a_j\text{ for both }j\in\{0,1\},$\label{C3}
  \task$E_{0}^{\curl}(F,\cK')=E_{0}^{\cc}(F,\cK).$\label{C4}
\end{tasks}
Then, for all $\eps\in(0,1)$, there exist a vector field $v_\eps\in W^{1,\infty}_0(\Omega;\RR^2)$ and a phase arrangement
$\chi_\eps\in BV(\Omega; \cK)$ such that, for $u_\eps := F + \nabla^{sym} v_\eps\in\cD_F^{\cc}(\Omega)$, the following energy estimate holds:
\begin{equation}\label{eq:UpperBoundCC2}
  E^{\cc}_\eps(u_\eps,\chi_\eps) - E^{\cc}_0(F,\cK) \leq C\eps^{\nicefrac{2}{3}}\quad\forall \eps\in (0,1)
\end{equation}
for some constant $C=C(a',\ttheta_{\curl}(F,\cK'))>0$.
\end{lemma}
\begin{proof}
Since $\Omega$ is a rotated unit square with two faces normal to $\xi^*$ and both \ref{C1} and \ref{C2} are satisfied,
we can apply \autoref{prop:a4} for the data $(F,\cK')$ after a change of variables, as demonstrated in the first step of the proof of \autoref{thm:incompgrad};
see \autoref{subsec:ChangeOfVar}.
From this, we obtain, for all $\eps\in(0,1)$, a vector field $v_\eps\in W^{1,\infty}_0(\Omega;\RR^2)$ and a phase arrangement
$\chi'_\eps\in BV(\Omega; \cK')$ such that, for $u'_\eps := F + \nabla v_\eps\in\cD_F^{\curl}(\Omega)$, the following energy estimate holds:
\begin{equation}\label{eq:a4Estimate2}
  E^{\curl}_\eps(u'_\eps,\chi'_\eps) - E^{\curl}_0(F,\cK') \leq C\eps^{\nicefrac{2}{3}}\quad\forall \eps\in (0,1)
\end{equation}
for some constant $C=C(a',\ttheta_{\curl}(F,\cK'))>0$.
For $\eps\in(0,1)$, we define the maps $u_\eps\in\cD_F^{\cc}(\Omega)$ and $\chi_\eps\in BV(\Omega;\cK)$ by taking the symmetric part:
\begin{equation}
        u_\eps:=\sym(u'_\eps)= F + \nabla^{sym} v_\eps,\hspace{2cm} \chi_\eps:=\sym(\chi'_\eps),
\end{equation}
where $\chi_\eps$ takes values in $\cK$ due to \ref{C3}.
Using that $\nabla\chi_\eps=\sym(\nabla\chi'_\eps)$ in $\Omega$ and that $\abs{\sym M}\leq \abs{M}$ for all $M\in\rddd{2}$, we find
\begin{equation}\label{eq:EstUppBoundTransitionToGrad2}
\begin{aligned}
       E_{\eps}^{\cc}(u_\eps,\chi_\eps) &= \int_\Omega \lvert u_\eps-\chi_\eps\rvert^2\,dx + \eps \norm{\nabla\chi_\eps}_{TV(\Omega)}\\
       &\leq \int_\Omega \lvert u'_\eps-\chi'_\eps\rvert^2\,dx + \eps \norm{\nabla\chi'_\eps}_{TV(\Omega)}
       = E_{\eps}^{\curl}(u'_\eps,\chi'_\eps)  \quad\forall\eps\in(0,1).
\end{aligned}
\end{equation}
In combination with \ref{C4}, the estimates \eqref{eq:a4Estimate2} and \eqref{eq:EstUppBoundTransitionToGrad2} yield the desired upper bound \eqref{eq:UpperBoundCC2}.
\end{proof}
We are now ready to prove \autoref{prop:UpperBoundCC1}.
\begin{proof}[Proof of \autoref{prop:UpperBoundCC1}]
For the proof, we distinguish two cases: In the first case, we prove the assertion under the assumption that $a$ is (positive or negative) semidefinite.
In the second case, we deal with the situation, where $a$ is indefinite.
In both cases, we denote by $\lambda_-=\lambda_-(a)$ and $\lambda_+=\lambda_+(a)$ the smallest and largest eigenvalues of $a$, respectively.
\Case{1: The difference of the wells is semidefinite}\\
Suppose that $a$ is semidefinite.
By exchanging the roles of $a_0$ and $a_1$, we may assume without loss of generality that $a$ is positive semidefinite.\medbreak

We aim to apply \autoref{lem:UpperBoundCC} with $a_0':=a_0$ and $a_1':=a_1$, considering the two-well energy for the symmetrized gradient
and the two-well energy for the gradient with the same data $(F,\cK)$. It is immediate that \ref{C3} is satisfied.
We now verify the remaining conditions.\medbreak

Since $a$ is positive semidefinite, we have $0\leq\lambda_-\leq\lambda_+$.
Due to \autoref{prop:optlaminationdirectioncurlcurl} and the spectral theorem, our assumption $\xi^*\in S_{\cc}(a)$ implies that
\begin{equation}
    a=\lambda_-\,\xi^*_\perp\otimes\xi^*_\perp + \lambda_+\,\xi^*\otimes\xi^* \text{ and } g_{\cc}(a)=\lambda_+^2,
\end{equation}
where $\xi^*_\perp=(-\xi^*_2,\xi^*_1)\in\RR^2$. Next, we note that
\begin{equation}
    a^Ta=\lambda_-^2\,\xi^*_\perp\otimes\xi^*_\perp + \lambda_+^2\,\xi^*\otimes\xi^*.
\end{equation}
By \autoref{lem:EquicompatibleStatesCurl}, we obtain
\begin{equation}
    S_{\curl}(a)=S_{\cc}(a) \text{ and } g_{\curl}(a)=\lambda_{max}(a^Ta) = \lambda_+^2 = g_{\cc}(a),
\end{equation}
where $\lambda_{max}(a^Ta)$ denotes the largest eigenvalue of $a^Ta$. 
This observation yields \ref{C1}.
From \eqref{eq:compquanteq}, it follows that
\begin{equation}
    h_{\cc}(a)=\abs{a}^2-g_{\cc}(a) = \abs{a}^2-g_{\curl}(a) = h_{\curl}(a).
\end{equation}
Using \autoref{thm:qwdom} and \autoref{prop:optvolfrac}, this implies that
\begin{equation}
\begin{gathered}
        E_{0}^{\curl}(F,\cK)=\min_{\theta\in[0,1]} \Big(\abs{F-a_{\theta}}^2 +\theta (1-\theta)h_{\curl}(a)\Big)=  E_{0}^{\cc}(F,\cK),\\
        \ttheta_{\curl}(F,\cK) =\argmin_{\theta\in[0,1]} \Big(\abs{F-a_{\theta}}^2 +\theta (1-\theta)h_{\curl}(a)\Big)=\ttheta_{\cc}(F,\cK).
\end{gathered}
\end{equation}
As we assumed that $\ttheta_{\cc}(F,\cK)\in(0,1)$, both \ref{C2} and \ref{C4} are verified.
Finally, an application of \autoref{lem:UpperBoundCC} concludes the proof for this case,
where we note that the constant $C$ in \eqref{eq:UpperBoundCC1} only depends on the parameters $a$ and $\ttheta_{\cc}(F,\cK)$.
\Case{2: The difference of the wells is indefinite}\\
We now suppose that $a$ is indefinite and so $\lambda_- < 0 < \lambda_+$. For this part of the proof, we will often write $\ttheta_{\cc}$ as shorthand for $\ttheta_{\cc}(F,\cK)$
with the understanding that it is always considered with respect to the data $(F,\cK)$.\smallbreak

The spectral decomposition now yields an orthonormal basis $(\xi_-,\xi_+)\subset\RR^2$ such that
\begin{equation}
    a=\lambda_-\,\xi_-\otimes\xi_- + \lambda_+\,\xi_+\otimes\xi_+.
\end{equation}
The wells are compatible with respect to $\cA=\ccurl$ since
\begin{equation}\label{eq:aCompatibleUpperBoundCC}
    a = \underbrace{\left(\sqrt{\lambda_+}\,\xi_++\sqrt{-\lambda_-}\,\xi_-\right)}_{=:\zeta\in\RR^2}\odot\underbrace{\left(\sqrt{\lambda_+}\,\xi_+ -\sqrt{-\lambda_-}\,\xi_-\right)}_{=:\eta\in\RR^2}\in\Lambda_{\cc}.
\end{equation}
It follows that $S_{\cc}(a)=\{\pm \zeta,\pm\eta\}$, which is a special case of \autoref{prop:optlaminationdirectioncurlcurl}.
Moreover, it holds that $\ppcc(\xi)a=a$ for $\xi\in S^1$ if and only if $\xi\in S_{\cc}(a)$.
Since $\xi^*\in S_{\cc}(a)$, we have $\ppcc(\xi^*)a=a$ and there exists a unique $b\in\RR^d$ such that
\begin{equation}\label{eq:DifferenceOfWellsPfUpperBoundCC}
    a=b\odot\xi^*.
\end{equation}
We now proceed with the setup for \autoref{lem:UpperBoundCC}.
We define $\cK'=\{a_0',a_1'\}\subset\rddd{2}$ by setting
\begin{equation}
    a_0' := a_{\ttheta_{\cc}} - \ttheta_{\cc}\, b\otimes\xi^*,\hsp  a_1' := a_{\ttheta_{\cc}} + (1- \ttheta_{\cc}) b\otimes\xi^*,
\end{equation}
where $a_{\ttheta_{\cc}}$ is given as in \eqref{eq:atheta}. Since 
\begin{equation}\label{eq:aPrimeCompatibleUpperBoundCC}
    a':=a_1'-a_0'=b\otimes\xi^*\in\Lambda_{\curl},
\end{equation}
it follows that $S_{\curl}(a')=\{\pm \xi^*\}$ and \ref{C1} is satisfied.
Using \eqref{eq:DifferenceOfWellsPfUpperBoundCC}, we obtain \ref{C3} by observing that
\begin{equation}
    a_0 = a_{\ttheta_{\cc}} - \ttheta_{\cc}\, b\odot\xi^*,\hsp  a_1 = a_{\ttheta_{\cc}} + (1- \ttheta_{\cc}) b\odot\xi^*.
\end{equation}
Together \autoref{thm:qwdom} and \autoref{prop:optvolfrac} yield
\begin{equation}\label{eq:PfUpperBoundBaseEnergyCC}
    E_{0}^{\cc}(F,\cK) = \min_{\theta\in[0,1]}\Big( \big\lvert F-a_{\theta} \big\rvert^2 +\theta(1-\theta)h_{\cc}(a)\Big)
    = \min_{\theta\in[0,1]} \big\lvert F-a_{\theta} \big\rvert^2= \big\lvert F-a_{\ttheta_{\cc}} \big\rvert^2,
\end{equation}
where we used that $h_{\cc}(a)=0$ due to \eqref{eq:aCompatibleUpperBoundCC}.\medbreak

Our next goal is to compute $E_{0}^{\curl}(F,\cK')$ and $\ttheta_{\curl}(F,\cK')$.
Towards this goal, we define $a'_\theta:=(1-\theta)a'_0+\theta a'_1$ for $\theta\in[0,1]$ and compute
\begin{equation}
    a'_\theta = a'_{\ttheta_{\cc}} + (\theta -\ttheta_{\cc}) a' = a_{\ttheta_{\cc}} + (\theta -\ttheta_{\cc}) b\otimes\xi^* \quad\forall\theta\in[0,1].
\end{equation}
Decomposing into the symmetric and skew-symmetric part, we find
\begin{multline}
    a'_\theta =  \left[a_{\ttheta_{\cc}} + (\theta -\ttheta_{\cc}) b\odot\xi^*\right] + \left[(\theta -\ttheta_{\cc}) (b\odot\xi^*-b\otimes\xi^*)\right]\\
    = \underbrace{a_\theta}_{\in\rddsymd{2}} + (\theta -\ttheta_{\cc}) \underbrace{(b\odot\xi^*-b\otimes\xi^*)}_{\in\rddd{2}_{skew}}    \quad\forall\theta\in[0,1].
\end{multline}
As the decomposition $\rddd{2}=\rddsymd{2}\oplus\rddd{2}_{skew}$ is orthogonal with respect to the Frobenius inner product and $F$ is symmetric, it follows that
\begin{equation}\label{eq:SkewSymDecomp}
    \big\lvert F-a'_{\theta} \big\rvert^2  = \big\lvert F-a_{\theta} \big\rvert^2 + (\theta -\ttheta_{\cc})^2  \big\lvert  (b\odot\xi^*-b\otimes\xi^*)\big\rvert^2\quad\forall\theta\in[0,1].
\end{equation}
Finally, as \eqref{eq:aPrimeCompatibleUpperBoundCC} implies $h_{\curl}(a')=0$, \autoref{thm:qwdom} yields
\begin{align}
        E_{0}^{\curl}(F,\cK') &= \min_{\theta\in[0,1]}\Big(\big\lvert F-a'_{\theta} \big\rvert^2 +\theta(1-\theta)h_{\curl}(a')\Big)\\
        &=\min_{\theta\in[0,1]} \Big(\big\lvert F-a_{\theta} \big\rvert^2 + (\theta -\ttheta_{\cc})^2  \big\lvert  (b\odot\xi^*-b\otimes\xi^*)\big\rvert^2\Big)
        = \big\lvert F-a_{\ttheta_{\cc}} \big\rvert^2 = E_{0}^{\cc}(F,\cK),
\end{align}
where we used \eqref{eq:PfUpperBoundBaseEnergyCC} to deduce that both terms on the right-hand side of \eqref{eq:SkewSymDecomp}
attain their minima over $\theta\in[0,1]$ at $\ttheta_{\cc}$. This argument also proves $\ttheta_{\curl}(F,\cK') = \ttheta_{\cc}\in(0,1)$
and so both conditions \ref{C2} and \ref{C4} are verified.\smallbreak
The desired estimate \eqref{eq:UpperBoundCC1} now follows from \autoref{lem:UpperBoundCC}, where the constant $C$ only depends on the parameters $a$ and $\ttheta_{\cc}$ since
these uniquely determine $a'$ and $\ttheta_{\curl}(F,\cK')$.
\end{proof}

\subsection{An \texorpdfstring{$\eps^{\frac{4}{5}}$}{eps4/5}-upper bound via the Chan-Conti branching construction}\label{subsec:branchingConstruction2}
Our next goal is to establish an improved upper bound in the case where the wells differ by a rank-one matrix.
This improved upper bound is due to Chan and Conti \cite[Theorem 1.2]{cc15} and the key result of \autoref{subsec:branchingConstruction2}.
Although it is a direct consequence of the arguments in \cite{cc15}, we provide a proof for completeness since we consider a slightly different setting.
Finally, after proving \autoref{prop:UpperBoundCCImproved}, we conclude \autoref{subsec:branchingConstruction2} by completing the proof of \autoref{thm:incompsyma}.

\begin{breaktheorem}[Improved upper bound of the geometrically linear two-well energy]\label{prop:UpperBoundCCImproved}
Assume the hypotheses of \autoref{prop:UpperBoundCC1}. In addition, suppose that $\rank a = 1$.
Then, there exists a constant $C=C(a,\ttheta_{\cc})>0$ such that
\begin{equation}\label{eq:UpperBoundCCImproved}
  E^{\cc}_\eps(F,\cK) - E^{\cc}_0(F,\cK) \leq C\eps^{\nicefrac{4}{5}}\quad\forall \eps\in (0,1).
\end{equation}
\end{breaktheorem}
The proof requires some preparation and follows the same structure as the proof of \autoref{thm:incompgrad}~\textcolor{blue}{(iii)} in \autoref{sec:UpperBoundGradDiv}.
The argument is divided into two parts, corresponding to the following sections:\smallbreak
In \autoref{subsubsec:CCChangeOfVar}, we state \autoref{prop:UpperBoundCCImprovedSimplify}, which serves as an auxiliary result and reduces the problem to a simplified setting.
Using a change of variables, we show that \autoref{prop:UpperBoundCCImprovedSimplify} indeed implies \autoref{prop:UpperBoundCCImproved}.
It thus suffices to prove \autoref{prop:UpperBoundCCImprovedSimplify}, to which we turn to in \autoref{subsubsec:CCBranchingConst}. 
There, we recall the unit-cell construction (\autoref{lem:unitCellCC}) and the cut-off layer (\autoref{lem:CutOffCC}) from \cite{cc15}, which we then use to
prove \autoref{prop:UpperBoundCCImprovedSimplify}.

\subsubsection{Reduction to a simplified setting}\label{subsubsec:CCChangeOfVar}
We now show that for the proof of \autoref{prop:UpperBoundCCImproved},
we may assume without loss of generality that the coordinate direction $e_1$ is an optimal lamination direction.
In this simplified setting, we have the following upper bound construction.
\begin{breakproposition}[The Chan-Conti branching construction]\label{prop:UpperBoundCCImprovedSimplify}
Let $\cK=\{a_0,a_1\}\subset\rddsymd{2}$ and $F\in\rddsymd{2}$.
For the differential operator $\cA=\ccurl$ from \eqref{eq:curlcurl},
consider the energy $E^{\cc}_\eps(F,\cK)$ in the domain $Q=(0,1)^2$ given by \eqref{eq:MinSingPertAfreeEnergy}.
Let $\ttheta_{\cc}=\ttheta_{\cc}(F,\cK)$ be given as in \autoref{prop:optvolfrac}.
Suppose that $\ttheta_{\cc}\in(0,1)$ and that $a:=a_1-a_0= e_1\otimes e_1$.\\
Then, for all $\eps\in(0,1)$, there exist a vector field $v_\eps\in W^{1,\infty}_0(Q;\RR^2)$ and a phase arrangement
$\chi_\eps\in BV(Q; \cK)$ such that, for $u_\eps := F + \nabla^{sym} v_\eps\in\cD_F^{\cc}(Q)$, the following energy estimate holds:
\begin{equation}\label{eq:UpperBoundCCImprovedSimplify}
  E^{\cc}_\eps(u_\eps,\chi_\eps) - E^{\cc}_0(F,\cK) \leq C\eps^{\nicefrac{4}{5}}\quad\forall \eps\in (0,1)
\end{equation}
for some constant $C=C(a,\ttheta_{\cc})>0$.
\end{breakproposition}
We postpone the proof of \autoref{prop:UpperBoundCCImprovedSimplify}, which is the content of \autoref{subsubsec:CCBranchingConst}.
Assuming the validity of \autoref{prop:UpperBoundCCImprovedSimplify}, we now prove the key result \autoref{prop:UpperBoundCCImproved}.

\begin{proof}[Proof of \autoref{prop:UpperBoundCCImproved}]
The proof is organized into two steps. In the first step, we recollect the assumptions of \autoref{prop:UpperBoundCCImproved} and establish some preliminary observations.
In the second step, we use a change of reference frame to apply \autoref{prop:UpperBoundCCImprovedSimplify}.
\Step{1: Preliminary observations}
Let $\cK=\{a_0,a_1\}\subset\rddsymd{2}$ and $F\in\rddsymd{2}$ with $\ttheta_{\cc}(F,\cK)\in(0,1)$ and $\xi^*\in S_{\cc}(a)$ for $a=a_1-a_0$.
In addition, we suppose that $\Omega$ is a rotated unit square with two faces normal to $\xi^*$ and that the matrix $a$ has rank one.\smallbreak
Let $\lambda>0$ be the nonzero eigenvalue of $a$.
Then, \autoref{prop:optlaminationdirectioncurlcurl} implies that $\xi^*$ is an eigenvector of $a$ associated with the eigenvalue $\lambda$. In particular, there holds
\begin{equation}\label{eq:aRankOne}
  a=\lambda \xi^*\otimes \xi^*.
\end{equation}
Without loss of generality, we may assume that $\lambda=1$.
To see this, set $\hF:=\tfrac{1}{\lambda}F$ and $\hcK:=\tfrac{1}{\lambda}\cK$ and note that the scaling in \eqref{eq:UpperBoundCCImproved}
in independent of $\lambda$ since
\begin{equation}\label{eq:rescalingEnergy}
    E^{\cc}_\eps(F,\cK)- E^{\cc}_0(F,\cK) = \lambda^2 \big(E^{\cc}_\eps(\hF;\hcK) - E^{\cc}_0(\hF;\hcK) \big) \quad\forall\eps\geq 0,
\end{equation}
which follows by the rescaling: 
\begin{equation}
    \begin{aligned}
        u\in \cD_F^{\cc}(\Omega) &\rightsquigarrow \tfrac{1}{\lambda}u\in \cD_{\hF}^{\cc}(\Omega),\\
        \chi\in BV(\Omega;\cK) &\rightsquigarrow \tfrac{1}{\lambda}\chi \in BV(\Omega;\hcK).
    \end{aligned}
\end{equation}
In this context, it should be emphasized that the constant $C$ in \eqref{eq:UpperBoundCCImproved} depends on $\lambda$.
However, this is already accounted for through its dependence on $a$.

\Step{2: Change of reference frame}
Now, the idea of the proof is to derive an equivalent two-well problem for the data $(F',\cK')$ in the standard unit square, which allows us to prove the assertion
by an application of \autoref{prop:UpperBoundCCImprovedSimplify}.
The data $(F',\cK')$ is obtained by a change of variables similar to the one in first step of the proof of \autoref{thm:incompgrad}~\textcolor{blue}{(iii)}
in \autoref{subsec:ChangeOfVar}.
However, the transformation is slightly different as it corresponds to a change of reference frame, which we outline next.\medbreak

Let $R\in SO(2)$ be a rotation with 
\begin{equation}\label{eq:ReOneIsXi}
  R e_1 = \xi^*.
\end{equation}
Since $E^{\cc}_\eps(F,\cK;\Omega)$ is invariant under translations of $\Omega$, we may assume without loss of generality that
\begin{equation}\label{eq:ChangeofVarDomain2}
  \Omega = R Q, \text{ where } Q=(0,1)^2.
\end{equation}
We consider a change of reference frame, described by the transformation $x=Ry$ for $y\in Q$.
Drawing from the modeling perspective, we interpret a vector field $v:\Omega\to \RR^2$ as a displacement field,
which indicates the displacement of the elastic material at each point in $\Omega$.
Expressing the domain and the codomain of $v$ in the $y$-coordinates,
the displacement field transforms under the change of reference frame as an objective vector:
\begin{equation}\label{eq:objectiveVector}
  v'(y)=R^T v(Ry)\quad y\in Q.
\end{equation}
Taking the derivative, we see that the displacement gradient transforms as an objective tensor:
\begin{equation}\label{eq:objectiveTensor}
  \nabla v'(y) = R^T (\nabla v)(Ry) R\quad y\in Q.
\end{equation}
\medbreak

Based on the transformation rule \eqref{eq:objectiveTensor}, we define
$F' := R^TFR$ and $\cK':=\{a_0',a_1'\}$ with $a_j':=R^T a_j R$. Setting $a':=a_1'-a_0'$, it follows from \eqref{eq:aRankOne} and \eqref{eq:ReOneIsXi} that
\begin{equation}\label{eq:a'ise1dote1}
  a'= R^T a R = R^T(\xi^*\otimes \xi^*)R =  (R^T\xi^*)\otimes (R^T\xi^*) =  e_1\otimes e_1.
\end{equation}
Since $a,a'\in\Lambda_{\cc}$, we have $h_{\cc}(a)=h_{\cc}(a')=0$, which yields
\begin{equation}\label{eq:ChangeofVarQuasiConv2}
    \abs{F'-a'_\theta}^2+\theta(1-\theta)h_{\cc}(a') = \big\lvert{F-a_\theta}\big\rvert^2+\theta(1-\theta)h_{\cc}(a)\quad\forall\theta\in[0,1],
\end{equation}
where $a_\theta'$ and $a_\theta$ are given as in \eqref{eq:atheta} and we used that the Frobenius norm has the property 
\begin{equation}\label{eq:FrobProp2}
  \abs{R^TAR}=\big\lvert{A}\big\rvert\quad\forall A\in\rddd{2},\, R\in SO(2).
\end{equation}
By \autoref{thm:qwdom} and \autoref{prop:optvolfrac}, it follows from \eqref{eq:ChangeofVarDomain2} and \eqref{eq:ChangeofVarQuasiConv2} that
\begin{equation}\label{eq:ChangeofVarOptVolFracCC}
  E^{\cc}_0(F',\cK';Q)=E^{\cc}_0(F,\cK;\Omega), \hspace{1.5cm}  \ttheta_{\cc}(F',\cK') = \ttheta_{\cc}(F,\cK) \in(0,1).
\end{equation}

We apply \autoref{prop:UpperBoundCCImprovedSimplify} for the data $(F',\cK')$ and obtain, for all $\eps\in(0,1)$,
a vector field $v'_\eps\in W^{1,\infty}_0(Q;\RR^2)$ and a phase arrangement
$\chi'_\eps\in BV(Q; \cK')$ such that $u'_\eps := F' + \nabla^{sym} v'_\eps\in\cD_{F'}^{\cc}(Q)$ satisfies
\begin{equation}\label{eq:UpperBoundCCImprovedSimplify2}
  E^{\cc}_\eps(u'_\eps,\chi'_\eps;Q) - E^{\cc}_0(F',\cK';Q) \leq C\eps^{\nicefrac{4}{5}}\quad\forall \eps\in (0,1)
\end{equation}
for some constant $C=C(a',\ttheta_{\cc}(F',\cK'))>0$.\medbreak
For $\eps\in(0,1)$, we now define $v_\eps\in W^{1,\infty}_0(\Omega;\RR^2)$ and $\chi_\eps\in BV(\Omega;\cK)$ 
by expressing $v'_\eps$ and $\chi_\eps'$ in terms of the $x$-coordinates. Specifically, for $x\in\Omega$, we set
\begin{equation}\label{eq:ObjectiveVectorCC}
    v_\eps(x):=R v'_\eps(R^Tx),\hspace{2cm} \chi_\eps(x):=R \chi'_\eps(R^Tx) R^T.
\end{equation}
Defining $u_\eps:=F+\nabla^{sym}v_\eps$ for $\eps\in(0,1)$, it follows that $u_\eps\in\cD_{F}^{\cc}(\Omega)$ and that
\begin{equation}\label{eq:ObjectiveTensorCC}
  \begin{aligned}
    u_\eps(x) &= R F' R^T + \sym\big(R (\nabla v'_\eps)(R^Tx) R^T\big)\\
    &= R \big[F' +(\nabla^{sym}v'_\eps)(R^Tx)\big] R^T = R u'_\eps(R^Tx) R^T\quad\forall x\in\Omega.
  \end{aligned}
\end{equation}
Note that in the second identity, we used that $\sym(RAR^T) = R \sym(A) R^T$ for all $A, R\in\rddd{2}$.\medbreak

Using \eqref{eq:FrobProp2}, \eqref{eq:ObjectiveVectorCC} and \eqref{eq:ObjectiveTensorCC}, we compute
\begin{equation}\label{eq:ChangeofVarScalingCC}
    \begin{aligned}
      E^{\cc}_\eps(u_\eps,\chi_\eps;\Omega)&= \int_\Omega \abs{u_\eps(x)-\chi_\eps(x)}^2\,dx + \eps \norm{\nabla \chi_\eps}_{TV(\Omega)}\\
      &= \int_Q \abs{R[u'_\eps(y)-\chi'_\eps(y)]R^T}^2\,dy + \eps \norm{\nabla \chi_\eps'}_{TV(Q)} = E^{\cc}_\eps(u'_\eps,\chi'_\eps;Q)\quad\forall \eps\in (0,1).
    \end{aligned}
\end{equation}
Together with \eqref{eq:ChangeofVarOptVolFracCC} and \eqref{eq:UpperBoundCCImprovedSimplify2}, this allows us to infer
\begin{equation}\label{eq:ChangeofVarScalingCC2}
  E^{\cc}_\eps(u_\eps,\chi_\eps;\Omega) - E^{\cc}_0(F,\cK;\Omega) \leq C\eps^{\nicefrac{4}{5}}\quad\forall \eps\in (0,1).
\end{equation}
The constant $C$ only depends on the parameters $a$ and $\ttheta_{\cc}(F,\cK)$ since these uniquely determine $a'$ and $\ttheta_{\cc}(F',\cK')$.
This shows \eqref{eq:UpperBoundCCImproved} and concludes the proof.
\end{proof}

\subsubsection{The Chan-Conti branching construction}\label{subsubsec:CCBranchingConst}
Our next goal is to show the improved upper bound (\autoref{prop:UpperBoundCCImprovedSimplify}) by recalling the Chan-Conti branching construction \cite{cc15}.
This construction exploits the vectorial nature of the problem to construct a linear strain $u=F+\nabla^{sym}v\in\cD_F^{\cc}(Q)$
where the components of the displacement field $v\in W_0^{1,\infty}(Q;\RR^2)$ interact as to optimize the off-diagonal entries of the symmetrized gradient.  
This results in an additional reduction of the energy in comparison to the upper bound construction \autoref{prop:UpperBoundCC1},
where there is no such interaction and the vector field only takes values in a one-dimensional subspace of $\RR^2$;
see \eqref{eq:gradOfVectorfieldBranchConstr} and \eqref{eq:DisplacementMapCutOffLayer}.\medbreak

Before we come to the branching construction, we outline how to deal with incompatible boundary data.
Assume the hypotheses of \autoref{prop:UpperBoundCCImprovedSimplify}.
Based on the strategy detailed in \autoref{rmk:DisplGradComp}, we disregard the data $(F,\cK)$ and instead consider the Chan-Conti branching construction
for a compatible approximation $(F,\tcK)$.
As $S_{\cc}(e_1\otimes e_1)=\{\pm e_1\}$, the compatible approximation is unique with $\tcK=\{\ta_0,\ta_1\}$ determined by
\begin{equation}\label{eq:tildeWellsCCurlSpecific}
    \ta_0 = F-\ttheta_{\cc}\, e_1\otimes e_1, \hspace{2cm}\ta_1 = F+(1-\ttheta_{\cc}) e_1\otimes e_1.
\end{equation}
We therefore aim to construct a linear strain $u=F+\nabla^{sym} v \in \cD_F^{\cc}(Q)$
that takes values close to $\ta_0$ and $\ta_1$ with volume proportion $\ttheta_{\cc}$ for the $\ta_1$-phase.
The following lemma shows that this approach yields an improved (localized) energy estimate.

\begin{breaklemma}[Unit cell construction]\label{lem:unitCellCC}
Assume the hypotheses of \autoref{prop:UpperBoundCCImprovedSimplify}.
Let $\omega=(-l,l)\times(0,h)$ with $0<l\leq h\leq 1$.\\
Then, there exist a vector field $v\in W^{1,\infty}(\omega;\RR^2)$ and a phase arrangement $\chi\in BV(\omega;\cK)$
such that, defining $u=F+\nabla^{sym} v$, the localized energy can be estimated by
\begin{equation}\label{eq:unitCellCCEst}
  E^{\cc}_\eps(u,\chi;\omega) := \int_\omega \abs{u-\chi}^2\,dx + \eps \norm{\nabla \chi}_{TV(\omega)} \leq 
  2lh E^{\cc}_0(F,\cK) + C\left(\frac{l^5}{h^3} + \eps h\right) \quad\forall\eps>0
\end{equation}
for some constant $C=C(a, \ttheta_{\cc})>0$. Moreover, the vector field has the following boundary values:
\begin{enumerate}
  \item (Lateral sides) For all $x_2\in[0,h]$, we have 
  \begin{equation}\label{eq:sidestraceCC}
    v(-l,x_2)= v(l,x_2)=0.
  \end{equation}
  \item (Bottom and top side) For all $x_1\in[-l,l]$, it holds that 
  \begin{equation}\label{eq:TopBottomTraceCC}
  v(x_1,0)= \varphi_l(x_1)e_1, \hspace{2cm} v(x_1,h)= \varphi_{{l}/{2}}(x_1-\tfrac{l}{2})e_1,
  \end{equation}
  where $\varphi_l$ is given as in \eqref{eq:varphirBoundaryData} but with $\ttheta_{\curl}$ replaced by $\ttheta_{\cc}$.
\end{enumerate}
\end{breaklemma}
\begin{proof}
We follow the proof of \cite[Lemma 2.1]{cc15}, proceeding in two steps:
In the first step, we construct $u$ and $\chi$. In the second step, the localized energy is estimated.
\Step{1: Linear strain and phase arrangement}
For brevity, we omit the subscript $\cc$ of $\ttheta_{\cc}$ and introduce the parameter
\begin{equation}\label{eq:alphaParamCC}
    \alpha:=\frac{(1-\ttheta)l}{2}.
\end{equation}
Let $\gamma:[0,1]\to[0,1]$ be a smooth interpolation function satisfying 
\begin{equation}\label{eq:interpolFunction}
  \gamma(0)=0,\; \gamma(1)=1\text{ and }\gamma'(0)=\gamma'(1)=0.
\end{equation}
Similar to the partition illustrated in \autoref{fig:branching}, we partition $\omega$ into $(\omega_i)_{i\leq5}$ by setting
\begin{equation}
    \begin{gathered}
      \omega_1 := \{x\in\omega: x_1\in (-l,\,  -\ttheta l - \alpha\gamma (\tfrac{x_2}{h}))\},\hspace{1.1cm}
      \omega_2 := \{x\in\omega: x_1\in (-\ttheta l - \alpha\gamma (\tfrac{x_2}{h}),\, - \alpha\gamma (\tfrac{x_2}{h}))\},\\
      \hspace{-0.58cm}\omega_3 := \{x\in\omega: x_1\in (- \alpha\gamma (\tfrac{x_2}{h}),\, \alpha\gamma (\tfrac{x_2}{h}))\},\hspace{1.13cm}
      \omega_4 := \{x\in\omega: x_1\in (\alpha\gamma (\tfrac{x_2}{h}),\, \ttheta l + \alpha\gamma (\tfrac{x_2}{h}))\},\\
      \omega_5 := \{x\in\omega: x_1\in (\ttheta l+ \alpha\gamma (\tfrac{x_2}{h}),\, l)\}.\label{eq:omegapartsCC}
    \end{gathered}
\end{equation}
These sets satisfy the following volume proportions:
\begin{equation}\label{eq:volumeproportionsCC}
  \abs{\omega_2}=\abs{\omega_4}=\frac{\ttheta}{2}\abs{\omega},\qquad\qquad \abs{\omega_1\cup\omega_3\cup\omega_5} = (1-\ttheta)\abs{\omega}.
\end{equation}
As the linear strain $u=F+\nabla^{sym} v$ should take values close to the states $\ta_0$ and $\ta_1$ from \eqref{eq:tildeWellsCCurlSpecific}
with volume proportion $\ttheta$ for the $\ta_1$-phase, we impose
\begin{equation}\label{eq:delv1CC}
    \del_1 v_1(x)=\begin{cases}
        - \ttheta & x\in\omega_1\cup\omega_3\cup\omega_5,\\
        1- \ttheta & x\in\omega_2\cup\omega_4.
    \end{cases}
\end{equation}
Integrating $\del_1 v_1(\cdot,x_2)$ from $-l$ to $x_1$ with $v_1(-l, x_2)=0$ according to the boundary condition \eqref{eq:sidestraceCC}, we obtain
\begin{equation}
  v_1(x):=\begin{cases}
    -\ttheta (l+x_1) & x\in \omega_1,\\
    (1-\ttheta) x_1+ \alpha\gamma(\tfrac{x_2}{h}) & x\in \omega_2,\\
    -\ttheta x_1 & x\in \omega_3,\\
    (1-\ttheta) x_1- \alpha\gamma(\tfrac{x_2}{h}) & x\in \omega_4,\\
    -\ttheta (x_1-l) & x\in \omega_5.
  \end{cases}
\end{equation}
We optimize the off-diagonal entries of the linear strain by imposing 
\begin{equation}\label{eq:offDiagonalCCConstraint}
    \del_2 v_1 + \del_1 v_2 = 0 \text{ in } \omega.
\end{equation}
To this end, we compute
\begin{equation}
  \del_2 v_1(x):=\begin{cases}
    0 & x\in \omega_1\cup\omega_3\cup\omega_5,\\
    \tfrac{\alpha}{h}\gamma'(\tfrac{x_2}{h}) & x\in \omega_2,\\
    -\tfrac{\alpha}{h}\gamma'(\tfrac{x_2}{h}) & x\in \omega_4.
  \end{cases}
\end{equation}
Again, setting $v_2(-l, x_2)=0$ according to the boundary condition \eqref{eq:sidestraceCC} and integrating $\del_1 v_2(\cdot,x_2) = -\del_2 v_1(\cdot,x_2)$ from $-l$ to $x_1$ yields
\begin{equation}
  v_2(x):=\begin{cases}
    0 & x\in \omega_1,\\
    -\tfrac{\alpha}{h}\gamma'(\tfrac{x_2}{h})(\ttheta l + \alpha \gamma(\tfrac{x_2}{h}) +x_1) & x\in \omega_2,\\
    -\tfrac{\alpha}{h}\gamma'(\tfrac{x_2}{h}) \ttheta l & x\in \omega_3,\\
    -\tfrac{\alpha}{h}\gamma'(\tfrac{x_2}{h}) (\ttheta l+\alpha \gamma(\tfrac{x_2}{h}) -x_1) & x\in \omega_4,\\
    0 & x\in \omega_5.
  \end{cases}
\end{equation}
Note that $v\in W^{1,\infty}(\omega;\RR^2)$ and satisfies the boundary conditions \eqref{eq:sidestraceCC} and \eqref{eq:TopBottomTraceCC} due to \eqref{eq:interpolFunction}.
As we will use the symmetrized gradient of $v$, we also compute
\begin{equation}\label{eq:delv2CC}
  \del_2 v_2(x):=\begin{cases}
    0 & x\in \omega_1,\\
    -\tfrac{\alpha}{h^2}\gamma''(\tfrac{x_2}{h})(\ttheta l + \alpha \gamma(\tfrac{x_2}{h}) +x_1) -[ \tfrac{\alpha}{h}\gamma'(\tfrac{x_2}{h})]^2& x\in \omega_2,\\
    -\tfrac{\alpha}{h^2}\gamma''(\tfrac{x_2}{h}) \ttheta l & x\in \omega_3,\\
    -\tfrac{\alpha}{h^2}\gamma''(\tfrac{x_2}{h}) (\ttheta l+\alpha \gamma(\tfrac{x_2}{h}) -x_1) -[ \tfrac{\alpha}{h}\gamma'(\tfrac{x_2}{h})]^2& x\in \omega_4,\\
    0 & x\in \omega_5.
  \end{cases}
\end{equation}
Finally, we define the phase arrangement $\chi\in BV(\Omega;\cK)$ by setting
\begin{equation}
    \chi(x)=\begin{cases}
        a_0 & x\in \omega_1\cup\omega_3\cup\omega_5,\\
        a_1 & x\in \omega_2\cup\omega_4.
        \end{cases}
\end{equation}
\Step{2: Localized energy -- estimates \& extraction of the excess energy}
In this step, we show that the arguments in the proof of \autoref{lem:aa1} also apply for the Chan-Conti branching construction.
Let us first compute the excess energy. Since $a=e_1\otimes e_1\in\Lambda_{\cc}$, we have $h_{\cc}(a)=0$ and \autoref{thm:qwdom} implies
\begin{equation}
    E_0^{\cc}(F,\cK) = \min_{\theta\in[0,1]} \big\lvert F-a_{\theta}\big\rvert^2 = \abs{F-a_{\ttheta}}^2.
\end{equation}
Next, as in \eqref{eq:MeanExpansion}, we expand $\chi$ about its mean $\overline{\chi}=a_{\ttheta}$, which gives
\begin{equation}
    \chi = a_{\ttheta} + (\chi_1-\ttheta) a \text{ in }\omega, \text{ where }\chi_1 := \mathbbm{1}_{\{\chi=a_1\}}\in BV(\Omega).
\end{equation}
This allows us to split the localized elastic energy into three parts:
\begin{multline}\label{eq:stintCC}
  E^{\cc}_{el}(u,\chi;\omega):=
  \int_\omega \abs{u -  \chi }^2\,dx = \int_\omega \abs{F-  a_{\ttheta} +\nabla^{sym} v  - (\chi_1-\ttheta) a}^2\,dx\\
  = \underbrace{\int_\omega \abs{F-a_{\ttheta}}^2 \,dx}_{=2lh  E_0^{\cc}(F,\cK)}
  + 2 \underbrace{\int_\omega (F-a_{\ttheta}, \nabla^{sym} v - (\chi_1-\ttheta) a ) \,dx}_{=:I_2}+ \underbrace{\int_\omega \abs{\nabla^{sym} v - (\chi_1-\ttheta) a}^2\,dx}_{=:I_3}.
\end{multline}
With the excess energy extracted, we will now show that $I_2=0$ and $\abs{I_3}\lesssim {l^5}/{h^3}$.
To this end, using \eqref{eq:delv1CC}, \eqref{eq:offDiagonalCCConstraint} and $a=e_1\otimes e_1$, we infer that
\begin{equation}\label{eq:SymGradMinusChi}
  \nabla^{sym} v - (\chi_1-\ttheta) a = \begin{pmatrix}
    \del_1v_1 & 0\\
    0 & \del_2 v_2
  \end{pmatrix} - \begin{pmatrix}
    \chi_1-\ttheta & 0\\
    0 & 0
  \end{pmatrix} = \begin{pmatrix}
    0 & 0\\
    0 & \del_2 v_2
  \end{pmatrix}.
\end{equation}
As $v_2$ vanishes along the top and bottom boundary of $\omega$, we obtain $I_2=0$ by computing
\begin{equation}\label{eq:I2VanishCC}
  \int_\omega \del_2 v_2 \,dx = \int_{-l}^{l} v_2(x_1,h) - v_2(x_1,0) \,dx_1= 0.
\end{equation}
For the remainder of the proof, we denote by $C$ positive constants, which may vary but depend only on the parameters $a$ and $\ttheta$.
Using \eqref{eq:alphaParamCC}, \eqref{eq:delv2CC} and \eqref{eq:SymGradMinusChi}, we estimate 
\begin{equation}
\Big\lvert{\nabla^{sym} v(x) - (\chi_1(x)-\ttheta) a}\Big\rvert  \leq
\frac{\alpha}{h^2} \big\lvert\gamma''(\tfrac{x_2}{h})\big\rvert\Big(\big\lvert{\ttheta l}\big\rvert+\abs{\alpha \gamma(\tfrac{x_2}{h})}+\abs{x_1}\Big)
+ \frac{\alpha^2}{h^2}\big\lvert\gamma'(\tfrac{x_2}{h}) \big\rvert^2
\leq C \frac{l^2}{h^2} \quad\forall x\in\omega.
\end{equation}
By integration, we obtain the following bound
\begin{equation}\label{eq:I3EstCC}
  I_3 = \int_\omega \abs{\nabla^{sym} v - (\chi_1-\ttheta) a}^2\,dx \leq C \frac{l^5}{h^3}.
\end{equation}
Finally, we turn to the surface energy. By symmetry, we have
\begin{equation}\label{eq:aa1surfaceCC}
    \norm{\nabla \chi}_{TV(\omega)} = \abs{a}\per(\{\chi=a_0\};\omega) = \abs{a}\per(\omega_1\cup\omega_3\cup\omega_5;\omega)
    =  4 \abs{a}\length (\tgamma),
\end{equation}
where the curve $\tgamma :[0,h]\to \omega$ parametrizes the interface of $\omega_3$ and $\omega_4$ and is given by
\begin{equation}
  \tgamma(t):=(\alpha\gamma (\tfrac{t}{h}), t),\quad t\in[0,h].
\end{equation}
Using \eqref{eq:alphaParamCC} and $l\leq h$, we estimate
\begin{equation}
\begin{aligned}
     \length (\tgamma) = \int_{0}^{h}\abs{\tgamma'(t)}\, dt = \int_{0}^{h}  \sqrt{ 1+ \frac{\alpha^2 }{h^2}\gamma'^{\hspace{0.02cm} 2}\Big(\frac{t}{h}\Big)   \,}   \, dt 
      & =\int_{0}^{1}  \sqrt{ h^2+ \alpha^2 \gamma'^{\hspace{0.02cm} 2}(s)   \,}   \, ds\\
      &\leq\int_{0}^{1}  \sqrt{ h^2+ Cl^2  \,}   \, ds
      \leq C h.
\end{aligned}
\end{equation}
Together with \eqref{eq:stintCC}, \eqref{eq:I2VanishCC} and \eqref{eq:I3EstCC}, this shows the desired estimate \eqref{eq:unitCellCCEst} and concludes the proof.
\end{proof}
In the branching construction, we also need the following cut-off layer.
\begin{breaklemma}[Cut-off layer]\label{lem:CutOffCC}
Assume the hypotheses of \autoref{prop:UpperBoundCCImprovedSimplify}.
Let $\omega=(-l,l)\times(0,h)$ with $0<l\leq 2h\leq 1$. \\
Then, there exists $v\in W^{1,\infty}(\omega;\RR^2)$ and a phase arrangement $\chi\in BV(\omega;\cK)$
such that, defining $u=F+\nabla^{sym} v$, the localized energy can be estimated by
\begin{equation}\label{eq:CutOffCCEst2}
  E^{\cc}_\eps(u,\chi;\omega) = \int_\omega \abs{u-\chi}^2\,dx + \eps \norm{\nabla \chi}_{TV(\omega)} \leq 
  2lh E^{\cc}_0(F,\cK) + C\left(lh + \eps h\right) \quad\forall\eps>0
\end{equation}
for some constant $C=C(a, \ttheta_{\cc})>0$. Moreover, the vector field has the following boundary values:
\begin{enumerate}
  \item (Lateral sides) For all $x_2\in[0,h]$, we have 
  \begin{equation}\label{eq:sidestrace2CC}
    v(-l,x_2)= v(l,x_2)=0.
  \end{equation}
  \item (Bottom and top side) For all $x_1\in[-l,l]$, it holds that 
  \begin{equation}\label{eq:TopBottomTrace2CC}
  v(x_1,0)= \varphi_l(x_1)e_1, \hspace{2cm} v(x_1,h)=0,
  \end{equation}
  where $\varphi_l$ is given as in \eqref{eq:varphirBoundaryData} but with $\ttheta_{\curl}$ replaced by $\ttheta_{\cc}$.
\end{enumerate}
\end{breaklemma}
\begin{proof}
Instead of redoing the construction, we interpret the data in terms of a gradient two-well energy and apply \autoref{lem:aa2}.
To this end, we make some observations. Since $a=e_1\otimes e_1$ is positive semidefinite with $F\in\rddsymd{2}$ and $\cK\subset\rddsymd{2}$,
the proof of \autoref{prop:UpperBoundCC1} shows that
\begin{gather}
    S_{\curl}(a)  = S_{\cc}(a)= \{\pm e_1\},\hsp
    E^{\curl}_0(F,\cK) = E^{\cc}_0(F,\cK),\label{eq:EnergyCutOffCC}\\
    \ttheta_{\curl}(F,\cK) = \ttheta_{\cc}(F,\cK)\in(0,1).\label{eq:VolFracCutOffCC}
\end{gather}
As $\ppc(e_1)a=e_1\otimes e_1$, we can apply \autoref{lem:aa2} with $b=e_1$ and obtain $v\in W^{1,\infty}(\omega;\RR^2)$ and $\chi\in BV(\omega;\cK)$
such that, for $\tu := F +\nabla v$, we have
\begin{equation}\label{eq:aa1est2CC}
  E^{\curl}_\eps(\tu,\chi;\omega) = \int_\omega \abs{\tu-\chi}^2\,dx + \eps \norm{\nabla \chi}_{TV(\omega)} \leq 
  2lh E^{\curl}_0(F,\cK) + C\left(lh + \eps h\right) \quad\forall\eps>0
\end{equation}
for some constant $C=C(a, \ttheta_{\cc}(F,\cK))>0$.
The vector field $v$ satisfies the boundary conditions \eqref{eq:sidestrace2CC} and \eqref{eq:TopBottomTrace2CC} because $b=e_1$ and the optimal volume fractions are equal \eqref{eq:VolFracCutOffCC}.\medbreak

For $u=F+\nabla^{sym} v=\sym \tu$, we have
$\int_\omega \abs{u-\chi}^2\,dx \leq \int_\omega \abs{\tu-\chi}^2\,dx$
since $\chi=\sym \chi$ in $\omega$ and $\lvert \sym A\rvert\leq \lvert A\rvert$ for all $A\in\rddd{2}$.
Together with \eqref{eq:EnergyCutOffCC} and \eqref{eq:aa1est2CC}, this shows
\begin{equation}
  E^{\cc}_\eps(u,\chi;\omega) \leq  E^{\curl}_\eps(\tu,\chi;\omega)  \leq 2lh E^{\cc}_0(F,\cK) + C\left(lh + \eps h\right) \quad\forall\eps>0
\end{equation}
and completes the proof.
\end{proof}

With \autoref{lem:unitCellCC} and \autoref{lem:CutOffCC} in hand, we have the following upper bound construction.
\begin{breakproposition}[Branching construction]\label{prop:BranchingConstrCC}
Assume the hypotheses of \autoref{prop:UpperBoundCCImprovedSimplify}.
Then, for all $N\in\NN_{>1}$, there exists a vector field $v\in W^{1,\infty}_0(Q;\RR^2)$ and a phase arrangement
$\chi\in BV(Q; \cK)$ such that, for $u = F + \nabla^{sym} v\in\cD_F^{\cc}(Q)$, the following energy estimate holds:
\begin{equation}\label{eq:BranchingConstrEstCC}
  E^{\cc}_\eps(u,\chi) - E^{\cc}_0(F,\cK) \leq C\Big(\frac{1}{N^4}+\eps N\Big)\quad\forall \eps>0
\end{equation}
for some constant $C=C(a,\ttheta_{\cc})>0$.
\end{breakproposition}
\begin{proof}
The assertion follows by the same arguments as in the proof of \autoref{prop:a3}.
For more details, we refer to the proof in \cite[Lemma 2.3]{cc15}.
\end{proof}
We are now ready to prove \autoref{prop:UpperBoundCCImprovedSimplify}, which then also completes the proof of \autoref{prop:UpperBoundCCImproved} as demonstrated
in \autoref{subsubsec:CCChangeOfVar}.
\begin{proof}[Proof of \autoref{prop:UpperBoundCCImprovedSimplify}]
Given $\eps\in(0,1)$, we apply \autoref{prop:BranchingConstrCC} for $N:=\lceil \eps^{\nicefrac{-1}{5}} \rceil\in\NN_{>1}$.
Since $\eps^{\nicefrac{-1}{5}}\leq N\leq 2 \eps^{\nicefrac{-1}{5}}$,
the desired estimate \eqref{eq:UpperBoundCCImprovedSimplify} follows from \eqref{eq:BranchingConstrEstCC}. 
\end{proof}
We conclude this section by completing the proof of \autoref{thm:incompsyma}.
\begin{proof}[Proof of \autoref{thm:incompsyma}~\textcolor{blue}{(iii)}]
The general upper bound \autoref{prop:UpperBoundCC1} together with the improved upper bound \autoref{prop:UpperBoundCCImproved} prove
the asserted estimate.
\end{proof}

\appendix
\titleformat{\section}
  {\centering\large\scshape}
  {Appendix \Alph{section}.}
  {0.5em}
  {}
\section{Relaxation and potentials}\label{sec:AppendixA}
In this section, we prove \autoref{lem:domaininvarianceInText} which addresses the domain independence of the $\cA$-quasiconvex envelope.
Although this result follows from the work of Rai\c t\u a \cite{raitapot19} and is presumably known to experts,
it does not seem to appear in the literature. For completeness, we therefore provide a detailed proof below.\medbreak

We begin by observing that the $\cA$-free two-well energy does not depend on the geometry of the domain but only on its volume.
\begin{lemma}\label{lem:domaininvariance}
Let $\cA$ be a differential operator as in \eqref{eq:diffopera}.
Let $F\in X$ and $\cK=\{a_0,a_1\}\subset X$.
For a bounded Lipschitz domain $\Omega\subset\RR^d$, let $E_0^\cA(F,\cK;\Omega)$ be given as in \eqref{eq:MinSingPertAfreeEnergy}.
Furthermore, let $Q=(0,1)^d$ be the unit cube.
Then, it holds that
\begin{equation}
  E_0^\cA(F,\cK;\Omega) = \abs{\Omega}E_0^\cA(F,\cK;Q)
\end{equation}
for any bounded domain $\Omega\subset\RR^d$.
\end{lemma}
The proof is an application of Vitali's covering theorem and uses that for any $u\in\cD^\cA_F(\Omega)$, the map $u-F\in\cD^\cA_0(\Omega)$ vanishes outside $\Omega$.
We omit the proof as it is standard in the theory of relaxation; see, for instance, \cite[Lemma 5.2]{rindler}.\smallbreak

It thus suffices to prove \autoref{lem:domaininvarianceInText} for the domain $Q=(0,1)^d$, which we capture in the following proposition.
\begin{proposition}\label{prop:qwdomCube}
Let $\cA$ be a differential operator as in \eqref{eq:diffopera} satisfying the constant rank property; see \autoref{def:constantrank}.
Let $Q=(0,1)^d$ be the unit cube.
Given $F\in X$ and $\cK=\{a_0,a_1\}\subset X$, let $E_0^\cA(F,\cK;Q)$ be given as in \eqref{eq:MinSingPertAfreeEnergy} for the domain $Q$.
Moreover, let $Q^\cA W$ be the $\cA$\nobreakdash-\hspace{0pt}quasiconvex envelope of $W=\dist^2_\cA(\cdot,\cK)$ given as in \autoref{def:aquasi}.
Then, it holds that
\begin{equation}
  E_0^\cA(F,\cK;Q) =Q^\cA W(F).
\end{equation}
\end{proposition}
Roughly speaking, this proposition asserts that the minimum of the $\cA$-free two-well energy in the unit cube
is independent of whether we impose Dirichlet boundary conditions \eqref{eq:MinSingPertAfreeEnergy} or the periodic constraint \eqref{eq:AQuasiconvexification}.\smallbreak
Before proving \autoref{prop:qwdomCube}, we need some auxiliary results.
A key ingredient in its proof is the existence of potentials for constant rank differential operators.
The relation of the gradient and the curl operator serves as an instructive example:
In a simply connected domain $\Omega\subset\RR^d$, we may verify if a matrix field $u\in C^\infty(\Omega;\RR^{d\times d})$ is the gradient 
of a vector field $v\in C^\infty(\Omega;\RR^{d})$ by checking if $\curl u=0$ in $\Omega$. 
Motivated by translating the $\cA$-quasiconvexity condition to the level of potentials, 
it was recently shown that a constant rank differential operator $\cA$ admits an exact potential $\cB$ in frequency space.
\begin{proposition}[{\cite[Theorem 1]{raitapot19}}]\label{prop:potfreq}
  Let $\cA$ be a differential operator as in \eqref{eq:diffopera} and $\AA(\xi)$ its symbol as in \eqref{eq:symbol}.
  Then $\cA$ has constant rank $\raa\in\NN$ (see \autoref{def:constantrank}) if and only if 
  there exists a linear, homogeneous, constant coefficient differential operator
  $\cB:C^\infty(\RR^d; X )\rightarrow C^\infty(\RR^d; X )$ such that
  \begin{equation}\label{eq:kerran}
    \ker \AA(\xi) = \Ran \BB(\xi)\quad \forall\xi\in\RR^d\setminus\{0\},
  \end{equation}
  where $\BB(\xi)$ is the symbol $\cB$.
\end{proposition}
We adopt the terminology from \cite{raitapot19} and refer to $\cA$ as \textit{annihilator} and $\cB$ as \textit{potential}.
\begin{rmk}\label{rem:ellipticorder}
In \autoref{def:constantrank}, we required constant rank differential operators $\cA$ to have rank $\raa\in\NN$. Consequently, we do not 
consider the trivial differential operator $\cA=0$ to be of constant rank.
Under this convention, a careful reading of the proof of \cite[Theorem 1]{raitapot19} reveals that $\cB$ has order zero if and only if $\cA$ is elliptic; see \eqref{eq:symbol}.
To see this, note that in \cite[Equation 9]{raitapot19}, the potential $\cB$ is defined via its symbol:
\begin{equation}\label{eq:bbxi}
  \BB(\xi)=a_r(\xi) \ppa(\xi),\quad \xi\in\RR^d,
\end{equation}
where $a_r(\xi)$ is the $\raa$-th coefficient of the characteristic polynomial of $ \cH(\xi):=\AA(\xi)\AA(\xi)^*$ for $\xi\in\RR^d$.
Since $\cA$ has order $k\in\NN$, $\cH$ is $2k$-homogeneous.
For all $\xi\in\RR^d$, the coefficient $a_r(\xi)$ is a linear combination of the principal minors of
$\cH(\xi)$ of order $\raa$. Hence, $a_r$ is $2\raa k$-homogeneous.
With \eqref{eq:bbxi} and zero-homogeneity of $\ppa$ \eqref{eq:zerohom},
we see that $\cB$ is a differential operator of order $2\raa k$ if $\ppa\ne0$. 
Now, note that $\ppa\equiv 0$ is equivalent to ellipticity of $\cA$, as only then $V_\cA(\xi)=\{0\}$ for all $\xi\in\RR^d\setminus\{0\}$; see \eqref{eq:xicompatiblestates}.
\end{rmk}

\begin{lemma}[{\cite[Corollary 1]{raitapot19}}]\label{lem:AQuasiconvexificationAsPotential}
Let $\cA$ be a differential operator as in \eqref{eq:diffopera} satisfying the constant rank property; see \autoref{def:constantrank}.
Let $\cB$ be the differential operator obtained from \autoref{prop:potfreq}.
Let $Q=(0,1)^d$ be the unit cube.
Given $\cK=\{a_0,a_1\}\subset X$, let $Q^\cA W $ be the
$\cA$-quasiconvex envelope of $W:=\dist_\cA^2(\cdot,\cK)$ as introduced in \autoref{def:aquasi}. Then, it holds that
\begin{equation}\label{eq:AQuasiconvexificationAsPotential}
  Q^\cA W (F) = \inf \left\{\int_{Q} W(F+\cB v): v\in C^\infty_c( Q ; X )   \right\}\quad \forall F\in X .
\end{equation}
\end{lemma}
\begin{rmk}
  Roughly speaking, \autoref{lem:AQuasiconvexificationAsPotential} proves that the two-well energy for the potential $\cB$ with Dirichlet boundary data
  (the right-hand side of \eqref{eq:AQuasiconvexificationAsPotential})
  is equivalent to the periodic $\cA$-free two-well problem \eqref{eq:AQuasiconvexification}.
  The transition from the periodic setting to the Dirichlet problem is achieved using homogenization and a cut-off argument as detailed in \cite{raitapot19}.
\end{rmk}

Given boundary data $F\in\rdd$ on a bounded Lipschitz domain $\Omega\subset\RR^d$, the divergence theorem implies that the average gradient of 
$u\in W^{1,2}(\Omega;\RR^d)$ with $u(x)=Fx$ on $\del\Omega$ (in the sense of traces) is
\begin{equation}
\overline{\nabla u}=\fint_\Omega \nabla u\,dx =F.
\end{equation}
By the following lemma, this carries over to the $\cA$-free setting.
\begin{lemma}\label{lem:raitamean}
Let $\cA$ be a differential operator as in \eqref{eq:diffopera} satisfying the constant rank property; see \autoref{def:constantrank}.
Let $\Omega\subset\RR^d$ be a bounded domain.
Given $F\in X $, let $\cD^\cA_F(\Omega)$ be as in \eqref{eq:admapsa}. Then, it holds that
\begin{equation}\label{eq:meanforadmissible}
\overline{u}=\fint_\Omega u\, dx = F\quad \forall u\in \cD^\cA_F(\Omega).
\end{equation}
\end{lemma}
\begin{proof}
  Let $u\in \cD^\cA_F(\Omega)$.
  If $\cA$ is elliptic, taking the Fourier transform of $\cA u=0$ shows that $\hat{u}(\xi)=0$ for all $\xi\in\RR^d\setminus\{0\}$. In particular, $u\equiv F$ is constant in $\RR^d$ and \eqref{eq:meanforadmissible} is immediate.
  Therefore, we can assume that $\cA$ is not elliptic.
  Let $\cB$ be the potential obtained from \autoref{prop:potfreq}.
  By \autoref{rem:ellipticorder}, the differential operator $\cB$ is of order at least one.
  Now, denote by $\eta\in C^\infty_c(\RR^d;\RR)$ the standard mollifier. Setting $u_0 = (u-F)*\eta$, it holds 
  $u_0\in \SP(\RR^d; X )$ and $\cA u_0=0$. By \cite[Lemma 2]{raitapot19}, it follows that there exists
  $w\in\SP(\RR^d; X )$ with $\cB w= u_0$. Observing that
  \begin{equation}
    0= \int_{\RR^d}\cB w\,dx =\int_{\RR^d}u_0\,dx = \int_{\RR^d}u-F\,dx = \int_\Omega u \, dx - \abs{\Omega} F
  \end{equation}
  completes the proof.
\end{proof}
We are now ready to prove \autoref{prop:qwdomCube}.
\begin{proof}[Proof of \autoref{prop:qwdomCube}]
We apply \autoref{lem:AQuasiconvexificationAsPotential} to obtain a potential $\cB$ for the differential operator $\cA$
such that \eqref{eq:AQuasiconvexificationAsPotential} holds. Taking the Fourier transform, it is immediate that $F+Bv\in\cD^\cA_F(Q)$ for all $v\in C^\infty_c( Q ; X )$.
Recalling \eqref{eq:MinSingPertAfreeEnergy}, the estimate $E_0^\cA(F,\cK;Q) \leq Q^\cA W(F)$ follows from \eqref{eq:AQuasiconvexificationAsPotential}.\smallbreak
For the reverse inequality, let $u\in\cD^\cA_F(Q)$. We restrict $u$ to the unit cube $Q$ and extend $\TT^d$-periodically to obtain $u_{\text{per}}\in L^2(\TT^d; X )$.
A quick computation shows that $\cA u_{\text{per}}=0$ in $\mathscr{D}'(\TT^d)$. Together with \autoref{lem:raitamean} this yields $u_{\text{per}}\in\cD^{\cA,\textup{per}}_F$; see \eqref{eq:admissibleMapsPeriodic}. By \eqref{eq:MinSingPertAfreeEnergy} and \eqref{eq:AQuasiconvexification}, we obtain $E_0^\cA(F,\cK;Q) \geq Q^\cA W(F)$.
\end{proof}
We are now ready to prove the domain independence of the $\cA$-quasiconvex envelope.
\begin{proof}[Proof of \autoref{lem:domaininvarianceInText}]
The statement follows by applying \autoref{lem:domaininvariance} and \autoref{prop:qwdomCube}.
\end{proof}

\addcontentsline{toc}{section}{References}
\emergencystretch=1em
\fancyhead[LE]{\small\nouppercase{\leftmark}}
\hypersetup{urlcolor={black}}    
\printbibliography

\end{document}